%% file: balg_cmt.tex
\documentclass[12pt]{article}

\input{pre}

\begin{document}

\author{\normalsize  Rory B. B. Lucyshyn-Wright\thanks{We acknowledge the support of the Natural Sciences and Engineering Research Council of Canada (NSERC), [funding reference numbers RGPIN-2019-05274, RGPAS-2019-00087, DGECR-2019-00273].  Cette recherche a \'et\'e financ\'ee par le Conseil de recherches en sciences naturelles et en g\'enie du Canada (CRSNG), [num\'eros de r\'ef\'erence RGPIN-2019-05274, RGPAS-2019-00087, DGECR-2019-00273].}\let\thefootnote\relax
\\
\small Brandon University, Brandon, Manitoba, Canada}

\title{\large \textbf{Bifold algebras and commutants\\ for enriched algebraic theories}}

\date{}

\maketitle

\abstract{
Commuting pairs of algebraic structures on a set have been studied by several authors and may be described equivalently as algebras for the tensor product of Lawvere theories, or more basically as certain bifunctors that here we call \textit{bifold algebras}.  The much less studied notion of \textit{commutant} for Lawvere theories was first introduced by Wraith and generalizes the notion of \textit{centralizer clone} in universal algebra.  Working in the general setting of enriched algebraic theories for a system of arities, we study the interaction of the concepts of bifold algebra and commutant.  We show that the notion of commutant arises via a universal construction in a two-sided fibration of bifold algebras over various theories.  On this basis, we study special classes of bifold algebras that are related to commutants, introducing the notions of \textit{commutant bifold algebra} and \textit{balanced bifold algebra}.  We establish several adjunctions and equivalences among these categories of bifold algebras and related categories of algebras over various theories, including commutative, contracommutative, saturated, and balanced algebras.  We also survey and develop examples of commutant bifold algebras, including examples that employ Pontryagin duality and a theorem of Ehrenfeucht and \L o\'s on reflexive abelian groups.  Along the way, we develop a functorial treatment of fundamental aspects of bifold algebras and commutants, including tensor products of theories and the equivalence of bifold algebras and commuting pairs of algebras.  Because we work relative to a (possibly large) system of arities in a closed category $\V$, our main results are applicable to arbitrary $\V$-monads on a finitely complete $\V$, the enriched theories of Borceux and Day, the enriched Lawvere theories of Power relative to a regular cardinal, and other notions of algebraic theory.
}

\section{Introduction} \label{sec:intro}

\begin{para}[\textbf{Bifold algebras}] Given a pair of algebraic structures on a set $C$, so that $C$ underlies both a $\T$-algebra and a $\U$-algebra for a pair of Lawvere theories $\T$ and $\U$, we may ask whether these algebra structures \textit{commute}, which means that every $\T$-operation $C^n \rightarrow C$ is a $\U$-algebra homomorphism (or, equivalently, every $\U$-operation is a $\T$-homomorphism).  A commuting pair of $\T$- and $\U$-algebra structures on the same set can also be expressed equivalently as a $\T$-algebra in the category of $\U$-algebras, or as a $\U$-algebra in the category of $\T$-algebras, or as an algebra for the \textit{tensor product} of the Lawvere theories $\T$ and $\U$ \cite[18.7]{Sch} (cf. \cite{Fr}, \cite[\S 13]{Wra:AlgTh}, \cite[\S 4]{Wra:AlgOverTh}).  Several authors have studied aspects of this topic, each employing selections from among the above equivalent formulations and variations thereupon, and with some authors employing not Lawvere theories or monads but instead the language of Birkhoff's universal algebra.  The latter approach is used in Freyd's classic paper \cite{Fr} and in Bergman's book \cite[10.13]{Berg}.  Schubert \cite[18.7]{Sch} gave one of the earliest treatments of tensor products of Lawvere theories \textit{per se}, and Wraith \cite{Wra:AlgTh,Wra:AlgOverTh} treated commutation and tensor products of infinitary Lawvere-Linton theories.  Replacing the category of sets with an arbitrary symmetric monoidal closed category $\V$, Kock \cite[\S 4]{Kock:DblDln} defined a notion of commutation for cospans of $\V$-monads on $\V$ (i.e. pairs of monad morphisms with the same codomain) and used this to define a notion of commuting pair of algebra structures on an object of $\V$.  Hyland, Plotkin, and Power \cite{HyPlPow} studied the tensor product of \textit{countable Lawvere $\V$-theories} enriched in a locally countably presentable cartesian closed category.  Garner and L\'opez Franco \cite{GaLf} showed that commuting cospans of $\V$-monads on $\V$ are instances of a more abstract concept of commuting cospan of monoids in a duoidal category, and they showed that the tensor product of $\alpha$-accessible $\V$-monads on a locally $\alpha$-presentable $\V$ is an instance of a construction in the duoidal setting.  Commuting cospans of \textit{$\V$-enriched algebraic theories for a system of arities $\J$} \cite{Lu:EnrAlgTh} were studied by the author \cite{Lu:Cmt,Lu:FDistn} and an equivalence with commuting cospans of \textit{$\J$-ary $\V$-monads} was established.

If $\T$ and $\U$ are Lawvere theories then, by a straightforward transposition, a \mbox{$\T$-algebra} in the category of $\U$-algebras can be expressed also as a functor \linebreak $\T \times \U \rightarrow \Set$ that preserves finite powers in each variable separately.  Here we introduce the term \textit{bifold algebra} for such functors\footnote{We note that Wraith used the term \textit{bimodel} for a substantially different concept \cite{Wra:AlgTh,Wra:AlgOverTh}.} and their generalization to the setting of enriched algebraic theories for a system of arities.
\end{para}

\begin{para}[\textbf{Commutants}]
Alongside the concept of bifold algebra, the other concept central to this paper is the much less studied notion of \textit{commutant} for algebraic theories, which generalizes the notion of \textit{centralizer clone} in universal algebra and was introduced by Wraith \cite[\S 10]{Wra:AlgTh} for infinitary Lawvere-Linton theories.  Commutants of morphisms of Lawvere theories and, more generally, enriched algebraic theories and monads were studied by the author in \cite{Lu:Cmt,Lu:CvxAffCmt,Lu:FDistn}, and a notion of centralizer for morphisms of monoids in duoidal categories was defined by Garner and L\'opez Franco \cite{GaLf}.  In this paper, we are concerned not directly with commutants of morphisms of enriched algebraic theories, but rather with a closely related notion of the \textit{commutant of an algebra} $A$ for an enriched algebraic theory $\T$:  In the classical case, given a Lawvere theory $\T$ and a $\T$-algebra $A$ with underlying set $C$, the commutant of $\T$ (with respect to $A$) is a Lawvere theory $\T^\perp_A$ whose $n$-ary operations are all the $\T$-algebra homomorphisms $C^n \rightarrow C$.  In other words, $\T^\perp_A$ consists of all those operations on $C$ that commute with all the $\T$-operations carried by $C$.  It follows that $C$ itself carries the structure of a $\T^\perp_A$-algebra $A^\perp$, which we call the \textit{commutant of $A$}, and the $\T$- and $\T^\perp_A$-algebra structures $A$ and $A^\perp$ on $C$ commute, so that $C$ carries the structure of a bifold algebra.  Commutants for algebras in this sense were approached indirectly in the author's papers \cite{Lu:Cmt,Lu:CvxAffCmt,Lu:FDistn} by encoding algebras as certain morphisms of theories, while in this paper we consider the commutant of an algebra as a central concept, for in fact it not only subsumes the commutant of a morphism of theories as a special case, but it also allows us to view the commutant notion as functorial with respect to a broader class of morphisms of algebras over various theories (\ref{para:str_mor_algs}, \ref{thm:cmtnt_adjn}).  The commutant $\T^\perp_A$ is equally the \textit{algebraic structure} of the functor $A:\T \rightarrow \Set$, in Lawvere's sense \cite{Law:PhD}, but while the notion of algebraic structure is applicable to any tractable set-valued functor, the special case of the commutant $\T^\perp_A$ exhibits a host of special phenomena not present for algebraic structure in general, so that the theory of commutants acquires a distinct character of its own, clear evidence of which is available in \cite[\S 8]{Lu:Cmt} and in Sections \ref{sec:face_adj_cmt_adj} through \ref{sec:bal_bif} of the present paper.
\end{para}

\begin{para}[\textbf{Aims of this paper}]
In this paper, we study the interaction of the concepts of bifold algebra and commutant.  We work in the general setting of enriched algebraic theories for a system of arities \cite{Lu:EnrAlgTh}.  In addition to developing a functorial treatment of fundamental aspects of bifold algebras and commutants for algebras in this setting, we study several novel aspects of the interconnected theory of bifold algebras and commutants.  In particular, one of the central insights of this paper is that the notion of commutant arises functorially via a universal construction in a certain two-sided fibration of bifold algebras over various theories.  Using this result, we define and study special classes of bifold algebras that are related to commutants, including \textit{commutant bifold algebras} and \textit{\mbox{balanced} bifold algebras}.  We establish several adjunctions and equivalences among these categories of bifold algebras and related categories of algebras over various theories, with attention to commutative algebras and a new notion of \textit{contracommutative algebra} that is related by a dual adjunction to commutative algebras.  We also survey and develop several examples of commutant bifold algebras and balanced bifold algebras.
\end{para}

\begin{para}[\textbf{Enriched algebraic theories for a system of arities}]
The framework of $\V$-enriched algebraic theories for a system of arities $\J \hookrightarrow \V$ \cite{Lu:EnrAlgTh} that we employ in this paper includes the following specific notions of algebraic theory as examples: (1) Lawvere theories in the usual sense, with $\V = \Set$ and $\J = \FinCard$ the full subcategory of finite cardinals; (2) infinitary Lawvere-Linton theories, by which we mean the \textit{varietal theories} of Linton \cite{Lin:Eq}, with $\J = \V = \Set$, equivalently, arbitrary monads on $\Set$; (3) the enriched Lawvere theories of Power \cite{Pow:EnrLaw} and their \mbox{$\alpha$-ary} generalization for a regular cardinal $\alpha$, which we may describe equivalently as $\alpha$-accessible $\V$-monads on a locally $\alpha$-presentable closed category $\V$, taking $\J = \V_\alpha$ to be the full subcategory of $\alpha$-presentable objects; (4) the $\V$-theories of Dubuc \cite{Dub:StrSem}, equivalently, arbitrary $\V$-monads on a closed category $\V$; (5) the enriched theories of Borceux and Day \cite{BoDay}, where $\V$ is a \textit{$\pi$-category} and $\J = \{n \cdot I \mid n \in \NN\}$ consists of the finite copowers of the unit object $I$ of $\V$.  Note that (4) generalizes (2), while (3) and (5) provide two distinct generalizations of (1); the theories of Borceux and Day in (5) may be described as single-sorted $\V$-enriched conical finite power theories, while Power's theories in (3) instead involve cotensors (i.e., fully enriched powers) by finitely presentable objects.  Also, (4) illustrates that in general $\J$ need not be small and that $\V$ need not be complete or cocomplete.  Indeed, via (4) we can apply the results in this paper to arbitrary $\V$-monads on a symmetric monoidal closed category $\V$ with finite limits.
\end{para}

\begin{para}[\textbf{Summary of the paper}]
In Section \ref{sec:background} we begin with a review of enriched algebraic theories and commutants for a system of arities $\J \hookrightarrow \V$.  Given a suitable $\V$-category $\C$, in Section \ref{sec:cls_alg} we consider the category $\Algs^\s(\C)$ of algebras over various theories, i.e. pairs $(\T,A)$ consisting of a theory $\T$ and a $\T$-algebra $A$ in $\C$, with \textit{strong morphisms} of algebras, and we discuss the \textit{full algebra} on an object of $\C$ and its universal property; these concepts facilitate working with multiple algebra structures on an object while allowing both the theory and the carrier to vary, while one may also fix a specific object $C$ of $\C$ and consider the category $\Algs(C)$ of algebras with carrier $C$.  We begin to apply these concepts in Section \ref{sec:comm_pairs_algs} to arrive at a convenient formalism for commuting pairs of algebra structures on an object and their relationship to the commutant of an algebra $(\T,A)$, which is another algebra
$$(\T,A)^\perp = (\T^\perp_A,A^\perp)$$
in which the theory $\T^\perp_A$ is the commutant of $\T$ with respect to $A$, while $A^\perp:\T^\perp_A \rightarrow \C$ is the associated $\T^\perp_A$-algebra on the same carrier as $A$.

In Section \ref{sec:bifold_algs} we introduce bifold algebras in $\C$, defined as $\V$-functors
$$D:\T \otimes \U \rightarrow \C$$
that preserve $\J$-cotensors in each variable separately, for specified theories $\T$ and $\U$, where $\T \otimes \U$ is the usual monoidal product of $\V$-categories.  Every bifold algebra $D$ determines a commuting pair of algebras $D_\ell$ and $D_r$, called the \textit{left and right faces} of $D$.  In Section \ref{thm:tens_prods} we show that bifold algebras for a fixed pair of theories $\T$ and $\U$ can be described equivalently as $\otimesJ{\T}{\U}$-algebras for a theory $\otimesJ{\T}{\U}$ called the tensor product, provided that $\J$ is small and $\V$ is locally bounded, which we do \textit{not} assume in the rest of the paper.  In Section \ref{sec:two_sided_fibr} we employ a version of the Grothendieck construction for two-sided fibrations \cite{Str:FibrYon2Cats} to define a category of bifold algebras over various theories, $\BAlg^\sx(\C)$, in which the morphisms are rather peculiar: Called \textit{strong cross-morphisms of bifold algebras}, these morphisms include a pair of theory morphisms that go in opposite directions.

This peculiar category of bifold algebras over various theories, $\BAlg^\sx(\C)$, holds the key to a fully functorial elucidation of commutants and their interaction with the notion of bifold algebra.  Indeed, this starts to become clear with the central results of Section \ref{sec:face_adj_cmt_adj}, where we begin by showing that the functors $L$ and $R$ that furnish the left and right faces of a bifold algebra have fully faithful left and right adjoints, respectively, as in the following diagram:
\begin{equation}\label{eq:face_ajdns}
\xymatrix{
& \BAlg^\sx(\C) \ar[dl]^L \ar[dr]_R^[@!-30]{\top} &\\
\Algs^\s(\C) \ar@/^3ex/@{ >-->}[ur]_[@!30]{\bot}^\RCom & & \Algs^\s(\C)^\op \ar@/_3ex/@{  >-->}[ul]_\LCom
}
\end{equation}
The left adjoint functor $\RCom$ sends each algebra $(\T,A)$ to a bifold algebra  $\RCom(\T,A)$ whose left face is $(\T,A)$ and whose right face is the commutant $(\T,A)^\perp$ of $(\T,A)$.  In particular, this shows that the notion of commutant arises by way of a universal construction in $\BAlg^\sx(\C)$, so that as a consequence it follows immediately that the assignment $(\T,A) \mapsto (\T,A)^\perp$ is functorial with respect to strong morphisms of algebras, because by composing $\RCom$ and $R$ we obtain a functor $\Com = (-)^{\perp}$ that sends each algebra $(\T,A)$ to its commutant $(\T,A)^\perp$ and is left adjoint to its opposite, as in the diagram
\begin{equation}\label{eq:cmt_adjn}
\xymatrix{
\Algs^\s(\C) \ar@{}[rr]|\top \ar@/_2ex/[rr]_{\Com} & & \Algs^\s(\C)^\op. \ar@/_2ex/[ll]_{\Com^\op}
}
\end{equation}

In Section \ref{sec:lrcmt_bif} we consider the coreflective subcategory $$\RComBAlg^\sx(\C) \;\;\hookrightarrow\;\; \BAlg^\sx(\C)$$
determined by the coreflective embedding $\RCom$ in \eqref{eq:face_ajdns}, and similarly the reflective subcategory $\LComBAlg^\sx(\C)$ determined by $\LCom$.  We say that a bifold algebra $D$ with carrier $C$ is a \textit{right-commutant bifold algebra} if $D$ lies in $\RComBAlg^\sx(\C)$, equivalently, if its right face $D_r$ is isomorphic in $\Algs(C)$ to the commutant $D_\ell^\perp$ of its left face $D_\ell$; analogously we define \textit{left-commutant bifold algebras}.  Consequently we obtain equivalences of categories
$$\RComBAlg^\sx(\C) \simeq \Algs^\s(\C)\;\;\;\;\text{and}\;\;\;\;\LComBAlg^\sx(\C) \simeq \Algs^\s(\C)^\op$$
that provide a sense in which a right-commutant or left-commutant bifold algebra is equivalently given by an algebra.  We characterize right-commutant (resp. left-commutant) bifold algebras $D:\T \otimes \U \rightarrow \C$ as those whose transpose $\U \rightarrow [\T,\C]$ (resp. $\T \rightarrow [\U,\C]$) is fully faithful.

In Section \ref{sec:cmt_bifold_algs}, we consider \textit{commutant bifold algebras}, which are those bifold algebras that are both left-commutant and right-commutant, and we establish an equivalence between commutant bifold algebras and \textit{saturated algebras}, which are those algebras $(\T,A)$ such that $(\T,A)^{\perp\perp} \cong (\T,A)$.  The commutant adjunction \eqref{eq:cmt_adjn} is idempotent, and its fixed points are the saturated algebras, which therefore form a reflective subcategory $\SatAlgs^\s(\C)$ of $\Algs^\s(\C)$.  We establish equivalences
$$\SatAlgs^\s(\C) \simeq \ComBAlg^\sx(\C) \;\simeq\; \SatAlgs^\s(\C)^\op\;$$
under which a commutant bifold algebra corresponds to its left and right faces, respectively, which are commutants of one another.  We show also that $\ComBAlg^\sx(\C)$ is reflective in $\RComBAlg^\sx(\C)$ and coreflective in $\LComBAlg^\sx(\C)$.

In Section \ref{sec:cmt_algs}, and with the benefit of the above functorial methods for bifold algebras and commutants, we establish the categorical equivalence of bifold algebras and commuting pairs of algebras.  In particular, we establish an equivalence of \mbox{$\V$-categories}
$$\Alg{(\T,\U)}(\C) \;\simeq\; \CAlgPair{\T}{\U}(\C),$$
natural in $\T$ and $\U$, between the $\V$-category of bifold $(\T,\U)$-algebras and a $\V$-category $\CAlgPair{\T}{\U}(\C)$ of commuting $\T$-$\U$-algebra pairs.  Also, we define a category $\CPair^\sx(\C)$ of commuting algebra pairs over various theories, with strong cross-morphisms, and we show that
$$\BAlg^\sx(\C) \;\simeq\; \CPair^\sx(\C)\;.$$
On this basis, we define the notions of right-commutant, left-commutant, and commutant algebra pair, for which we obtain corollaries to several of the above results on bifold algebras.  In Section \ref{sec:pres_refl_calgpair} we establish several fundamental results on the preservation and reflection of commuting algebra pairs, showing in particular that the hom $\V$-functors $\C(G,-):\C \rightarrow \V$ for the objects $G$ of any enriched generating class $\G$ jointly reflect commutation of algebra pairs.

In Section \ref{sec:comm_contracomm_bal} we discuss commutative algebras, which are those algebras that commute with themselves, and we study two further special classes of algebras: \textit{Contracommutative algebras}, which are those whose commutant is commutative, and \textit{balanced algebras}, which are those algebras $A$ such that $A \cong A^\perp$ in $\Algs(C)$ where $C = |A|$.  We show that an algebra $A$ is commutative if and only if $A^\perp$ is contracommutative, while  $A$ is balanced if and only if $A$ is commutative, contracommutative, and saturated.  We show that the commutant adjunction \eqref{eq:cmt_adjn} restricts to (1) a dual adjunction between commutative algebras and contracommutative algebras, which further restricts to (2) a dual equivalence between commutative saturated algebras and contracommutative saturated algebras, which in turn restricts to (3) an equivalence between the category $\BalAlgs^\s(\C)$ of balanced algebras and its opposite.  We show that $\BalAlgs^\s(\C)$ is a groupoid whose canonical anti-involution is isomorphic to the latter equivalence.

In Section \ref{sec:one_comm_face}, we return to our studies of left- and right-commutant bifold algebras, with special attention to the case in which the left face is commutative, or equivalently the right face is contracommutative.  We obtain an equivalence between the category of commutative algebras (resp. commutative saturated algebras) and the category of right-commutant (resp. commutant) bifold algebras with commutative left face.  We show that any bifold $(\T,\U)$-algebra of this special kind induces a central morphism from $\T$ to $\U$, and consequently there is a $\V$-functor $\Alg{\U}(\C) \rightarrow \Alg{(\T,\U)}(\C)$ that witnesses that every $\U$-algebra carries the structure of a bifold $(\T,\U)$-algebra.  As discussed in \ref{para:f_an}, these results are applicable in particular to the \textit{functional-analytic contexts} of \cite{Lu:FDistn},  which form the basis for a study of distribution monads via commutants for enriched theories and monads.

In Section \ref{sec:bal_bif}, we define the concept of a \textit{balanced bifold algebra}, which is a commutant bifold algebra $D$ whose left and right faces are isomorphic in $\Algs(C)$, where $C$ is the carrier of $D$.  We show that a commutant bifold algebra is balanced if and only if its left and right faces are both commutative.  We establish an equivalence $$\BalAlgs^\s(\C) \simeq \BalBAlg^\sx(\C)$$
between the category (indeed, groupoid) of balanced algebras and the category of balanced bifold algebras, which therefore is a groupoid.

In Section \ref{sec:exa_bif} we survey and develop various examples of commutant bifold algebras and commutants, involving (1) bimodules over pairs of rings; (2) internal rings, rigs (also known as semirings), and preordered rings in cartesian closed categories, and internal modules and affine spaces for internal rigs, including convex spaces for preordered rings; (3) semilattices, with and without top and/or bottom element; (4) topological groups, convergence groups, the circle group, and Pontryagin duality; (5) complete lattices with supremum-preserving maps; (6) the abelian group $\ZZ$ and a theorem of Ehrenfeucht and \L o\'s \cite{EhrLos} on the question of reflexivity of free abelian groups.  With regard to (4), we show that the circle group $\TT$ carries the structure of a commutant bifold algebra in three settings as a consequence of Pontryagin duality: (a) $\TT$ is a balanced bifold algebra in the category of locally compact Hausdorff spaces, with respect to the Lawvere theory of abelian groups, (b) $\TT$ is a balanced bifold algebra with respect to the theory of internal abelian groups in the category of convergence spaces, for the enriched Borceux-Day system of arities, and (c) $\TT$ is a commutant bifold algebra in $\Set$ for the infinitary Lawvere-Linton theories of abelian groups and of compact Hausdorff abelian groups.  With regard to (6) we show that the statement that $\ZZ$ is a balanced bifold algebra with respect to the Lawvere-Linton theory of abelian groups (so for the system of arities consisting of all sets) is equivalent to the non-existence of measurable cardinals.
\end{para}

\begin{para}[\textbf{Future work}]
In a subsequent paper \cite{Lu:AlgDual}, we shall employ the results of this paper in order to establish biequivalences between certain categories of bifold algebras and certain 2-categories of \textit{$\J$-algebraic dual adjunctions}, which are dual adjunctions between pairs of $\J$-algebraic $\V$-categories over $\V$ (in the sense of \cite[12.1]{Lu:EnrAlgTh}) and form the basis for a study of enriched algebraic dualization processes in analysis, order theory, and topology, as initiated in the talk \cite{Lu:CT2017}.
\end{para}

\section{Background}\label{sec:background}

\subsection{Enriched categories}\label{sec:enr_cats}

Throughout, we write $\V$ to denote a given symmetric monoidal closed category, and we write $\V_0$ to denote the ordinary category underlying $\V$.  We employ the theory of categories enriched in $\V$, as expounded in \cite{Ke:Ba,Dub}.  We use the notation and terminology of \cite{Ke:Ba}, except that we write cotensors as $[X,C]$ rather than $X \pitchfork C$ \cite[\S 3.7]{Ke:Ba}.  In referring to $\V$-categories, $\V$-functors, etcetera, we always include the prefix ``$\V$-''.

We employ the universe enlargement methodology that is discussed in \cite[\S 3.11]{Ke:Ba}.  Hence $\Set$ denotes the category of sets in some universe $\mathfrak{U}$ of \textit{small sets} for which $\V_0$ is locally small, while $\mathfrak{U}'$ is a universe for which both $\Set$ and $\V_0$ are $\mathfrak{U}'$-small sets, and for which all given $\V$-categories and (possibly large) sets under discussion are $\mathfrak{U}'$-small, and $\V'$ is a $\mathfrak{U}'$-complete and $\mathfrak{U}'$-cocomplete enlargement of $\V$ such that the full inclusion $\V \hookrightarrow \V'$ preserves all limits that exist in $\V$.

\subsection{Enriched algebraic theories}

\begin{parasub}\label{para:sys_arities}
A \textbf{system of arities} in $\V$ is a fully faithful symmetric strong monoidal $\V$-functor $j:\J \rightarrow \V$ \cite[3.1]{Lu:EnrAlgTh}.   As discussed in \cite[3.9]{Lu:EnrAlgTh}, we may assume for most purposes that $j$ is the inclusion of a full sub-$\V$-category $\J \hookrightarrow \V$ that is closed under the monoidal product and contains the unit object $I$ of $\V$.  A system of arities $j:\J \hookrightarrow \V$ is \textbf{eleutheric} \cite[\S 7]{Lu:EnrAlgTh} if for every $\V$-functor $F:\J \rightarrow \V$, the left Kan extension of $F$ along $j$ exists and is preserved by the $\V$-functor $\V(J,-):\V \rightarrow \V$ for each object $J$ of $\J$.  We recall several examples of systems of arities in \ref{para:exa_sys_ar} below.
\end{parasub}

\begin{parasub}\label{para:jcots}
Let $\J \hookrightarrow \V$ be a system of arities.  By a \textbf{$\J$-cotensor} in a $\V$-category $\C$ we mean a cotensor $[J,C]$ for some $J \in \ob\J$ and $C \in \ob\C$.  When discussing $\V$-categories $\C$ with $\J$-cotensors, we assume that $\C$ is equipped with a designated cotensor $[J,C]$ for all $J \in \ob\J$ and $C \in \ob\C$.  In particular, the object $[J,C]$ is equipped with a specified \textit{counit morphism} $\gamma_J:J \rightarrow \C([J,C],C)$ \cite[\S 3.7]{Ke:Ba}.  We assume that $[I,C] = C$ and that $\gamma_I$ is the identity morphism on $C$ in $\C_0$.

For each pair of $\V$-categories $\E$ and $\F$ with $\J$-cotensors, we know that $[\E,\F]$ exists as a $\V'$-category \pref{sec:enr_cats}, and we denote by
\begin{equation}\label{eq:jcot_pres_func_cat}[\E,\F]_\J\end{equation}
the full sub-$\V'$-category of $[\E,\F]$ consisting of the $\J$-cotensor preserving $\V$-functors.
\end{parasub}

\begin{parasub}\label{para:jth}
Given a system of arities $j:\J \hookrightarrow \V$, a \textbf{$\J$-theory} enriched in $\V$ (or a \mbox{\textbf{$j$-theory}}) is a \mbox{$\V$-category} $\T$ equipped with an identity-on-objects $\V$-functor $\tau:\J^\op \rightarrow \T$ that preserves $\J$-cotensors \cite[4.1]{Lu:EnrAlgTh}.  Equivalently, $\T$ is a $\V$-category whose objects are those of $\J$, in which each object $J$ is equipped with a morphism $\gamma_J:J \rightarrow \T(J,I)$ that exhibits $J$ as a cotensor $[J,I]$ in $\T$ \pref{para:jcots}, with the requirement that $\gamma_I$ is the identity morphism on $I$ in $\T_0$ \cite[5.8]{Lu:EnrAlgTh}.  It follows that $\T$ has $\J$-cotensors \cite[4.3, 4.5]{Lu:EnrAlgTh}, and that $\tau$ is the $\V$-functor $[-,I]:\J^\op \rightarrow \T$ that supplies the designated $\J$-cotensors $[J,I] = J$ of $I$ \cite[5.8]{Lu:EnrAlgTh}.

Given $\J$-theories $\T_1$ and $\T_2$, a \textbf{morphism of $\J$-theories} $M:\T_1 \rightarrow \T_2$ is a $\V$-functor that commutes with the associated $\V$-functors $\tau_i:\J^\op \rightarrow \T_i$ $(i = 1,2)$, equivalently, that strictly preserves the designated $\J$-cotensors of $I$ \cite[5.16]{Lu:EnrAlgTh}.  With these morphisms, $\J$-theories are the objects of a category $\ThJ$, and $\J^\op$ is an initial object of this category.  If $\J$ is eleutheric \pref{para:sys_arities}, then there is an equivalence between $\ThJ$ and the category of \textit{$\J$-ary $\V$-monads} on $\V$ \cite[11.8]{Lu:EnrAlgTh}.  A morphism of $\J$-theories $M$ is called a \textbf{subtheory embedding} if the $\V$-functor $M$ is faithful.
\end{parasub}

\begin{parasub}\label{para:talg}
Given a $\J$-theory $\T$ and a $\V$-category $\C$ with $\J$-cotensors, a \textbf{$\T$-algebra} in $\C$ is a $\V$-functor $A:\T \rightarrow \C$ that preserves $\J$-cotensors.  An \textbf{algebra} in $\C$ is a pair $(\T,A)$ consisting of a $\J$-theory $\T$ and a $\T$-algebra in $\C$.  We also write the algebra $(\T,A)$ simply as $A$.  The object $|A| = A(I)$ of $\C$ is called the \textbf{carrier} of $A$, where $I$ is the unit object of $\V$.  If $|A| = C$ then we call $A$ a \textbf{($\T$-)algebra on $C$}. Using the notation \eqref{eq:jcot_pres_func_cat}, $\T$-algebras in $\C$ are the objects of a $\V'$-category
$$\Alg{\T}(\C) = [\T,\C]_\J\;.$$
We are primarily concerned with cases where the ends needed in forming $[\T,\C]_\J$ exist in $\V$, in which case we say that $\Alg{\T}(\C)$ \textbf{exists as a $\V$-category}.  A $\V$-category $\C$ is \textbf{$\J$-admissible} if $\C$ has $\J$-cotensors and $\Alg{\T}(\C)$ exists for every $\J$-theory $\T$.  For example, if $\V$ has equalizers and $\J$ is eleutheric, then $\V$ itself is $\J$-admissible, by \cite[8.8, 8.9]{Lu:EnrAlgTh}.  Every $\V$-category $\C$ with $\J$-cotensors is $\J$-admissible if $\V$ has equalizers and has intersections of $(\ob\J)$-indexed families of strong monomorphisms (with common codomain) \cite[4.11]{Lu:Cmt}.  Obviously this applies in particular when $\J$ is small and $\V_0$ is complete, but we are interested also in examples beyond this case, e.g. when $\J = \V$ (See \ref{para:exa_sys_ar} below).
\end{parasub}

\begin{parasub}\label{para:carrier_func}
If $\C$ is a $\J$-admissible $\V$-category, then there is a faithful $\V$-functor
$$G^\T:\Alg{\T}(\C) \rightarrow \C$$
that is given by evaluation at the unit object $I$ (\cite[5.4]{Lu:EnrAlgTh}, \cite[4.8]{Lu:Cmt}).  Hence $G^\T$ sends each $\T$-algebra $A$ in $\C$ to its carrier $|A|$.  Accordingly, if $f:A \rightarrow B$ is a \mbox{\textbf{$\T$-homomorphism}}, i.e., a morphism in $\Alg{\T}(\C)$, then we write $|f| = f_I:|A| \rightarrow |B|$.  The $\V$-category $\Alg{\T}(\C)$ has $\J$-cotensors, formed pointwise, so we may choose designated $\J$-cotensors in $\Alg{\T}(\C)$ in such a way that they are sent by $G^\T$ to the designated $\J$-cotensors in $\C$ \pref{para:jcots}.  Every morphism of $\J$-theories $M:\T \rightarrow \U$ induces a $\V$-functor $M^*:\Alg{\U}(\C) \rightarrow \Alg{\T}(\C)$ that commutes with $G^\U$, $G^\T$ and preserves $\J$-cotensors.
\end{parasub}

\begin{parasub}[\textbf{Examples of systems of arities and their theories}]\label{para:exa_sys_ar}
\emptybox
\medskip

\noindent(1) By \cite[7.5(2)]{Lu:EnrAlgTh}, there is an eleutheric system of arities $\J = \FinCard \hookrightarrow \V = \Set$, for which $\J$-theories are Lawvere theories in the usual sense \cite{Law:PhD}.

\medskip

\noindent(2) If $\V$ is locally $\alpha$-presentable as a closed category, for a regular cardinal $\alpha$, then there is an eleutheric system of arities $\J = \V_\alpha \hookrightarrow \V$ consisting of the $\alpha$-presentable objects of $\V$ \cite[7.5(1)]{Lu:EnrAlgTh}, for which $\J$-theories are (the $\alpha$-ary generalization of) the \textit{enriched Lawvere theories} of Power \cite{Pow:EnrLaw}, which are equivalently described as $\V$-monads on $\V$ that preserve $\alpha$-filtered colimits.

\medskip

\noindent(3) There is an eleutheric system of arities $\J = \V$ \cite[7.5(3)]{Lu:EnrAlgTh}.  The category of $\V$-theories $\Th_\V$ for this system of arities $\V$ was considered by Dubuc in \cite{Dub:StrSem} and is equivalent to the category of $\V$-monads on $\V$ \cite[11.10]{Lu:EnrAlgTh}.  In the special case where $\J = \V = \Set$, we recover the notion of \textit{varietal theory} in the sense of Linton \cite{Lin:Eq}.  Here we call Linton's varietal theories \textit{Lawvere-Linton theories}.

\medskip

\noindent(4) There is an eleutheric system of arities $\J = \{I\} \hookrightarrow \V$ consisting of just the unit object $I$ of $\V$ \cite[7.5(4)]{Lu:EnrAlgTh}.  For this system of arities, an $\{I\}$-theory $\T$ is precisely a monoid $R$ in $\V$ (e.g. a ring if $\V = \Ab$), and a $\T$-algebra may be described equivalently as a left $R$-module in $\V$ \cite[5.3(3)]{Lu:EnrAlgTh}.  Moreover, $\Alg{\T}(\V)_0$ is the category of left $R$-modules in $\V$, which underlies a $\V$-category $\Mod{R} = \Alg{\T}(\V)$ provided that $\V$ has equalizers.

\medskip

\noindent(5) If $\V$ is cartesian closed and countably cocomplete, then there is an eleutheric system of arities $j:\DF(\V) \rightarrow \V$, where the objects of $\DF(\V)$ are the finite cardinals $n$, with $\DF(\V)(n,m) = \V(n \cdot 1, m \cdot 1)$ where $n\cdot 1$ denotes the $n$-th copower of the terminal object $1$ of $\V$ \cite[7.5(5)]{Lu:EnrAlgTh}.  $j$ is the evident $\V$-functor given on objects by $n \mapsto n \cdot 1$.  A $\DF(\V)$-theory $\T$ is a \textbf{discretely finitary-algebraic theory} enriched in $\V$ (\cite[4.2(6)]{Lu:EnrAlgTh}, \cite[\S 3.2]{Lu:FDistn}), i.e., a $\V$-category $\T$ whose objects are the finite cardinals $n$, each of which is equipped with a family of morphisms $\pi_i^n:n \rightarrow 1$ $(i = 1,...,n)$ that present $n$ as an $n$-th power of $1$ in $\T$, with $\pi_1^1$ the identity morphism on $1$.  We can also replace the assumption that $\V$ is cartesian closed with the assumption that $\V$ is a \textit{$\pi$-category} in the sense of \cite{BoDay}, in which case $\DF(\V)$-theories are essentially the enriched algebraic theories defined by Borceux and Day in that paper, as discussed in \cite[4.2(6)]{Lu:EnrAlgTh}.
\end{parasub}

\begin{parasub}\label{para:conv_cah}
Two examples of cartesian closed categories $\V$ to which refer are (i) the \textbf{Cahiers topos} \cite{Dub:ModSDG}, which is a well-adapted model of synthetic differential geometry, and (ii) the category of \textbf{convergence spaces} \cite{BeBu}, $\Conv$, into which the category of topological spaces $\Top$ embeds as a full, reflective subcategory; see, e.g., \cite[2.3]{Lu:FDistn} for a brief overview of convergence spaces and their cartesian closed structure.  Viewing topological spaces as certain convergence spaces, and sets as discrete topological spaces, we have full subcategory inclusions $\Set \hookrightarrow \Top \hookrightarrow \Conv$, the first of which preserves finite limits, and the second of which preserves all limits.  Consequently, Lawvere theories in the usual sense are precisely those $\DF(\Conv)$-theories that are \textit{discrete}, i.e. have discrete hom-objects, noting that a $\DF(\Conv)$-theory is discrete iff it is the free $\DF(\Conv)$-theory on a Lawvere theory \cite[6.3]{Lu:FDistn}.
\end{parasub}

\begin{parasub}[\textbf{{Given data and assumptions}}]\label{para:given_data}
\textit{Throughout the remainder of the paper, we let $j:\J \hookrightarrow \V$ be an arbitrary system of arities in a symmetric monoidal closed category $\V$, and we let $\C$ be a $\J$-admissible $\V$-category \pref{para:talg}.  We also assume that $\V_0$ has pullbacks, noting that we use this assumption only from Section \ref{sec:cmt_algs} onward.}
\end{parasub}

\begin{parasub}\label{para:jop_algs}
Recalling that $\J^\op$ is the initial $\J$-theory \pref{para:jth}, the $\V$-category $\Alg{\J^\op}(\C)$ is equivalent to $\C$ \cite[4.2]{Lu:Cmt}.  Indeed, $G^{\J^\op}:\Alg{\J^\op}(\C) \rightarrow \C$ is an equivalence, with a pseudo-inverse that sends each object $C$ of $\C$ to the $\V$-functor $[-,C]:\J^\op \rightarrow \C$ that supplies the designated $\J$-cotensors of $C$ \pref{para:jcots}.  In particular, it follows that there is a fully faithful $\V'$-functor $\Xi:\C \rightarrow [\J^\op,\C]$ given by $\Xi C = [-,C]$ $(C \in \C)$.
\end{parasub}

\begin{parasub}\label{para:normal_talgs}
If $A:\T \rightarrow \C$ is a $\T$-algebra, then since $A$ preserves $\J$-cotensors and $[J,I] = J$ in $\T$ \pref{para:jth}, there are isomorphisms $AJ \cong [J,\ca{A}]$ $(J \in \ob\J = \ob\T)$.  A $\V$-functor $A:\T \rightarrow \C$ is a \textbf{normal $\T$-algebra} \cite[5.10]{Lu:EnrAlgTh} if $A$ sends the designated $\J$-cotensors $[J,I] = J$ of $I$ in $\T$ to the designated $\J$-cotensors $[J,\ca{A}]$ of $\ca{A} = AI$ in $\C$ \pref{para:jcots}.  Every normal $\T$-algebra is a $\T$-algebra \cite[5.9, 5.10]{Lu:EnrAlgTh}, so normal $\T$-algebras in $\C$ are the objects of a full sub-$\V$-category $\Alg{\T}^\nml(\C)$ of $\Alg{\T}(\C)$.  The inclusion $\Alg{\T}^\nml(\C) \hookrightarrow \Alg{\T}(\C)$ is an equivalence $$\Alg{\T}^\nml(\C) \simeq \Alg{\T}(\C)\;,$$
by \cite[5.14]{Lu:EnrAlgTh}.  The latter equivalence associates to each $\T$-algebra $A$ in $\C$ a normal $\T$-algebra $A^\nml$ on $|A|$ that is called the \textbf{normalization} of $A$ and is equipped with an isomorphism $\varphi_A:A \xrightarrow{\sim} A^\nml$ with $|\varphi_A| = 1_{|A|}$.

By \cite[5.12]{Lu:EnrAlgTh}, a $\V$-functor $A:\T \rightarrow \C$ is a normal $\T$-algebra if and only if
$$A\tau = [-,AI]:\J^\op \rightarrow \C$$ with the notation of \ref{para:jth} and \ref{para:jop_algs}.   It follows that there is a pullback square
\begin{equation}\label{eq:pb_normal_talgs}
\xymatrix{
\Alg{\T}^\nml(\C) \ar@{^(->}[r] \ar[d]_{G^\T_\nml} & [\T,\C] \ar[d]^{\tau^* \:=\: [\tau,\C]}\\
\C \ar[r]^\Xi & [\J^\op,\C]
}
\end{equation}
in $\eCAT{\V'}$, where $\Xi$ is the fully faithful $\V'$-functor defined in \ref{para:jop_algs} and we write $G^\T_\nml$ for the restriction of $G^\T$ to normal $\T$-algebras.  The $\V'$-functor $\Xi$ sends every limit that exists in $\C$ to a pointwise limit in $[\J^\op,\C]$, and since $\tau$ is identity-on-objects, the $\V'$-functor $[\tau,\C]$ creates pointwise limits\footnote{I.e., given $\V'$-functors $W:\K \rightarrow \V'$ and $D:\K \rightarrow [\T,\C]$, any pointwise limit cylinder $(\{W,\tau^*D\},\lambda)$ that exists in $[\J^\op,\C]$ lifts uniquely to a cylinder in $[\T,\C]$, and the latter cylinder is a pointwise limit cylinder.}.  Since the square \eqref{eq:pb_normal_talgs} is a pullback, it follows that $G^\T_\nml$ creates $\V'$-enriched limits\footnote{I.e., given $\V'$-functors $W:\K \rightarrow \V'$ and $D:\K \rightarrow \Alg{\T}^\nml(\C)$, any limit cylinder $(\{W,G^\T_! D\},\lambda)$ that exists in $\C$ lifts uniquely to a cylinder in $\Alg{\T}^!(\C)$, and the latter cylinder is a limit cylinder.}.
\end{parasub}

\begin{parasub}\label{para:jth_enr_in_vprime}
With the notation of \ref{sec:enr_cats}, the composite $\J \hookrightarrow \V \hookrightarrow \V'$ is a system of arities $j':\J \hookrightarrow \V'$.  Hence, we may consider $j'$-theories, $\T$, which we call \textbf{$\J$-theories enriched in $\V'$}.  Conveniently, every $\V'$-category with $\J$-cotensors is a $\J$-admissible $\V'$-category, in view of \ref{sec:enr_cats} and \ref{para:talg}.  Thus all the definitions and results that we develop for a given $\J$-admissible $\V$-category $\C$ may be applied to \textit{any} $\V$-category or $\V'$-category $\X$ with $\J$-cotensors by employing the system of arities $j':\J \hookrightarrow \V'$, and we apply this technique tacitly in several instances throughout the paper.
\end{parasub}

\begin{parasub}\label{para:mor_jth}
Given $\J$-theories $\T$ and $\U$, a morphism of $\J$-theories $M:\T \rightarrow \U$ is equivalently defined as a normal $\T$-algebra on $I$ in $\U$ \cite[5.16]{Lu:EnrAlgTh}.  A morphism of $\J$-theories $M$ is uniquely determined by its components $M_{JI}:\T(J,I) \rightarrow \U(J,I)$ with $J \in \ob\J$ \cite[3.12]{Lu:Cmt}, and $M$ is an isomorphism if and only if $M_{JI}$ is an isomorphism for each $J \in \ob\J$ \cite[3.1.3]{Lu:FDistn}.
\end{parasub}

\subsection{Commutants and the full theory of an object}

\begin{parasub}\label{para:full_jth}
Given an object $C$ of $\C$, there is a $\J$-theory $\C_C$, called the \textbf{full $\J$-theory of $C$} in $\C$ \cite[3.16]{Lu:Cmt}, with $\C_C(J,K) = \C([J,C],[K,C])$ for all $J,K \in \ob\C_C = \ob\J$, and with composition and identities as in $\C$.  There is clearly a fully faithful $\V$-functor $\Gamma_C:\C_C \rightarrow \C$ given on objects by $J \mapsto [J,C]$, and $\Gamma_C$ is a normal $\C_C$-algebra on $C$.  Given any normal $\T$-algebra $A$ on $C$ in $\C$, there is a unique morphism of $\J$-theories $[A]:\T \rightarrow \C_C$ such that $A = \Gamma_C [A]$ \cite[3.16]{Lu:Cmt}.  The assignment $M \mapsto \Gamma_C M$ defines a bijective correspondence between morphisms of $\J$-theories $M:\T \rightarrow \C_C$ and normal $\T$-algebras on $C$ in $\C$ \cite[3.16]{Lu:Cmt}.
\end{parasub}

\begin{parasub}[\textbf{Commutants}]\label{para:cmt}
Let $A:\T \rightarrow \C$ be a $\T$-algebra.  The \textbf{commutant} of $\T$ with respect to $A$ \cite[7.1, 7.10]{Lu:Cmt} is the full $\J$-theory of $A$ in the $\V$-category $\Alg{\T}(\C)$ and is denoted by $\T^\perp_A = \Alg{\T}(\C)_A$.  Hence $\T^\perp_A$ is a $\J$-theory with hom-objects
$$\T^\perp_A(J,K) = \Alg{\T}(\C)([J,A],[K,A])\;\;\;\;\;\;(J,K \in \ob\J),$$
and with composition and identities as in $\Alg{\T}(\C)$, where $[J,A]$ denotes the pointwise cotensor of $A$ by $J$ \pref{para:carrier_func}.

By \ref{para:full_jth}, there is a fully faithful normal $\T^\perp_A$-algebra
\begin{equation}\label{eq:cl_alg_cmt}\Gamma_A\;:\;\T^\perp_A = \Alg{\T}(\C)_A \longrightarrow \Alg{\T}(\C)\end{equation}
with carrier $A$, and we write $A^\perp$ to denote the composite
\begin{equation}\label{eq:a_lowerperp}A^\perp = \bigl(\T^\perp_A \xrightarrow{\Gamma_A} \Alg{\T}(\C) \xrightarrow{G^\T} \C\bigr)\;.\end{equation}
Since $G^\T$ strictly preserves the designated $\J$-cotensors \pref{para:carrier_func}, we deduce that $A^\perp$ is a normal $\T^\perp_A$-algebra with $|A^\perp| = |A|$, and since $\Gamma_A$ and $G^\T$ are faithful we deduce that $A^\perp$ is faithful.  We call the algebra
\begin{equation}\label{eq:cmtnt_of_alg}(\T,A)^\perp = (\T^\perp_A,A^\perp)\end{equation}
the \textbf{commutant} of $(\T,A)$ and denote it also by $A^\perp$, in keeping with our convention in \ref{para:talg}.  By \ref{para:full_jth}, there is a unique morphism of $\J$-theories $[A^\perp]:\T^\perp_A \rightarrow \C_C$
such that $\Gamma_C[A^\perp] = A^\perp$, where $C = |A|$, and since $A^\perp$ is faithful we deduce that $[A^\perp]$ is a subtheory embedding
\begin{equation}\label{eq:subth_emb_cmt}[A^\perp]\;:\;\T^\perp_A \hookrightarrow \C_C,\end{equation}
discussed in \cite[7.10, 7.5]{Lu:Cmt}.  The components of $[A^\perp]$ are the monomorphisms
\begin{equation}\label{eq:gtjaa}[A^\perp]_{JK} = A^\perp_{JK} =  G^\T_{[J,A],[K,A]}\;:\;\Alg{\T}([J,A],[K,A]) \hookrightarrow \C([J,C],[K,C])\end{equation}
with $J,K \in \ob\J$.

Here we assume that $\C$ is $\J$-admissible (\ref{para:talg}, \ref{para:given_data}).  More generally, if $A$ is a $\T$-algebra in any $\V'$-category $\X$ with $\J$-cotensors, then in view of \ref{para:jth_enr_in_vprime} we can still form the commutant $\T^\perp_A$, which in this case is a $\J$-theory enriched in $\V'$. 
\end{parasub}

\begin{parasub}[\textbf{Commutation of cospans of $\J$-theories}]\label{para:cmtn}
Let $M:\T \rightarrow \sS$ and $N:\U \rightarrow \sS$ be morphisms of $\J$-theories.  Since $(\T,M)$ is a normal algebra on $I$ in $\sS$, we may consider its commutant $(\T^\perp_M,M^\perp)$, which is a normal algebra on $I$ \pref{para:cmt}, so $M^\perp:\T^\perp_M \rightarrow \sS$ is a morphism of $\J$-theories by \ref{para:mor_jth}.  Also, $M^\perp$ is faithful, by \ref{para:cmt}, so $M^\perp$ is a subtheory embedding.   We say that $M$ \textbf{commutes with} $N$ if $N$ factors through the commutant $M^\perp:\T^\perp_M \hookrightarrow \sS$ of $M$ (\cite[5.12, 7.8]{Lu:Cmt}, \cite[3.3.2]{Lu:FDistn}).  This relation of commutation is symmetric, as $M$ commutes with $N$ iff $N$ commutes with $M$ (\cite[5.8]{Lu:Cmt}, \cite[3.3.2]{Lu:FDistn}), in which case we say that $M$ and $N$ \textbf{commute}.  A morphism of $\J$-theories $M:\T \rightarrow \U$ is \textbf{central} if $M$ commutes with the identity morphism $1_\U:\U \rightarrow \U$.  A $\J$-theory $\T$ is \textbf{commutative} if $1_\T$ commutes with itself \cite[5.9, 5.13]{Lu:Cmt}.

If $(\T,A)$ and $(\U,B)$ are normal algebras on an object $C$ of $\C$, then we may ask whether the morphisms of $\J$-theories $[A]:\T \rightarrow \C_C$ and $[B]:\U \rightarrow \C_C$ commute \pref{para:full_jth}.  But the commutant of $[A]$ is isomorphic to the commutant $\T^\perp_A$ of $A$ \cite[7.10]{Lu:Cmt}, so $[A]$ commutes with $[B]$ if and only if $[B]$ factors through the subtheory embedding $[A^\perp]:\T^\perp_A \hookrightarrow \C_C$ of \eqref{eq:subth_emb_cmt}.
\end{parasub}

\section{Strong morphisms of algebras and the full algebra on an object}\label{sec:cls_alg}

By \ref{para:full_jth}, each object $C$ of $\C$ is the carrier of an algebra $(\C_C,\Gamma_C)$, where $\C_C$ is the full $\J$-theory of $C$ in $\C$.  We call $(\C_C,\Gamma_C)$ the \textbf{full algebra} on $C$.  In this section, we show that the full algebra has a universal property when considered as an object of a category of algebras over various theories, which we now define:

\begin{para}[\textbf{Algebras for various theories, with strong morphisms}]\label{para:str_mor_algs}
For each $\V$-category $\A$, we write $\A_\si$ for the subcategory of $\A_0$ consisting of all the objects of $\A$ and all the isomorphisms in $\A_0$, and we call $\A_\si$ the \textbf{groupoid core} of $\A$.  For each $\J$-theory $\T$, we write $\Alg{\T}_\si(\C)$ for the groupoid core of $\Alg{\T}(\C)$.  There is a functor $\Phi = \Alg{(-)}_\si(\C):\ThJ^\op \rightarrow \CAT$ that sends each morphism of $\J$-theories $M:\T \rightarrow \U$ to the functor $M^*:\Alg{\U}_\si(\C) \rightarrow \Alg{\T}_\si(\C)$ given by $M^*(B) = BM$ $(B \in \Alg{\T}_\si(\C))$.  Applying the Grothendieck construction to $\Phi$, we obtain a category
$$\Algs^\s(\C)$$
whose objects are algebras $(\T,A)$ in $\C$.  A morphism  \begin{equation}\label{eq:mor_in_algs_si}(M,f):(\T,A) \rightarrow (\U,B)\end{equation}
in $\Algs^\s(\C)$ will be called a \textbf{strong morphism of algebras} and consists of a morphism of $\J$-theories $M:\T \rightarrow \U$ and an isomorphism of $\T$-algebras $f:A \rightarrow M^*(B) = BM$.  There is a functor $\ca{\text{$-$}}:\Algs^\s(\C) \rightarrow \C_\si$ that is given on objects by $(\T,A) \mapsto |A|$ and sends each morphism \eqref{eq:mor_in_algs_si} to $|f|:|A| \rightarrow |BM| = |B|$.
\end{para}

\begin{para}[\textbf{Transport of structure for $\T$-algebras}]\label{para:transp_str}
Let $\T$ be a $\J$-theory.  By \ref{para:normal_talgs}, the faithful functor $G^\T_\nml:\Alg{\T}^\nml(\C)_0 \rightarrow \C_0$ creates limits, so this functor \textit{creates isomorphisms} in the sense of \cite[13.35, 13.36]{AHS}.  In view of \ref{para:normal_talgs}, it follows that for \textit{every} $\T$-algebra $A$ in $\C$ (not necessarily normal) and every isomorphism $c:|A| \rightarrow C$ in $\C$, there is a unique pair $(B,f)$ consisting of a normal $\T$-algebra $B$ on $C$ and a $\T$-homomorphism $f:A \rightarrow B$ with $|f| = c$; furthermore, $f$ is an isomorphism.
\end{para}

\begin{thm}[\textbf{The universal property of the full algebra}]\label{thm:univ_prop_gamma}
The functor $\ca{\text{$-$}}:\Algs^\s(\C) \rightarrow \C_\si$ has a fully faithful right adjoint $\Gamma_{(-)}$ that sends each object $C$ of $\C$ to its full algebra $(\C_C,\Gamma_C)$, and the counit of the resulting adjunction is an identity.  In particular,
$$\Algs^\s(\C)\bigl(A,\Gamma_C\bigr) \cong \C_\si\bigl(|A|,C\bigr)$$
naturally in $A \in \Algs^\s(\C)$ and $C \in \C_\si$.  Explicitly, given any algebra $(\T,A)$ and any isomorphism $c:|A| \rightarrow C$ in $\C$, there is a unique morphism $(M,f):(\T,A) \rightarrow (\C_C,\Gamma_C)$ in $\Algs^\s(\C)$ with $|f| = c$.
\end{thm}
\begin{proof}
It suffices to prove the last sentence, for the result then follows.  By \ref{para:transp_str}, there is a unique pair $(B,f)$ consisting of a normal $\T$-algebra $B$ on $C$ and an isomorphism $f:A \rightarrow B$ with $|f| = c$.  By \ref{para:full_jth} there is a unique morphism of $\J$-theories $M:\T \rightarrow \C_C$ such that $B = \Gamma_C M$, and the result follows, using \ref{para:full_jth}.
\end{proof}

\begin{para}[\textbf{Algebras on a fixed object}]\label{para:algs_on_c}
Let $C$ be an object of $\C$.   We write
$$\Alg{\T}(C)$$
to denote the fibre over $C$ of the functor $G^\T:\Alg{\T}(\C)_0 \rightarrow \C_0$, i.e. the category whose objects are $\T$-algebras on $C$ and whose morphisms are $\T$-homomorphisms $f$ with $|f| = 1_C$.  We also write $\Alg{\T}^\nml(C)$ for the full subcategory of $\Alg{\T}(C)$ consisting of normal $\T$-algebras on $C$.  By \ref{para:transp_str}, $\Alg{\T}^!(C)$ is discrete, so by \ref{para:normal_talgs} we find that the normalization $A^\nml$ of each $\T$-algebra $A$ on $C$ is the unique normal $\T$-algebra on $C$ with $A \cong A^\nml$ in $\Alg{\T}(C)$.  Hence $\Alg{\T}(C) \simeq \Alg{\T}^!(C)$, and therefore $\Alg{\T}(C)$ is both a groupoid and a preordered class, so merely a class equipped with an equivalence relation $\cong$.  We write
$$\Algs(C)$$
for the fibre over $C$ of the functor $\ca{\text{$-$}}:\Algs^\s(\C) \rightarrow \C_\si$ \pref{para:str_mor_algs}, so that $\Algs(C)$ is a category whose objects are algebras $(\T,A)$ on $C$, and whose morphisms $(M,f):(\T,A) \rightarrow (\U,B)$ consist of a morphism $M:\T \rightarrow \U$ and an isomorphism $f:A \rightarrow M^*(B)$ in $\Alg{\T}(C)$.  Since the groupoid $\Alg{\T}(C)$ is also preordered class, such a morphism $(M,f)$ in $\Algs(C)$ may be described more succinctly as a morphism of $\J$-theories $M:\T \rightarrow \U$ such that $A \cong M^*(B)$ in $\Alg{\T}(C)$. 
\end{para}

\begin{cor}\label{thm:ind_morph}
The full algebra $(\C_C,\Gamma_C)$ on an object $C$ of $\C$ is a terminal object of $\Algs(C)$.  Explicitly, given any $\T$-algebra $A$ on an object $C$ of $\C$ there is a unique morphism of $\J$-theories $M:\T \rightarrow \C_C$ such that $A \cong M^*(\Gamma_C)$ in $\Alg{\T}(C)$, with the notations of \ref{para:full_jth}, \ref{para:str_mor_algs}, and \ref{para:algs_on_c}.
\end{cor}
\begin{proof}
This is obtained by taking $c = 1_C$ in \ref{thm:univ_prop_gamma}.
\end{proof}

Corollary \ref{thm:ind_morph} allows us to generalize to arbitrary $\T$-algebras the notation $[A]$ that we employed in \ref{para:full_jth} for a given normal $\T$-algebra $A$:

\begin{notn}\label{defn:ind_mor}
Given any $\T$-algebra $A$ on an object $C$ of $\C$, we write
$$[A]\;:\;\T \longrightarrow \C_C$$
to denote the unique morphism of $\J$-theories with $\Gamma_C [A] \cong A$ in $\Alg{\T}(C)$.  In the case where $A$ is a normal $\T$-algebra on $C$, we may characterize $[A]$ as the unique morphism of $\J$-theories with $\Gamma_C [A] = A$ \pref{para:full_jth}.  For an arbitrary $\T$-algebra $A$ on $C$, since $A \cong A^!$ in $\Alg{\T}(C)$ we find that $[A] = [A^\nml]$, so $\Gamma_C[A] = A^\nml$, and therefore
\begin{equation}\label{eq:ind_mor_comps_via_normaln}[A]_{JK} = A^\nml_{JK}\;:\;\T(J,K) \longrightarrow \C_C(J,K) = \C([J,C],[K,C])\end{equation}
for all $J,K \in \ob\J$.
\end{notn}

\begin{para}\label{para:morph_between_full_ths}
Let $H:\X \rightarrow \Y$ be a $\V$-functor that preserves $\J$-cotensors, where both $\X$ and $\Y$ are $\V$-categories with $\J$-cotensors, and let $X$ be an object of $\X$.  Then the composite $H\Gamma_X:\X_X \rightarrow \Y$ is an algebra on $HX$ in $\Y$, so by \ref{thm:ind_morph} there is a unique morphism
$$H_X = [H\Gamma_X]:(\X_X,H\Gamma_X) \rightarrow (\Y_{HX},\Gamma_{HX})$$
in $\Algs(HX)$.  For each algebra $(\T,A)$ on $X$, since $[A]:(\T,A) \rightarrow (\X_X,\Gamma_X)$ is a morphism in $\Algs(X)$, it follows that $[A]:\T \rightarrow \X_X$ also underlies a morphism $[A]:(\T,HA) \rightarrow (\X_X,H\Gamma_X)$ in $\Algs(HX)$, and the diagrams
$$
\xymatrix{
(\T,HA) \ar[d]_{[A]} \ar[dr]^{[HA]} & \;\;\;\;\text{in $\Algs(HX)$}                     &                      & \T \ar[d]_{[A]} \ar[dr]^{[HA]} & \;\;\;\;\text{in $\ThJ$}\\
(\X_X,H\Gamma_X) \ar[r]^(.45){H_X} &  (\Y_{HX},\Gamma_{HX}) &   & \C_C \ar[r]^(.45){H_X} &  \Y_{HX} 
}
$$
commute, by the uniqueness of $[HA]$.  Also, if $H$ is faithful, then $H_X$ is a subtheory embedding, because by definition $\Gamma_{HX}H_X \cong H\Gamma_X$ and $\Gamma_X$ is fully faithful.
\end{para}

\section{Commuting pairs of algebras}\label{sec:comm_pairs_algs}

We noted in \ref{para:cmtn} that the concept of commutation of cospans of $\J$-theories can be applied with respect to the morphisms of $\J$-theories $[A]$ and $[B]$ determined by a pair of normal algebras $(\T,A)$ and $(\U,B)$ on the same carrier.  We now formulate a convenient generalization of this notion of commutation that is applicable to arbitrary algebras on a given carrier and can be expressed without reference to $[A]$ or $[B]$.

\begin{prop}\label{thm:equiv_charns_b_pipe_a}
Let $A$ and $B$ be algebras on an object $C$ of $\C$.  Then there is at most one morphism $M:B \rightarrow A^\perp$ in $\Algs(C)$, where $A^\perp$ is the commutant of $A$ \eqref{eq:cmtnt_of_alg}.  Writing the algebras $A$ and $B$ as $(\T,A)$ and $(\U,B)$, respectively, a morphism $M:B \rightarrow A^\perp$ in $\Algs(C)$ is equivalently a morphism of $\J$-theories $M:\U \rightarrow \T^\perp_A$ that satisfies the following equivalent conditions: (1) $A^\perp M \cong B$ in $\Alg{\U}(C)$, (2) $[A^\perp] M = [B]$.  If $B$ is normal, then (1) and (2) are equivalent to the equation $A^\perp M = B$.
\end{prop}
\begin{proof}
By \ref{para:algs_on_c}, a morphism $M:B \rightarrow A^\perp$ in $\Algs(C)$ is equivalently a morphism of $\J$-theories $M:\U \rightarrow \T^\perp_A$ satisfying (1).  By \ref{para:cmt}, $A^\perp = \Gamma_C [A^\perp]$, so if $M:\U \rightarrow \T^\perp_A$ is a morphism of $\J$-theories satisfying (2), then $A^\perp M =  \Gamma_C [A^\perp] M = \Gamma_C [B] \cong B$ in $\Alg{\U}(C)$, so (1) holds.  Conversely, if (1) holds, then $\Gamma_C [A^\perp] M = A^\perp M \cong B$ in $\Alg{\U}(C)$, so (2) holds, by the uniqueness of $[B]$ \pref{defn:ind_mor}.  Since $[A^\perp]$ is a subtheory embedding \pref{para:cmt}, there is at most one morphism of $\J$-theories $M$ satisfying (2).  If $B$ is normal and (1) holds, then $A^\perp M = B$ by \ref{para:algs_on_c}, since $A^\perp M$ is also normal \pref{para:cmt}.
\end{proof}

\begin{defn}\label{defn:cmt}
Let $A$ and $B$ be algebras on an object $C$ of $\C$.  We say that $A$ \textbf{commutes with} $B$ if there exists a (necessarily unique, \ref{thm:equiv_charns_b_pipe_a}) morphism $B \rightarrow A^\perp$ in $\Algs(C)$, which we then denote by
$$[B|A]\;:\;B \longrightarrow A^\perp\;.$$
Writing $A$ and $B$ as $(\T,A)$ and $(\U,B)$, respectively, we say that $(A,B)$ is a \textbf{commuting $\T$-$\U$-algebra pair} on $C$ if $A$ commutes with $B$, equivalently, if there exists a (necessarily unique) morphism of $\J$-theories $M:\U \rightarrow \T^\perp_A$ that satisfies the equivalent conditions (1) and (2) of \ref{thm:equiv_charns_b_pipe_a}, in which case we write $M$ as
\begin{equation}\label{eq:b_pipe_a}[B|A]\;:\;\U \longrightarrow \T^\perp_A\;.\end{equation}
If $A$ commutes with $B$ then $A^\perp[B|A] \cong B$ in $\Alg{\U}(C)$, and if $B$ is normal then $A^\perp [B|A] = B$ by \ref{thm:equiv_charns_b_pipe_a}.
\end{defn}

It is immediate that this notion of commutation for algebras specializes to recover the notion of commutation for morphisms of $\J$-theories \pref{para:cmtn}, as such morphisms may be regarded as normal algebras \pref{para:mor_jth}.  Conversely, we shall later confirm in \ref{thm:cmtn_via_normaln} that $A$ commutes with $B$ if and only if the associated morphisms of $\J$-theories $[A]$ and $[B]$ commute in the sense of \ref{para:cmtn}.

\begin{rem}\label{para:alg_com_w_cmt}
Every algebra $A$ commutes with its commutant $A^\perp$, as witnessed by the identity morphism on $A^\perp$ in $\Algs(C)$, where $C =|A|$.
\end{rem}

\begin{para}\label{para:elem_descn_cmtn}
We now record elementary descriptions of commutation and of $[B|A]$ that can be useful in treating some examples.  Let  $(\T,A)$ and $(\U,B)$ be algebras on an object $C$ of $\C$.  By \ref{thm:equiv_charns_b_pipe_a}, $A$ commutes with $B$ if and only if the morphism of $\J$-theories $[B]:\U \rightarrow \C_C$ factors through the subtheory embedding $[A^\perp]:\T^\perp_A \hookrightarrow \C_C$ of \ref{para:cmt}, in which case $[B|A]:\U \rightarrow \T^\perp_A$ is the unique morphism of $\J$-theories such that
$$[A^\perp][B|A] = [B]\;.$$
But in view of \ref{para:mor_jth}, $[B]$ factors through $[A^\perp]$ if and only if $[B]_{JI}:\U(J,I) \rightarrow \C_C(J,I)$ factors through the monomorphism $[A^\perp]_{JI}:\T^\perp_A(J,I) \hookrightarrow \C_C(J,I)$ for each $J \in \ob\J$.  Hence, by \eqref{eq:gtjaa} and \eqref{eq:ind_mor_comps_via_normaln} we find that $A$ commutes with $B$ if and only if for each $J \in \ob\J$ there exists a (necessarily unique) morphism $[B|A]_{JI}$ such that the following diagram commutes:
\begin{equation}\label{eq:diagr_b_pipe_a}
\xymatrix{
\U(J,I) \ar[drr]_{B^\nml_{JI}} \ar@{-->}[rr]^(.4){[B|A]_{JI}} & & \Alg{\T}(\C)([J,A],A) \ar@{^(->}[d]^{G^\T_{[J,A],A}}\\
                          & & \C([J,C],C)
}
\end{equation}
\end{para}

\section{Bifold algebras}\label{sec:bifold_algs}

\begin{notn}\label{notn:pres_jcots_in_each_var}
Given three $\V$-categories $\E,\F,\G$ with $\J$-cotensors, we write
$$[\E,\F;\G]_\J$$
to denote the full sub-$\V'$-category of $[\E\otimes\F,\G]$ consisting of the $\V$-functors $D:\E \otimes \F \rightarrow \G$ that preserve $\J$-cotensors in each variable separately, meaning that the $\V$-functors $D(-,F)$ and $D(E,-)$ preserve $\J$-cotensors for all $E \in \ob\E$ and $F \in \ob\F$.
\end{notn}

\begin{defn}\label{defn:tu_alg}
Given $\J$-theories $\T$ and $\U$, a \textbf{$(\T,\U)$-algebra} in $\C$ is a $\V$-functor $D:\T \otimes \U \rightarrow \C$ that preserves $\J$-cotensors in each variable separately.  A \textbf{bifold algebra} in $\C$ is a triple $(\T,\U,D)$ consisting of $\J$-theories $\T$ and $\U$ and a $(\T,\U)$-algebra $D$.  We often write simply $D$ to denote the bifold algebra $(\T,\U,D)$.    We call $|D| = D(I,I)$ the \textbf{carrier} of $D$, and if $|D| = C$ then we call $D$ a \textbf{bifold algebra on $C$}.  Using the notation of \ref{notn:pres_jcots_in_each_var}, we write
\begin{equation}\label{eq:def_tualg}\Alg{(\T,\U)}(\C) = [\T,\U;\C]_\J\;,\end{equation}
so that the objects of $\Alg{(\T,\U)}(\C)$ are $(\T,\U)$-algebras in $\C$.  We later show that $[\T,\U;\C]_\J$ necessarily exists as a $\V$-category \pref{thm:tsalgs_equiv_talgs_in_salg} under our blanket assumptions in \ref{para:given_data}.  The notation \eqref{eq:def_tualg} is contravariantly functorial in $\T,\U \in \ThJ$, so that we obtain a functor $\Alg{(-,\blanktwo)}(\C) = [-,\blanktwo;\C]_\J:\ThJ^\op \times \ThJ^\op \rightarrow \eCAT{\V'}$ (which in fact takes its values in $\VCAT$, \ref{thm:tsalgs_equiv_talgs_in_salg}).  There is a $\V'$-functor \begin{equation}\label{eq:carrier_func_on_tualg}|\text{$-$}|:\Alg{(\T,\U)}(\C) \rightarrow \C\end{equation}
 given by evaluation at $(I,I)$.  For each object $C$ of $\C$ we write $\Alg{(\T,\U)}(C)$ to denote the fibre over $C$ of the ordinary functor underlying \eqref{eq:carrier_func_on_tualg}.
\end{defn}

\begin{rem}\label{rem:equiv_formns_of_tualgs}
There are natural isomorphisms
\begin{equation}\label{eq:equiv_descns_tualgs}
[\T,[\U,\C]_\J]_\J \;\;\cong\;\; [\T,\U;\C]_\J \;\;\cong\;\; [\U,\T;\C]_\J \;\;\cong\;\; [\U,[\T,\C]_\J]_\J\;,
\end{equation}
so that $(\T,\U)$-algebras can be described equivalently as $\T$-algebras in $\Alg{\U}(\C)$, as $\U$-algebras in $\Alg{\T}(\C)$, or as $(\U,\T)$-algebras.
\end{rem}

\begin{notn}\label{notn:bif}
Given a $(\T,\U)$-algebra $D$, we write
$$D_\ell = D(-,I):\T \longrightarrow \C\;,\;\;\;\;\;\;\;D_r = D(I,-):\U \longrightarrow \C\;,$$
noting that $D_\ell$ is a $\T$-algebra on $|D|$ and $D_r$ is a $\U$-algebra on $|D|$.  We call $D_\ell$ and $D_r$ the \textbf{left face} and \textbf{right face} of $D$, respectively.  We also introduce notations for the three transposes of $D$ under the isomorphisms \eqref{eq:equiv_descns_tualgs}, writing
$$D_{\ell r}:\T \longrightarrow [\U,\C]_\J = \Alg{\U}(\C),\;\;\;\;\;D_{r\ell}:\U \longrightarrow [\T,\C]_\J = \Alg{\T}(\C),\;\;\;\;\text{and}$$
$$D^\circ:\U \otimes \T \longrightarrow \C$$
for the $\V$-functors given by $(D_{\ell r}J)K = D(J,K) = D^\circ(K,J) = (D_{r\ell}K)J$, $\V$-naturally in $J \in \T$, $K \in \U$.  Given a morphism $f:D \rightarrow E$ in $\Alg{(\T,\U)}(\C)$, we write
$$f_\ell = f_{-I}:D_\ell \rightarrow E_\ell\;,\;\;\;\;f_r = f_{I-}:D_r \rightarrow E_r\;,$$
and we define $f^\circ:D^\circ \rightarrow E^\circ$ by $f^\circ_{KJ} = f_{JK}$ $(J \in \T, K \in \U)$.
\end{notn}

In Section \ref{sec:exa_bif} we discuss various examples of bifold algebras, but for now we take note of one of the most basic classes of examples:

\begin{exa}[\textbf{Bimodules}]\label{exa:bim}
Consider the system of arities $\{I\} \hookrightarrow \V$ consisting of just the unit object $I$ of $\V$.  By \ref{para:exa_sys_ar}, $\{I\}$-theories $\T$ are precisely monoids $T$ in $\V$, and $\T$-algebras are left $T$-modules in $\V$, i.e. objects of $\V$ equipped with an associative, unital $T$-action.  Given monoids $T$ and $S$ in $\V$, let us regard $T$ and $U = S^\op$ as $\{I\}$-theories $\T$ and $\U$, respectively.  Then a $(\T,\U)$-algebra is equivalently described as a \mbox{\textit{$T$-$S$-bimodule}} in $\V$.  In particular, when $\V = \Ab$ is the category of abelian groups, with the usual tensor product and the system of arities $\{\ZZ\}$, we recover the usual notion of $T$-$S$-bimodule for a pair of rings $T$ and $S$.  
\end{exa}

Note that bifold algebras are quite different from \textit{bimodels} in the sense of Wraith \cite{Wra:AlgTh,Wra:AlgOverTh}, which provide a different way of generalizing bimodules.  Indeed, in the case where $\J = \V = \Set$, the concept of bimodel was defined in \cite[\S 8]{Wra:AlgTh} and \cite[\S 2]{Wra:AlgOverTh} as a coproduct-preserving functor\footnote{The notion of algebraic theory employed in \cite{Wra:AlgTh,Wra:AlgOverTh} is the dual of the notion of Lawvere-Linton theory.} $\T^\op \rightarrow \Alg{\U}(\Set)$, equivalently, a $\T$-algebra in $(\Alg{\U}(\Set))^\op$.

\begin{para}
Let $(\T,\U,D)$ be a bifold algebra, and write $A = D_\ell$, $B = D_r$, $C = |D|$, and $E = D_{r\ell}$ with the notation of \ref{notn:bif}.  Then $E:\U \rightarrow \Alg{\T}(\C)$ is a $\U$-algebra on $A$, so there is a unique morphism of $\J$-theories $[E]:\U \rightarrow \T^\perp_A$ such that $\Gamma_A[E] \cong E$ in $\Alg{\U}(A)$ \pref{defn:ind_mor}, recalling that $\T^\perp_A = \Alg{\T}(\C)_A$.  Hence, with the notation of \eqref{eq:a_lowerperp}, $A^\perp[E] =  G^\T\Gamma_A[E] \cong G^\T E = B$ in $\Alg{\U}(C)$, so $[E]$ witnesses that $A$ commutes with $B$, by \ref{defn:cmt}, and $[E] = [B|A]$.  Since we can apply this argument also to $D^\circ$, this proves the following:
\end{para}

\begin{prop}\label{thm:dl_commutes_w_dr}
Let $(\T,\U,D)$ be a bifold algebra.  Then $D_\ell$ commutes with $D_r$, and $[D_r|D_\ell] = [D_{r\ell}]:\U \rightarrow \T^\perp_{D_\ell}$ with the notation of \ref{defn:ind_mor}, \eqref{eq:b_pipe_a}, \ref{notn:bif}.  Also, $D_r$ commutes with $D_\ell$, and $[D_\ell|D_r] = [D_{\ell r}]$.
\end{prop}

\section{Tensor products of \texorpdfstring{$\J$}{J}-theories}\label{thm:tens_prods}

Given $\J$-theories $\T$ and $\U$, we now discuss the possibility of forming a \textit{tensor product} of $\J$-theories, $\otimesJ{\T}{\U}$, whose algebras are equivalently described as $(\T,\U)$-algebras in the sense of \ref{defn:tu_alg}.  This we shall not require in the sequel, but we treat it here for completeness.  We assume \textit{in this section only} that $\J$ is small and that $\V$ is \textit{locally bounded} \cite[\S 6.1]{Ke:Ba} (and so, in particular, complete).  The following is a corollary to \cite[\S 6.5]{Ke:Ba}:
\begin{prop}\label{thm:tensor_in_vcatj}
Given any small $\V$-categories $\T$ and $\U$ with $\J$-cotensors, there is a small $\V$-category $\totimesJ{\T}{\U}$ that has $\J$-cotensors and is equipped with a $\V$-functor $P:\T \otimes \U \rightarrow \totimesJ{\T}{\U}$ that preserves $\J$-cotensors in each variable separately and has the following property:  For every $\V$-category $\X$ with $\J$-cotensors, there is an equivalence of $\V$-categories $[\totimesJ{\T}{\U},\X]_\J \simeq [\T,\U;\X]_\J$ given by composing with $P$.
\end{prop}
\begin{proof}
This follows from \cite[\S 6.5]{Ke:Ba}, which treats the general (dual) case of $\F$-colimits for a small class of weights $\F$, rather than $\F$-limits as in our case of $\J$-cotensors.
\end{proof}

\begin{notn}\label{def:otimesJ}
For the remainder of this section, we let $\T$ and $\U$ be $\J$-theories and let $\otimesJ{\T}{\U}$ denote the full $\J$-theory of the object $P(I,I)$ of $\totimesJ{\T}{\U}$ \pref{para:full_jth}.
\end{notn}

\begin{lem}\label{thm:lem_tensor_prod_jth}
There is an equivalence $\totimesJ{\T}{\U} \simeq \otimesJ{\T}{\U}$ such that the composite $\T \otimes \U \xrightarrow{P} \totimesJ{\T}{\U} \xrightarrow{\sim} \otimesJ{\T}{\U}$ is given on objects by $(J,K) \mapsto J \otimes K$.
\end{lem}
\begin{proof}
Let $\E = \totimesJ{\T}{\U}$ and $E = P(I,I)$.  Then $\otimesJ{\T}{\U} = \E_E$ is equipped with a fully faithful $\E_E$-algebra $\Gamma_E:\E_E \rightarrow \E$ given on objects by $J \mapsto [J,E]$, by \ref{para:full_jth}, and we shall show that $\Gamma_E$ is an equivalence.  The $\V$-functor $P:\T \otimes \U \rightarrow \E$ is a $(\T,\U)$-algebra, and it therefore sends each object $(J,K) \in \ob(\T \otimes \U) = \ob\J \times \ob\J$ to a cotensor
$$P(J,K) = [J,P(I,K)] \cong [J,[K,P(I,I)]] \cong [J \otimes K,E]$$
of $E$ by $J \otimes K$ in $\E$, so $P(J,K) \cong \Gamma_E(J \otimes K)$.  Hence, since $\Gamma_E$ is fully faithful, we deduce that there is a $\V$-functor $Q:\T \otimes \U \rightarrow \E_E$ given on objects by $Q(J,K) = J \otimes K$ such that $\Gamma_E Q \cong P$.  Since $P$ preserves $\J$-cotensors in each variable, and $\Gamma_E$ is fully faithful, it follows that $Q$ preserves $\J$-cotensors in each variable.  Therefore, by \ref{thm:tensor_in_vcatj} there is a $\V$-functor $Q^\sharp:\E \rightarrow \E_E$ that preserves $\J$-cotensors and has $Q^\sharp P \cong Q$.  Hence $\Gamma_E Q^\sharp:\E \rightarrow \E$ preserves $\J$-cotensors and has $\Gamma_E Q^\sharp P \cong \Gamma_E Q \cong P$ so by \ref{thm:tensor_in_vcatj} we deduce that $\Gamma_E Q^\sharp \cong 1_\E$.  But $\Gamma_E$ is fully faithful, so it also follows that $Q^\sharp \Gamma_E \cong 1_{\E_E}$ since $\Gamma_E Q^\sharp \Gamma_E \cong \Gamma_E$.  Hence $\Gamma_E$ and $Q^\sharp$ are equivalences.
\end{proof}

\begin{thm}\label{thm:tensor_prod_jth}
Given $\J$-theories $\T$ and $\U$, there is a $\J$-theory $\otimesJ{\T}{\U}$ equipped with a $(\T,\U)$-algebra $\otimes_{\kern -.2ex\T\U}:\T \otimes \U \rightarrow \otimesJ{\T}{\U}$ that is given on objects by $(J,K) \mapsto J \otimes K$ and has the following property:  For every $\V$-category $\X$ with $\J$-cotensors, there is an equivalence of $\V$-categories
$$\Alg{(\otimesJ{\T}{\U})}(\X) = [\otimesJ{\T}{\U},\X]_\J \;\simeq\; [\T,\U;\X]_\J = \Alg{(\T,\U)}(\X)$$
given by composition with $\otimes_{\kern -.2ex\T\U}$.
\end{thm}
\begin{proof}
This follows from \ref{thm:tensor_in_vcatj} and \ref{thm:lem_tensor_prod_jth}.
\end{proof}

\begin{defn}
We call the $\J$-theory $\otimesJ{\T}{\U}$ of \ref{thm:tensor_prod_jth} and \ref{def:otimesJ} the \textbf{tensor product} of the $\J$-theories $\T$ and $\U$.
\end{defn}

\section{A two-sided fibration of bifold algebras}\label{sec:two_sided_fibr}

For the remainder of the paper, we return to our general setting, imposing only the assumptions in \ref{para:given_data}.  There are multiple different categories of bifold algebras over various theories, but the one that will be most relevant in this paper will be obtained via the following general method for constructing split \textit{two-sided fibrations} \cite{Str:FibrYon2Cats}:

\begin{para}[\textbf{Two-sided Grothendieck construction}]\label{para:str_gr_constr_two-sided}
Let $\A$ and $\B$ be categories, and let $\Psi:\A^\op \times \B \rightarrow \CAT$ be a functor.  For each morphism $a:A \rightarrow A'$ in $\A$ and each object $B$ of $\B$, we write $a^* = \Psi(a,B):\Psi(A',B) \rightarrow \Psi(A,B)$.  Analogously, given an object $A$ of $\A$ and a morphism $b:B \rightarrow B'$ in $\B$, we write $b_! = \Psi(A,b):\Psi(A,B) \rightarrow \Psi(A,B')$.  There is a category $\mathsf{TwoSided}(\A,\B,\Psi)$ whose objects are triples $(A,B,X)$ with $A \in \ob\A$, $B \in \ob\B$, and $X \in \ob\Psi(A,B)$, where a morphism $(a,b,x):(A,B,X) \rightarrow (A',B',X')$ consists of morphisms $a:A \rightarrow A'$ in $\A$, $b:B \rightarrow B'$ in $\B$, and $x:b_!(X) \rightarrow a^*(X')$ in $\Psi(A,B')$.  Given morphisms $(a,b,x):(A,B,X) \rightarrow (A',B',X')$ and $(c,d,y):(A',B',X') \rightarrow (A'',B'',X'')$ in $\mathsf{TwoSided}(\A,\B,\Psi)$, the composite is $(c \cdot a, d \cdot b,a^*(y) \cdot d_!(x)):(A,B,X) \rightarrow (A'',B'',X'')$.  The category $\mathsf{TwoSided}(\A,\B,\Psi)$ is equipped with forgetful functors $P:\mathsf{TwoSided}(\A,\B,\Psi) \rightarrow \A$ and $Q:\mathsf{TwoSided}(\A,\B,\Psi) \rightarrow \B$ that constitute a \textit{split two-sided fibration} from $\A$ to $\B$, meaning that $P$ is a split fibration, $Q$ is a split op-fibration, and certain compatibility conditions are satisfied; split two-sided fibrations were introduced by Street in \cite[p. 123]{Str:FibrYon2Cats}, where they are called \textit{split bifibrations from $\A$ to $\B$}.
\end{para}

We now define the notion of \textit{strong cross-morphism} of bifold algebras, which will play a key role in this paper, as it enables Theorem \ref{thm:rcom_lcom} and its corollaries.

\begin{defn}\label{defn:bifold_algs_str_cross_morphs}
For each pair of $\J$-theories $\T$ and $\U$, we write $\Alg{(\T,\U)}_\si(\C)$ to denote the groupoid core of $\Alg{(\T,\U)}(\C)$.  In view of \ref{defn:tu_alg}, there is a functor $\Psi = \Alg{(-,\blanktwo)}_\si(\C):\ThJ^\op \times \ThJ^\op \rightarrow \CAT$ that is given on objects by $(\T,\U) \mapsto \Alg{(\T,\U)}_\si(\C)$.  Given morphisms of $\J$-theories $M:\T \rightarrow \T'$ and $N:\U \rightarrow \U'$, the associated functor $\Psi(M,N):\Alg{(\T',\U')}_\si(\C) \rightarrow \Alg{(\T,\U)}_\si(\C)$ is given by $$\Psi(M,N)(D) = D(M-,N\blanktwo)$$
naturally in $D \in \Alg{(\T,\U)}_\si(\C)$.  Writing
$$\BAlg^\sx(\C)\;=\;\mathsf{TwoSided}(\ThJ,\ThJ^\op,\Psi)\;,$$
with the notation of \ref{para:str_gr_constr_two-sided}, we find that $\BAlg^\sx(\C)$ is a category whose objects are bifold algebras $(\T,\U,D)$ in $\C$.  A morphism
$$(M,N,f)\;:\;(\T,\U,D) \longrightarrow (\T',\U',D')\;\;\;\;\text{in $\BAlg^\sx(\C)$}$$
will be called a \textbf{strong cross-morphism} and consists of morphisms of $\J$-theories $M:\T \rightarrow \T'$ and $N:\U' \rightarrow \U$ (noting the two different directions) together with an isomorphism of $(\T,\U')$-algebras
$$f\;:\;D(-,N\blanktwo) \overset{\sim}{\longrightarrow} D'(M-,\blanktwo)\;.$$
We call $\BAlg^\sx(\C)$ the \textbf{category of bifold algebras and strong cross-morphisms}.  By \ref{para:str_gr_constr_two-sided}, the forgetful functors in the following diagram constitute a two-sided fibration:
\begin{equation}\label{eq:2sided_fibr_balg}
\xymatrix{
& \BAlg^\sx(\C) \ar[dl]_P \ar[dr]^Q &\\
\ThJ & & \ThJ^\op
}
\end{equation}
Also, there is a functor $\ca{\text{$-$}}:\BAlg^\sx(\C) \rightarrow \C_\si$ that is given on objects by $(\T,\U,D) \mapsto |D|$ and on morphisms by $(M,N,f) \mapsto f_{II}$.
\end{defn}

\begin{rem}\label{rem:pq_jcons}
Since the categories $\Alg{(\T,\U)}_\si(\C)$ are groupoids, it follows that the functors $P$ and $Q$ in \eqref{eq:2sided_fibr_balg} are jointly conservative, meaning that a strong cross-morphism $(M,N,f)$ is invertible as soon as both $M$ and $N$ are invertible.
\end{rem}

\begin{para}[\textbf{The canonical anti-involution on strong cross-morphisms}]
There is an isomorphism of categories
$$\varbigcirc \;:\;\BAlg^\sx(\C) \overset{\sim}{\longrightarrow} \BAlg^\sx(\C)^\op$$
that is given on objects by $\varbigcirc(\T,\U,D) = (\U,\T,D^\circ)$, with the notation of \ref{notn:bif}, and sends each strong cross-morphism $(M,N,f):(\T,\U,D) \rightarrow (\T',\U',E)$ to the strong cross-morphism
$$\varbigcirc(M,N,f) = (N,M,(f^\circ)^{-1})\;:\; (\U',\T',E^\circ) \longrightarrow (\U,\T,D^\circ),$$
where $(f^\circ)^{-1}:E^\circ(-,M\blanktwo) \rightarrow D^\circ(N-,\blanktwo)$ is the isomorphism in $\Alg{(\U',\T)}(\C)$ consisting of the morphisms
$$(f^\circ_{KJ})^{-1} = f^{-1}_{JK}\;:\;E^\circ(K,MJ) = E(MJ,K) \rightarrow D(J,NK) = D^\circ(NK,J)$$
with $K \in \U'$ and $J \in \T$.  Clearly $\varbigcirc^\op:\BAlg^\sx(\C)^\op \rightarrow \BAlg^\sx(\C)$ is an inverse of $\varbigcirc$.
\end{para}

\begin{para}[\textbf{The left and right face functors on strong cross-morphisms}]\label{para:lr_face_functors}
We now define functors $L$ and $R$ as in the following diagram:
$$
\xymatrix{
& \BAlg^\sx(\C) \ar[dl]_L \ar[dr]^R &\\
\Algs^\s(\C) & & \Algs^\s(\C)^\op
}
$$
The \textbf{left face functor}, $L$, sends each bifold algebra $(\T,\U,D)$ to its left face $L(\T,\U,D) = (\T,D_\ell)$ and sends each strong cross-morphism of bifold algebras $(M,N,f):(\T,\U,D) \rightarrow (\T',\U',D')$ to the strong morphism of algebras
$$L(M,N,f) = (M,f_\ell):(\T,D_\ell) \rightarrow (\T',D'_\ell)\;,$$
where we employ the notation of \ref{notn:bif} in writing 
$$f_\ell = f_{-I}\;:\;D_\ell = D(-,I) = D(-,NI) \overset{\sim}{\longrightarrow} D'(M-,I) = M^*(D'_\ell)\;.$$
The \textbf{right face functor}, $R$, is given on objects by $R(\T,\U,D) = (\U,D_r)$ and sends each morphism $(M,N,f):(\T,\U,D) \rightarrow (\T',\U',D')$ to the strong morphism of algebras
$$R(M,N,f) = (N,f^{-1}_r):(\U',D'_r) \rightarrow (\U,D_r)\;,$$ where $f^{-1}_r:D'_r \rightarrow N^*(D_r)$ is the inverse of
$$f_r = f_{I-}\;:\;N^*(D_r) = D(I,N-) \overset{\sim}{\longrightarrow} D'(MI,-) = D'(I,-) = D_r'\;.$$
With these notations, $L = R^\op \varbigcirc$, and the following diagram commutes:
\begin{equation}\label{eq:lr_commt_w_carrier}
\xymatrix{
\Algs^\s(\C) \ar[d]_{|-|} & \BAlg^\sx(\C) \ar[l]_L \ar[d]_{|-|} \ar[r]^R & \Algs^\s(\C)^\op \ar[d]_{|-|^\op} \\
\C_\si \ar@{=}[r] & \C_\si \ar[r]^{i}_\sim & \C_\si^\op\\
}
\end{equation}
where $i$ denotes the identity-on-objects isomorphism given on morphisms by $c \mapsto c^{-1}$.
\end{para}

\section{The face adjunctions and the commutant adjunction}\label{sec:face_adj_cmt_adj}

The following theorem is central to the methodology of this paper:

\begin{thm}\label{thm:rcom_lcom}
The left face functor $L$ has a fully faithful left adjoint $\RCom$, and the right face functor $R$ has a fully faithful right adjoint $\LCom$, as in the following diagram:
\begin{equation}\label{eq:rcom_lcom}
\xymatrix{
& \BAlg^\sx(\C) \ar[dl]^L \ar[dr]_R^[@!-30]{\top} &\\
\Algs^\s(\C) \ar@/^3ex/@{ >-->}[ur]_[@!30]{\bot}^\RCom & & \Algs^\s(\C)^\op \ar@/_3ex/@{  >-->}[ul]_\LCom
}
\end{equation}
The left adjoint functor $\RCom$ sends each algebra $(\T,A)$ to a bifold algebra  $\RCom(\T,A)$ whose left face is $(\T,A)$ and whose right face is the commutant $(\T,A)^\perp = (\T^\perp_A,A^\perp)$ of $(\T,A)$.  The right adjoint functor $\LCom$ sends each algebra $(\T,A)$ to a bifold algebra whose left face is $(\T,A)^\perp$ and whose right face is $(\T,A)$. Furthermore, the unit of $\RCom \dashv L$ and the counit of $R \dashv \LCom$ are identities, so $\RCom$ is a section of $L$ and $\LCom$ is a section of $R$.  Also
\begin{equation}\label{eq:lcom_vs_rcom}\varbigcirc\RCom = \LCom^\op\;:\;\Algs^\s(\C) \longrightarrow \BAlg^\sx(\C)^\op
\end{equation}
\end{thm}
\begin{proof}
Given an algebra $(\T,A)$, the commutant $\T^\perp_A$ is the full $\J$-theory of the object $A$ of $\Alg{\T}(\C)$, so $\T^\perp_A$ is the domain of the full algebra $\Gamma_A:\T^\perp_A \rightarrow \Alg{\T}(\C)$ on $A$ \eqref{eq:cl_alg_cmt}.  Let $\Gamma^A:\T \otimes \T^\perp_A \rightarrow \C$ denote the transpose of $\Gamma_A$, so that $\Gamma^A(J,K) = (\Gamma_A K)J$ $(J \in \T, K \in \T^\perp_A)$.  Then $\Gamma^A$ is a $(\T,\T^\perp_A)$-algebra whose left face is $\Gamma_A I = A$ and whose right face is $G^\T \Gamma_A = A^\perp$ \pref{eq:a_lowerperp}.  Letting $\RCom(\T,A) = (\T,\T^\perp_A,\Gamma^A)$, we find that $L\RCom(\T,A) = (\T,A)$ and $R\RCom(\T,A) = (\T^\perp_A,A^\perp)$.  We claim that the identity morphism $1 = (1_\T,1_A):(\T,A) \rightarrow L\RCom(\T,A)$ in $\Algs^\s(\C)$ equips $\RCom(\T,A)$ with the structure of a universal arrow for the functor $L$.  To show this, let $(\T',\U',D)$ be a bifold algebra in $\C$, and let $(M,g):(\T,A) \rightarrow L(\T',\U',D) = (\T',D_\ell)$ be a morphism in $\Algs^\s(\C)$.  Then $M:\T \rightarrow \T'$ in $\ThJ$, and $g:A \rightarrow M^*(D_\ell)$ is an isomorphism in $\Alg{\T}(\C)$.  Let $E$ denote the composite $\V$-functor
$$\U' \xrightarrow{D_{r\ell}} \Alg{\T'}(\C) \xrightarrow{M^*} \Alg{\T}(\C),$$
which is a $\U'$-algebra by \ref{para:carrier_func} and is given by $EK = D(M-,K)$ $(K \in \U')$.  The carrier of $E$ is $|E| = D(M-,I) = M^*(D_\ell)$, so by the universal property of the full algebra $\Gamma_A$ \pref{thm:univ_prop_gamma}, there is a unique morphism $(N,h):(\U',E) \rightarrow (\T^\perp_A,\Gamma_A)$ in $\Algs^\s(\Alg{\T}(\C))$ with $|h| = g^{-1}:|E| = M^*(D_\ell) \rightarrow A$.  Explicitly, $h:E \Rightarrow \Gamma_A N$ is an isomorphism in $[\U',[\T,\C]_\J]_\J$.  Writing
$$f_{JK} = (h^{-1}_K)_J \;:\; \Gamma^A(J,NK) = (\Gamma_A N K)J \longrightarrow (EK)J = D(MJ,K),$$
we find that $(M,N,f):(\T,\T^\perp_A,\Gamma^A) \rightarrow (\T',\U',D)$ is the unique cross-morphism of bifold algebras $(M,N,f):\RCom(\T,A) \rightarrow (\T',\U',D)$ such that $L(M,N,f) = (M,g)$, as needed.

This shows that $L$ has a left adjoint $\RCom$ with the needed properties.  By \ref{para:lr_face_functors} we know that $R = L^\op \varbigcirc$, so by defining $\LCom = \varbigcirc^\op \RCom^\op$, we find that $\LCom$ is a right adjoint to $R$ with the needed properties.
\end{proof}

\begin{defn}
We call the adjunctions in \eqref{eq:rcom_lcom} the \textbf{left face adjunction} and the \textbf{right face adjunction}, respectively.  Given an algebra $(\T,A)$ in $\C$, we call $\RCom(\T,A)$ the \textbf{right-commutant bifold algebra determined by $(\T,A)$} and $\LCom(\T,A)$ the \textbf{left-commutant bifold algebra determined by $(\T,A)$}.
\end{defn}

\begin{rem}\label{rem:rcom_cmt_w_carrier}
The functor $\RCom$ commutes with the `carrier' functors $\ca{\text{$-$}}$ valued in $\C_\si$.  Indeed, since $\RCom$ is a section of $L$, this follows from \eqref{eq:lr_commt_w_carrier}.
\end{rem}

The following shows that the assignment to each algebra its commutant is functorial with respect to strong morphisms of algebras and thus determines a contravariant functor that is `self-adjoint'; this result can be seen as an extension of the more restrictive self-adjoint functoriality of the commutant of a morphism of $\J$-theories that was established in \cite[8.6]{Lu:Cmt}.

\begin{thm}\label{thm:cmtnt_adjn}
There is a functor $(-)^\perp = \Com:\Algs^\s(\C) \longrightarrow \Algs^\s(\C)^\op$ that sends each algebra $(\T,A)$ to its commutant $(\T,A)^\perp$ and is left adjoint to its opposite, as in the diagram
\begin{equation}\label{eq:cmt_adj}
\xymatrix{
\Algs^\s(\C) \ar@{}[rr]|\top \ar@/_2ex/[rr]_{\Com} & & \Algs^\s(\C)^\op. \ar@/_2ex/[ll]_{\Com^\op}
}
\end{equation}
Explicitly, $\Com = R\RCom$ and $\Com^\op = L\LCom$ with the notation of \ref{thm:rcom_lcom}.
\end{thm}
\begin{proof}
Defining $\Com = R\RCom$, we find that $\Com$ is given on objects by $\Com(\T,A) = (\T,A)^\perp$ by \ref{thm:rcom_lcom}.  Using \eqref{eq:lcom_vs_rcom} and \ref{para:lr_face_functors}, we compute that
$$\Com^\op = R^\op \RCom^\op = R^\op \varbigcirc \LCom = L\LCom$$
since $\varbigcirc^{-1} = \varbigcirc^\op$, so by composing the adjunctions in \eqref{eq:rcom_lcom} we find that $\Com \dashv \Com^\op$.
\end{proof}

\begin{rem}\label{rem:com_cmt_w_car}
Since $\Com = R\RCom$, we deduce by \ref{rem:rcom_cmt_w_carrier} and \eqref{eq:lr_commt_w_carrier} that the following diagram commutes:
$$
\xymatrix{
\Algs^\s(\C) \ar[d]_{|-|} \ar[r]^\Com & \Algs^\s(\C)^\op \ar[d]^{|-|^\op}\\
\C_\si \ar[r]^i_\sim & \C_\si^\op
}
$$
\end{rem}

\begin{notn}\label{notn:unit_counit}
Let us write  $\varepsilon:\RCom L \Rightarrow 1$ to denote the counit of the adjunction $\RCom \dashv L$, and write $\eta:1 \Rightarrow \LCom R$ to denote the unit of the adjunction $R \dashv \LCom$.  In view of the proof of \ref{thm:rcom_lcom}, $\eta = \varbigcirc^\op \varepsilon^\op \varbigcirc$.  Writing $u:1 \Rightarrow \Com^\op\Com$ to denote the unit of the adjunction $\Com \dashv \Com^\op$, we find that
$$u = L\eta \RCom \;:\;1 = L \RCom \Longrightarrow L \LCom R \RCom = \Com^\op\Com$$
while the counit of this adjunction is $R\varepsilon\LCom$ but can be expressed also as
$$u^\op = R\varepsilon\LCom\;:\; \Com\Com^\op = R\RCom L\LCom \Longrightarrow R\LCom = 1,$$
by using \ref{para:lr_face_functors} and \eqref{eq:lcom_vs_rcom}, since $u = L\eta \RCom$ and $\eta = \varbigcirc^\op \varepsilon^\op \varbigcirc$.
\end{notn}

\begin{prop}\label{thm:bifold_alg_comm_r}
Given any bifold algebra $D$ on an object $C$ of $\C$, the morphism
$$R\varepsilon_D \;:\; RD \longrightarrow R\RCom LD = \Com LD\;\;\;\;\text{in $\Algs^\s(\C)$}$$
lies in the subcategory $\Algs(C)$ and is the (unique) morphism $[D_r|D_\ell]:D_r \rightarrow D_\ell^\perp$ in $\Algs(C)$ that witnesses that $D_\ell$ commutes with $D_r$ (\ref{defn:cmt}, \ref{thm:dl_commutes_w_dr}).  Similarly, 
$$L\eta_D \;:\; LD \longrightarrow L\LCom RD = \Com^\op RD\;\;\;\;\text{in $\Algs^\s(\C)$}$$
is the (unique) morphism $[D_\ell|D_r]:D_\ell \rightarrow D_r^\perp$ in $\Algs(C)$ that witnesses that $D_r$ commutes with $D_\ell$.
\end{prop}
\begin{proof}
$\RCom LD = \RCom D_\ell$ is a bifold algebra with left and right faces $D_\ell$ and $D_\ell^\perp$, respectively.  Hence $\varepsilon_D:\RCom D_\ell \rightarrow D$ is sent by $R:\BAlg^\sx(\C) \rightarrow \Algs^\s(\C)^\op$ to a morphism $R\varepsilon_D:D_r = RD \rightarrow R\RCom D_\ell = D_\ell^\perp$ in $\Algs^\s(\C)$.  But since the unit of the adjunction $\RCom \dashv L$ is an identity, it follows that $L\varepsilon_D$ is the identity morphism on $LD = D_\ell$, so since $L$ commutes with the `carrier' functors $\ca{\text{$-$}}$ valued in $\C_\si$ \pref{eq:lr_commt_w_carrier}, we deduce that the underlying morphism $|\varepsilon_D|$ in $\C_\si$ is the identity morphism on $C = |D|$.  Hence, we deduce by \eqref{eq:lr_commt_w_carrier} that $|R\varepsilon_D|$ is also the identity morphism on $C$.  Therefore $R\varepsilon_D:D_r \rightarrow D_\ell^\perp$ is a morphism in the subcategory $\Algs(C)$ of $\Algs^\s(\C)$, so $R\varepsilon_D = [D_r|D_\ell]$ by the uniqueness of $[D_r|D_\ell]$ \pref{defn:cmt}.  Next, $L\eta = L\varbigcirc^\op\varepsilon^\op \varbigcirc = R^\op\varepsilon^\op \varbigcirc$ by \ref{notn:unit_counit} and \ref{para:lr_face_functors}, so $L\eta D = R\varepsilon {D^\circ}$, recalling that $\varbigcirc(D) = D^\circ$ with the notation of \ref{notn:bif}.  The remaining claim follows, since $D^\circ_\ell = D_r$ and $D^\circ_r = D_\ell$.
\end{proof}

\begin{cor}\label{thm:unit_counit_com_adj}
For each algebra $A \in \ob\Algs^\s(\C)$, the commutant $A^\perp$ commutes with $A$, and the unit $u:1 \Rightarrow \Com^\op\Com$ of the adjunction $\Com \dashv \Com^\op$ consists of the morphisms
$$[A|A^\perp] \;:\; A \longrightarrow A^{\perp\perp}\;\;\;\;\;\;\;\;(A \in \ob\Algs^\s(\C))$$
in $\Algs(C)$ (and hence in $\Algs^\s(\C)$) that witness that $A^\perp$ commutes with $A$.  The counit of $\Com \dashv \Com^\op$ is $u^\op$ (by \ref{notn:unit_counit}) and so consists of these same morphisms.
\end{cor}
\begin{proof}
In view of \ref{notn:unit_counit}, this result is obtained by applying \ref{thm:bifold_alg_comm_r} to the bifold algebra $D = \RCom A$, since $D_r = A^\perp$ and $D_\ell = A$.
\end{proof}

\begin{cor}\label{eq:cmt_adjn_over_c}
For each object $C$ of $\C$, the functor $\Com$ in \ref{thm:cmtnt_adjn} restricts to a functor $\Com_C:\Algs(C) \rightarrow \Algs(C)^\op$.  Moreover, the adjunction \eqref{eq:cmt_adj} restricts to an adjunction $\Com_C \dashv \Com_C^\op:\Algs(C)^\op \rightarrow \Algs(C)$.
\end{cor}
\begin{proof}
By \ref{rem:com_cmt_w_car}, the functors $\Com$ and $\Com^\op$ restrict as needed, and we deduce by \ref{thm:unit_counit_com_adj} that the unit and counit of $\Com \dashv \Com^\op$ restrict as needed also.
\end{proof}

\begin{cor}\label{thm:comm_symm}
Let $A$ and $B$ be algebras on an object $C$ of $\C$.  Then $A$ commutes with $B$ if and only if $B$ commutes with $A$.
\end{cor}
\begin{proof}
By \ref{eq:cmt_adjn_over_c}, $\Algs(C)(B,A^\perp) \cong\Algs(C)(A,B^\perp)$, and the result follows.
\end{proof}

\section{Left- and right-commutant bifold algebras}\label{sec:lrcmt_bif}

One of the benefits of Theorem \ref{thm:rcom_lcom} is that it enables powerful methods for working with bifold algebras in which one face is isomorphic to the commutant of the other, as in the following definition:

\begin{defn}
A \textbf{right-commutant bifold algebra} is a bifold algebra $D$ in $\C$ such that $D_r \cong D_\ell^\perp$ in $\Algs(C)$ where $C = |D|$.  In other words, a bifold algebra on $C$ is right-commutant if its right face is isomorphic, in $\Algs(C)$, to the commutant of its left face.  Similarly, a \textbf{left-commutant bifold algebra} is a bifold algebra $D$ such that $D_\ell \cong D_r^\perp$ in $\Algs(C)$.  Let $\RComBAlg^\sx(\C)$ and $\LComBAlg^\sx(\C)$ denote the full subcategories of $\BAlg^\sx(\C)$  consisting of the right-commutant bifold algebras and the left-commutant bifold algebras, respectively.
\end{defn}

Using the notations of \ref{defn:cmt}, \ref{thm:dl_commutes_w_dr}, \ref{thm:rcom_lcom}, \ref{notn:unit_counit}, we now establish several equivalent characterizations of right-commutant bifold algebras:

\begin{prop}\label{thm:charns_rcom_balg}
Given a bifold algebra $D$ on an object $C$ of $\C$, the following are equivalent: (1) $D$ is a right-commutant bifold algebra; (2) $D \cong \RCom A$ in $\BAlg^\sx(\C)$ for some algebra $A$; (3) the counit morphism $\varepsilon_D:\RCom L D \rightarrow D$ is an isomorphism in $\BAlg^\sx(\C)$; (4) the unique morphism $[D_r|D_\ell]:D_r \rightarrow D_\ell^\perp$ in $\Algs(C)$ is an isomorphism; (5) the morphism of $\J$-theories $[D_r|D_\ell]:\U \rightarrow \T^\perp_{D_\ell}$ is an isomorphism, where we write $\T$, $\U$ to denote the pair of $\J$-theories underlying $D$.
\end{prop}
\begin{proof}
Clearly $(4)\Rightarrow(1)$, while $(1)\Rightarrow(4)$ by the uniqueness of $[D_r|D_\ell]$.  Also, $(2)\Leftrightarrow(3)$ since the unit of the adjunction $\RCom \dashv L$ is an identity.  Next $(3)\Rightarrow(4)$ by \ref{thm:bifold_alg_comm_r}, and clearly $(4)\Rightarrow(5)$.  Now supposing (5), it suffices to show (3).  Let us write $D$ as $(\T,\U,D)$.  Since the forgetful functors $P$ and $Q$ in \eqref{eq:2sided_fibr_balg} are jointly conservative \pref{rem:pq_jcons}, it suffices to show that the underlying morphisms of $\J$-theories $P\varepsilon_D$ and $Q\varepsilon_D$ are invertible.  Since the unit of $\RCom \dashv L$ is an identity, it follows that $L\varepsilon_D$ is the identity on $L(\T,\U,D) = (\T,D_\ell)$, so $P\varepsilon_D = 1_\T:\T \rightarrow \T$, while by \ref{thm:bifold_alg_comm_r} we find that $Q\varepsilon_D = [D_r|D_\ell]:\U \rightarrow \T^\perp_{D_\ell}$, which is invertible by (5).  Hence (3) holds. 
\end{proof}

\begin{rem}\label{rem:charns_lcom_balgs}
A bifold algebra $D$ is left-commutant iff $D^\circ$ is right-commutant, so by applying \ref{thm:charns_rcom_balg} to $D^\circ$ and using \ref{para:lr_face_functors}, \eqref{eq:lcom_vs_rcom}, and \eqref{notn:unit_counit}, we obtain analogous characterizations of left-commutant bifold algebras $D$ by exchanging the roles of $D_\ell$ and $D_r$ and replacing $\RCom$ with $\LCom$.  In particular, $D$ is left-commutant iff the unit morphism $\eta_D:D \rightarrow \LCom RD$ is an isomorphism in $\BAlg^\sx(\C)$.
\end{rem}

\begin{thm}\label{thm:rcmt_corefl}
The category of right-commutant bifold algebras $\RComBAlg^\sx(\C)$ is a replete, coreflective subcategory of $\BAlg^\sx(\C)$, and the category of left-commutant bifold algebras $\LComBAlg^\sx(\C)$ is a replete, reflective subcategory of $\BAlg^\sx(\C)$.  There are equivalences of categories
$$\RComBAlg^\sx(\C) \;\simeq\; \Algs^\s(\C),\;\;\;\;\LComBAlg^\sx(\C)\;\simeq\;\Algs^\s(\C)^\op$$
obtained as restrictions of the adjunctions $\RCom \dashv L$ and $R \dashv \LCom$ of \ref{thm:rcom_lcom}.
\end{thm}
\begin{proof}
This follows from Theorem \ref{thm:rcom_lcom}, Proposition \ref{thm:charns_rcom_balg}, and Remark \ref{rem:charns_lcom_balgs}.
\end{proof}

The following further characterization of right- and left-commutant bifold algebras makes no reference to the notion of commutant (nor to $\RCom$ or $\LCom$):

\begin{thm}
Let $(\T,\U,D)$ be a bifold algebra in $\C$.  Then $D$ is right-commutant if and only if its transpose $D_{r\ell}:\U \rightarrow [\T,\C]_\J$ is fully faithful.  Similarly, $D$ is left-commutant if and only if its transpose $D_{\ell r}:\T \rightarrow [\U,\C]_\J$ is fully faithful.
\end{thm}
\begin{proof}
Let $M = [D_r|D_\ell]:\U \rightarrow \T^\perp_{D_\ell}$.  Then $D$ is right-commutant if and only if $M$ is an isomorphism \pref{thm:charns_rcom_balg}.  But $M = [D_{r\ell}]$ by \ref{thm:dl_commutes_w_dr}, and hence $\Gamma_{D_\ell} M \cong D_{r\ell}$ by \ref{defn:ind_mor}, so since $\Gamma_{D_\ell}$ is fully faithful we find that $D_{r\ell}$ is fully faithful if and only if $M$ is fully faithful.  But $M$ is a morphism of $\J$-theories, so $M$ is fully faithful if and only if $M$ is an isomorphism.
\end{proof}

\section{Commutant bifold algebras and saturated algebras}\label{sec:cmt_bifold_algs}

\begin{prop}\label{thm:com_adj_idem}
The adjunction $\Com \dashv \Com^\op:\Algs^\s(\C)^\op \rightarrow \Algs^\s(\C)$ in \ref{thm:cmtnt_adjn} is idempotent (in the sense of \cite[2.8]{MacSto}), as is the adjunction $\Com_C \dashv \Com_C^\op$ in \ref{eq:cmt_adjn_over_c} for each object $C$ of $\C$.  Consequently, the monad $(-)^{\perp\perp} = \Com^\op\Com$ on $\Algs^\s(\C)$ is idempotent, as is the monad $(-)^{\perp\perp} = \Com_C^\op\Com_C$ on $\Algs(C)$.
\end{prop}
\begin{proof}
Since $\Com_C \dashv \Com_C^\op$ is obtained by restricting $\Com \dashv \Com^\op$, it suffices to show that the latter adjunction is idempotent.  By \ref{notn:unit_counit}, the counit of this adjunction is $u^\op:\Com\Com^\op \Rightarrow 1$, where $u:1 \Rightarrow \Com^\op\Com$ is the unit.  By adjointness, $u^\op\Com \cdot \Com u = 1_{\Com}$, so it suffices to show that $\Com u \cdot u^\op \Com = 1_{\Com\Com^\op\Com}$.  But for each algebra $A$ on an object $C$ of $\C$, $(\Com u \cdot u^\op \Com)_A$ is a morphism $A^{\perp\perp\perp} \rightarrow A^{\perp\perp\perp}$ in $\Algs(C)$ (by \ref{eq:cmt_adjn_over_c}), which must be an identity by the uniqueness in \ref{thm:equiv_charns_b_pipe_a}.
\end{proof}

The following definition provides a convenient new way of describing $\J$-theories $\T$ that are \textit{saturated} with respect to a $\T$-algebra $A$ in the sense of \cite[3.3.5]{Lu:FDistn}:

\begin{defn}
An algebra $A$ on an object $C$ of $\C$ is \textbf{saturated} if and only if the unit morphism $u_A:A \rightarrow \Com^\op\Com A = A^{\perp\perp}$ is an isomorphism in $\Algs^\s(\C)$ (equivalently, in $\Algs(C)$).  Let us write $\SatAlgs^\s(\C)$ to denote the full subcategory of $\Algs^\s(\C)$ consisting of the saturated algebras, and for each object $C$ of $\C$, write $\SatAlgs(C)$ for the full subcategory of $\Algs(C)$ consisting of saturated algebras on $C$.
\end{defn}

\begin{cor}\label{thm:charns_sat_algs}
Given an algebra $A$ on an object $C$ of $\C$, the following are equivalent: (1) $A$ is saturated, (2) $A \cong A^{\perp\perp}$ in $\Algs^\s(\C)$ (3) $A \cong B^\perp$ in $\Algs^\s(\C)$ for some algebra $B$, (4) $A \cong A^{\perp\perp}$ in $\Algs(C)$, (5) $A \cong B^\perp$ in $\Algs(C)$ for some algebra $B$ on $C$.
\end{cor}

\begin{rem}\label{rem:satalg_preordered_class}
For each object $C$ of $\C$, the category $\SatAlgs(C)$ is a preordered class, in view of \ref{thm:charns_sat_algs} and the uniqueness in \ref{thm:equiv_charns_b_pipe_a}.
\end{rem}

\begin{cor}\label{thm:cat_sat_algs}
The category of saturated algebras $\SatAlgs^\s(\C)$ is a replete, reflective subcategory of $\Algs^\s(\C)$, with reflector given on objects by $A \mapsto A^{\perp\perp}$.  The adjunction $\Com \dashv \Com^\op$ restricts an equivalence
$$\SatAlgs^\s(\C) \simeq \SatAlgs^\s(\C)^\op$$
given on objects by $A \mapsto A^\perp$.
\end{cor}

\begin{rem}\label{rem:sat_alg_faithful}
Every saturated algebra $(\T,A)$ is \textbf{faithful}, meaning that the $\V$-functor $A:\T \rightarrow \C$ is faithful.  Indeed, this follows from the fact that $B^\perp:\U^\perp_B \rightarrow \C$ is faithful for every algebra $(\U,B)$, by \ref{para:cmt}.
\end{rem}

\begin{defn}
A \textbf{commutant bifold algebra} is a bifold algebra $D$ that is both left-commutant and right-commutant, meaning that $D_\ell \cong D_r^\perp$ and $D_r \cong D_\ell^\perp$ in $\Algs(C)$, where $C = |D|$.  We write $\ComBAlg^\sx(\C)$ to denote the full subcategory of $\BAlg^\sx(\C)$ consisting of the commutant bifold algebras.
\end{defn}

\begin{prop}\label{thm:charn_comm_balg}
Given a bifold algebra $D$ on an object $C$ of $\C$, the following are equivalent: (1) $D$ is commutant, (2) $D$ is left-commutant and $D_r$ is saturated, (3) $D$ is right-commutant and $D_\ell$ is saturated.
\end{prop}
\begin{proof}
Clearly $(1)\Rightarrow(2)$ by \ref{thm:charns_sat_algs}, while if (2) holds then $D_\ell \cong D_r^\perp$ in $\Algs(C)$, so by \ref{eq:cmt_adjn_over_c} we deduce that $D_\ell^\perp \cong D_r^{\perp\perp} \cong D_r$ in $\Algs(C)$, since $D_r$ is saturated, so (1) holds.  By considering $D^\circ$ the equivalence of (1) and (3) now follows.
\end{proof}

\begin{cor}\label{thm:a_sat_iff_rcom_is_cbalg}
Given an algebra $A$ in $\C$, the following are equivalent: (1) $A$ is saturated, (2) $\RCom A$ is a commutant bifold algebra, (3) $\LCom A$ is a commutant bifold algebra.
\end{cor}

\begin{thm}\label{thm:cmt_bifold_algs}
The equivalence $\RComBAlg^\sx(\C) \simeq \Algs^\s(\C)$ in \ref{thm:rcmt_corefl} restricts to an equivalence
$$\ComBAlg^\sx(\C) \simeq \SatAlgs^\s(\C)$$
between the category of commutant bifold algebras with strong cross-morphisms and the category of saturated algebras with strong morphisms.  $\ComBAlg^\sx(\C)$ is a replete, \mbox{reflective} subcategory of $\RComBAlg^\sx(\C)$.  Similarly, the equivalence $\LComBAlg^\sx(\C) \simeq \Algs^\s(\C)^\op$ restricts to an equivalence
$$\ComBAlg^\sx(\C) \simeq \SatAlgs^\s(\C)^\op,$$
and $\ComBAlg^\sx(\C)$ is a replete, coreflective subcategory of $\LComBAlg^\sx(\C)$.
\end{thm}
\begin{proof}
The first claim follows from \ref{thm:charn_comm_balg} and \ref{thm:a_sat_iff_rcom_is_cbalg}.  Consequently, the reflectivity of the full subcategory $\ComBAlg^\sx(\C)$ of $\RComBAlg^\sx(\C)$ follows from \ref{thm:cat_sat_algs}, while its repleteness follows immediately from \ref{thm:rcmt_corefl}.  The remaining claims can be established similarly.
\end{proof}

\section{Commuting pairs of algebras versus bifold algebras}\label{sec:cmt_algs}

\begin{defn}\label{defn:pairs_over_various_theories}
Equipping the category $\Algs^\s(\C)$ with the functor $\ca{\text{$-$}}:\Algs^\s(\C) \rightarrow \C_\si$, let us write $\Pair^\s(\C)$ for the fibre product of $\Algs^\s(\C)$ with itself over $\C_\si$ in $\CAT$.  Hence, the objects of $\Pair^\s(\C)$ are pairs of algebras $(A,B)$ with $|A| = |B|$, which we call \textbf{algebra pairs}.  A morphism in $\Pair^\s(\C)$ is a pair $(f,g)$ consisting of morphisms in $\Algs^\s(\C)$ with $|f| = |g|$.  Let us write $\CPair^\s(\C)$ to denote the full subcategory of $\Pair^\s(\C)$ consisting of \textbf{commuting algebra pairs}, i.e., algebra pairs $(A,B)$ such that $A$ commutes with $B$.
\end{defn}

\begin{prop}\label{thm:comm_closed_under_incoming_morphs}
If $(f,g):(X,Y) \rightarrow (A,B)$ is a morphism in $\Pair^\s(\C)$ and $(A,B)$ is a commuting algebra pair, then $(X,Y)$ is a commuting algebra pair.  Consequently, the full subcategory $\CPair^\s(\C) \hookrightarrow \Pair^\s(\C)$ is replete.
\end{prop}
\begin{proof}
Let $C = |A| = |B|$ and $C' = |X| = |Y|$.  By applying the functor $(-)^\perp:\Algs^\s(\C) \rightarrow \Algs^\s(\C)^\op$ of \ref{thm:cmtnt_adjn} we obtain a morphism $f^\perp:A^\perp \rightarrow X^\perp$ in $\Algs^\s(\C)$, and by \ref{rem:com_cmt_w_car} we know that $|f^\perp| = |f|^{-1}$.   Also, since $A$ commutes with $B$, there is a unique morphism $[B|A]:B \rightarrow A^\perp$ in the subcategory $\Algs(C)$ of $\Algs^\s(\C)$, so we obtain a composite morphism
$$h = \bigl(Y \xrightarrow{g} B \xrightarrow{[B|A]} A^\perp \xrightarrow{f^\perp} X^\perp\bigr)$$
in $\Algs^\s(\C)$ with $|h| = |f|^{-1} \cdot 1_C \cdot |g| = 1_{C'}$ since $|f| = |g|$, so that $h:Y \rightarrow X^\perp$ is a morphism in $\Algs(C')$.  Therefore $X$ commutes with $Y$, by \ref{defn:cmt}.
\end{proof}

\begin{defn}\label{para:cat_ts_algs}
Let $\T$ and $\U$ be $\J$-theories.  Equipping both $\Alg{\T}(\C)$ and $\Alg{\U}(\C)$ with their `carrier' $\V$-functors valued in $\C$, let us write $\AlgPair{\T}{\U}(\C)$ for the fibre product $\A \times_\C \B$ of $\A = \Alg{\T}(\C)$ and $\B = \Alg{\U}(\C)$ over $\C$ in $\VCAT$, which exists as we assume $\V_0$ has pullbacks \pref{para:given_data}.  The objects $(A,B)$ of $\AlgPair{\T}{\U}(\C)$ consist of $\T$- and $\U$-algebras $A$ and $B$, respectively, with $|A| = |B|$.  Let
$$\CAlgPair{\T}{\U}(\C)$$
denote the full sub-$\V$-category of $\AlgPair{\T}{\U}(\C)$ consisting of the commuting $\T$-$\U$-algebra pairs.  For each object $C$ of $\C$, let us write $\CAlgPair{\T}{\U}(C)$ to denote the fibre over $C$ of the `carrier' functor $\CAlgPair{\T}{\U}(\C)_0 \rightarrow \C_0$.  Therefore $\CAlgPair{\T}{\U}(C)$ may be identified with the full subcategory of $\Alg{\T}(C) \times \Alg{\U}(C)$ consisting of commuting $\T$-$\U$-algebra pairs on $C$.  By \ref{thm:comm_closed_under_incoming_morphs} we obtain the following:
\end{defn}

\begin{cor}\label{thm:calg_pair_repl}
The full subcategory $\CAlgPair{\T}{\U}(\C)$ of $\AlgPair{\T}{\U}(\C)$ is replete.
\end{cor}

\begin{cor}\label{thm:cmtn_via_normaln}
Let $(\T,A)$ and $(\U,B)$ be algebras on an object $C$ of $\C$.  The following are equivalent: (1) $A$ commutes with $B$, (2) $A^\nml$ commutes with $B^\nml$, (3) $[A]$ commutes with $[B]$ in the sense of \ref{para:cmtn}.
\end{cor}
\begin{proof}
Since $(A,B) \cong (A^\nml,B^\nml)$ in $\AlgPair{\T}{\U}(\C)$, the equivalence of (1) and (2) follows from \ref{thm:calg_pair_repl}.  In the case where $A$ and $B$ are normal, we noted in \ref{para:cmtn} that $[A]$ commutes with $[B]$ if and only if $[B]:\U \rightarrow \C_C$ factors through $[A^\perp]:\T^\perp_A \hookrightarrow \C_C$, which by \ref{defn:cmt} is equivalent to the statement that $A$ commutes with $B$.  Without the assumption of normality, we can apply this to the pair $A^\nml$, $B^\nml$ to deduce that (2) is equivalent to (3), using the fact that $[A^\nml] = [A]$ and $[B^\nml] = [B]$ by \ref{defn:ind_mor}.
\end{proof}

\begin{thm}\label{thm:commutation_and_tu_algs}
Let $(\T,A)$ and $(\U,B)$ be algebras on an object $C$ of $\C$.  The following are equivalent: (1) $A$ commutes with $B$; (2) there exists a $(\T,\U)$-algebra $D$ on $C$ such that $D_\ell \cong A$ in $\Alg{\T}(C)$ and $D_r \cong B$ in $\Alg{\U}(C)$; (3) there exists a $(\T,\U)$-algebra $D:\T \otimes \U \rightarrow \C$ with $(D_\ell,D_r) \cong (A,B)$ in the pullback $\Alg{\T}(\C) \times_\C \Alg{\U}(\C) = \AlgPair{\T}{\U}(\C)$.  If $B$ is normal, then these conditions are equivalent to the following: (4) There exists a $(\T,\U)$-algebra $D$ in $\C$ with $D_\ell = A$ and $D_r = B$.
\end{thm}
\begin{proof}
We shall show moreover that (1)-(3) are equivalent to the following condition: \textit{(5) There exists a $\U$-algebra $P:\U \rightarrow \Alg{\T}(\C)$ on $A$ such that $G^\T P \cong B$ in $\Alg{\U}(C)$.} Clearly $(2)\Rightarrow(3)$, and the implication $(3)\Rightarrow(1)$ follows from \ref{thm:dl_commutes_w_dr} and  \ref{thm:calg_pair_repl}.  Also $(5)\Rightarrow(2)$, because if $P$ is as in (5) then by letting $D:\T \otimes \U \rightarrow \C$ be the transpose of $P$ \eqref{eq:equiv_descns_tualgs} we find that (2) holds.  Next we prove $(1)\Rightarrow(5)$. Suppose (1) holds. By \ref{para:cmt}, $\Gamma_A:\T^\perp_A \rightarrow \Alg{\T}(\C)$ is a $\T^\perp_A$-algebra on $A$, so by composing with the morphism of $\J$-theories $[B|A]:\U \rightarrow \T^\perp_A$ of \eqref{eq:b_pipe_a} we obtain a $\U$-algebra $\Gamma_A [B|A]:\U \rightarrow \Alg{\T}(\C)$ on $A$.  By the definition of $A^\perp$ in \eqref{eq:a_lowerperp} we compute that $G^\T \Gamma_A [B|A] = A^\perp [B|A] \cong B$ in $\Alg{\U}(C)$ by \ref{defn:cmt}, so (5) holds.  Note also that in the case where $B$ is normal, we have moreover that $G^\T \Gamma_A [B|A] = A^\perp [B|A] = B$ by \ref{defn:cmt}, so by letting $D:\T\otimes \U \rightarrow \C$ be the transpose of $\Gamma_A[B|A]:\U \rightarrow \Alg{\T}(\C)$ we deduce that (4) holds.

This shows that (1)-(3) and (5) are equivalent, and that if $B$ is normal then $(1)\Rightarrow(4)$.  Clearly $(4)\Rightarrow(2)$.  
\end{proof}

\begin{para}\label{para:func_cpair}
Given morphisms of $\J$-theories $M:\T \rightarrow \T'$ and $N:\U \rightarrow \U'$, if $(A,B)$ is a commuting $\T'$-$\U'$-algebra pair in $\C$, then it follows from \ref{thm:commutation_and_tu_algs} together with the functoriality of $\Alg{(-,\blanktwo)}(\C)$ in \ref{defn:tu_alg} that $(AM,BN)$ is a commuting $\T$-$\U$-algebra pair.  Thus we obtain a $\V$-functor $(M,N)^*:\CAlgPair{\T'}{\U'}(\C) \rightarrow \CAlgPair{\T}{\U}(\C)$ that is given on objects by $(M,N)^*(A,B) = (AM,BN)$.  In this way, we obtain a functor $\CAlgPair{-}{\blanktwo}(\C):\ThJ^\op \times \ThJ^\op \rightarrow \VCAT$.
\end{para}

Our next objective is to show that $\Alg{(\T,\U)}(\C) \simeq \CAlgPair{\T}{\U}(\C)$, and for this we shall need the following general material:

\begin{defn}\label{def:lcart}
We say that a commutative square
$$
\xymatrix{
\D \ar[d]_F \ar[r]^G & \B \ar[d]^Q\\
\A \ar[r]_P & \E
}
$$
in $\VCAT$ is \textbf{locally cartesian} if the induced $\V$-functor $(F,G):\D \rightarrow \A \times_{\E} \B$ is fully faithful, where $\A \times_{\E} \B$ is the pullback of $P,Q$ in $\VCAT$---equivalently, if for every pair of objects $X,Y$ in $\D$ the associated commutative square in $\V$ (consisting of $F_{XY},G_{XY},P_{FX,FY},Q_{GX,GY}$) is a pullback.
\end{defn}

One source of locally cartesian squares is the following:

\begin{lem}\label{prop:source_of_loc_cart_sq}
Let $F:\A \rightarrow \A'$ and $G:\B \rightarrow \B'$ be identity-on-objects $\V$-functors, and let $\E$ be a $\V$-category.  Then the following is a locally cartesian square in $\eCAT{\V'}$:
$$
\xymatrix{
[\A'\otimes\B',\E] \ar[d]_{[1\otimes G,\E]} \ar[r]^(.51){[F\otimes 1,\E]} & [\A\otimes\B',\E] \ar[d]^{[1\otimes G,\E]}\\
[\A'\otimes\B,\E] \ar[r]_{[F\otimes 1,\E]} & [\A\otimes\B,\E]
}
$$
\end{lem}
\begin{proof}
The proof is a straightforward exercise using the following observation.  Let $\X = \ob \A = \ob\A'$ and $\Y = \ob\B = \ob\B'$, let $H,K$ be objects of $[\A'\otimes\B',\E]$, and let $V$ be an object of $\V$.  Since $F$ and $G$ are identity-on-objects, a family $f_{XY}:V \rightarrow \E(H(X,Y),K(X,Y))$ $(X \in \X, Y \in \Y)$ is (extraordinarily) $\V$-natural in $X \in \A'$ and $Y \in \B'$ if and only if it is $\V$-natural in each variable separately, if and only if both of the following hold: (1) $f_{XY}:V \rightarrow \E(H(X,GY),K(X,GY))$ is $\V$-natural in $X \in \A'$ and $Y \in \B$ and (2) $f_{XY}:V \rightarrow \E(H(FX,Y),K(FX,Y))$ is $\V$-natural in $X \in \A$ and $Y \in \B'$.
\end{proof}

\begin{para}\label{notn:incl}
Viewing the identity $\V$-functor $1:\T \otimes \U \rightarrow \T \otimes \U$ as a bifunctor and fixing its second argument as the unit object $I \in \ob\J = \ob\U$, we obtain a $\V$-functor $1(-,I)$ \cite[(1.20)]{Ke:Ba} that we write as $(-,I) :\T \rightarrow \T \otimes \U$, so that $(-,I)$ is given on objects by $J \mapsto (J,I)$.  We define $(I,-):\U \rightarrow \T \otimes \U$ analogously.
\end{para}

\begin{thm}\label{thm:tsalgs_equiv_talgs_in_salg}
There is an equivalence of $\V$-categories
\begin{equation}\label{eq:tualg_tualgpair_equiv}\Alg{(\T,\U)}(\C) \;\;\simeq\;\; \CAlgPair{\T}{\U}(\C)\end{equation}
that sends each $(\T,\U)$-algebra $D$ to the commuting $\T$-$\U$-algebra pair $(D_\ell,D_r)$ (\ref{notn:bif}, \ref{thm:dl_commutes_w_dr}).  In particular, $\Alg{(\T,\U)}(\C)$ necessarily exists as a $\V$-category.  For each object $C$ of $\C$, the equivalence \eqref{eq:tualg_tualgpair_equiv} restricts to an equivalence
\begin{equation}\label{eq:tualgc_tualgpairc_equiv}\Alg{(\T,\U)}(C) \;\simeq\; \CAlgPair{\T}{\U}(C)\;.\end{equation}
\end{thm}
\begin{proof}
We shall establish an equivalence of $\V'$-categories of the form \eqref{eq:tualg_tualgpair_equiv}, from which it then follows that $\Alg{(\T,\U)}(\C)$ exists as a $\V$-category, since $\CAlgPair{\T}{\U}(\C)$ is a $\V$-category.  For brevity, let us write $\A = [\T,\C]_\J$, $\B = [\U,\C]_\J$, and $\D = [\T,\U;\C]_\J$ with the notation of \eqref{eq:jcot_pres_func_cat}, \ref{notn:pres_jcots_in_each_var}.  There are $\V'$-functors $(-,I)^*:\D \rightarrow \A$ and $(I,-)^*:\D \rightarrow \B$ given by pre-composition with the $\V$-functors $(-,I)$ and $(I,-)$ defined in \ref{notn:incl}.  On objects $(-,I)^*(D) = D(-,I) = D_\ell$ and $(I,-)^*(D) = D(I,-) = D_r$.  The square
\begin{equation}\label{eq:inducing_square}
\xymatrix{
\D \ar[d]_{(-,I)^*}\ar[r]^(.5){(I,-)^*} & \B \ar[d]^{\Ev_I}\\
\A \ar[r]_{\Ev_I} & \C
}
\end{equation}
commutes, where we write $\Ev_I$ for $G^\T$ and $G^\U$, so we obtain an induced $\V'$-functor $\mathsf{Faces} := \bigl((-,I)^*,(I,-)^*\bigl)\;:\;\D \rightarrow \A \times_\C \B$, recalling that $\A \times_\C \B = \AlgPair{\T}{\U}(\C)$.  Given any object $D$ of $\D$, we deduce by \ref{thm:dl_commutes_w_dr} that $\mathsf{Faces} D = (D_\ell,D_r)$ is a commuting $\T$-$\U$-algebra pair.  Therefore $\mathsf{Faces}$ factors through $\CAlgPair{\T}{\U}(\C) \hookrightarrow \A \times_\C \B$ by way of a unique $\V'$-functor $\mathsf{Faces}':\D \rightarrow \CAlgPair{\T}{\U}(\C)$, which commutes with the `carrier' $\V'$-functors valued in $\C$ and so restricts to a functor $\mathsf{Faces}'_C:\Alg{(\T,\U)}(C) \rightarrow \CAlgPair{\T}{\U}(C)$ for each object $C$ of $\C$.  By \ref{thm:commutation_and_tu_algs}, each $\mathsf{Faces}'_C$ is essentially surjective on objects, so $\mathsf{Faces}'$ is also essentially surjective on objects.  Hence it suffices to show that $\mathsf{Faces}$ is fully faithful.  In view of \ref{def:lcart}, it suffices to show that the square \eqref{eq:inducing_square} is locally cartesian.  We have a commutative cube
\begin{equation}\label{eq:cube}
\xymatrix@!0@C=17ex @R=8ex{
[\T,\U;\C]_\J \ar@{=}[dd] \ar[dr]|{[1,\upsilon;\C]_\J} \ar[rr]^{[\tau,1;\C]_\J} & & [\J^\op,\U;\C]_\J \ar@{..>}[dd]^(.7){(I,-)^*} \ar[dr]^{[1,\upsilon;\C]_\J} &\\
& [\T,\J^\op;\C]_\J \ar[dd]^(.7){(-,I)^*} \ar[rr]^(.3){[\tau,1;\C]_\J} & & [\J^\op,\J^\op;\C]_\J \ar[dd]^{\Ev_{(I,I)}}\\
[\T,\U;\C]_\J \ar[dr]_{(-,I)^*} \ar@{..>}[rr]^(.7){(I,-)^*} & & [\U,\C]_\J \ar[dr]^{\Ev_I} & \\
& [\T,\C]_\J \ar[rr]_{\Ev_I} & & \C
}
\end{equation}
in $\eCAT{\V'}$, in which each $\V'$-functor is given either by evaluation at $I$ in certain arguments or by pre-composition, where $\tau:\J^\op \rightarrow \T$ and $\upsilon:\J^\op \rightarrow \U$ are the unique morphisms in $\ThJ$ \pref{para:jth} and we employ the notation of \ref{defn:tu_alg}.  The base of this cube is the square \eqref{eq:inducing_square} that induces $\mathsf{Faces}$, and we claim that the vertical arrows in this cube are all equivalences.  Indeed, the $\V'$-functor $(I,-)^*:[\J^\op,\U;\C]_\J \rightarrow [\U,\C]_\J$ is a composite 
$$[\J^\op,\U;\C]_\J \xrightarrow{\sim} [\J^\op,[\U,\C]_\J]_\J \xrightarrow{\Ev_I} [\U,\C]_\J$$
whose second factor $\Ev_I$ is an equivalence by \ref{para:jop_algs}.  Similarly, $(-,I)^*:[\T,\J^\op;\C] \rightarrow [\T,\C]_\J$ is an equivalence, and the vertical arrow $\Ev_{(I,I)}$ is a composite
$$[\J^\op,\J^\op;\C]_\J \xrightarrow{\sim} [\J^\op,[\J^\op,\C]_\J]_\J \xrightarrow{\Ev_I} [\J^\op,\C]_\J \xrightarrow{\Ev_I} \C,$$
which is an equivalence, by two applications of \ref{para:jop_algs}.

Hence it suffices to show that the top face of the cube \eqref{eq:cube} is locally cartesian.  But the top face is induced by the $\V$-functors $\tau \otimes 1:\J^\op \otimes \U \rightarrow \T \otimes \U$ and $1 \otimes \upsilon:\T \otimes \J^\op \rightarrow \T \otimes \U$, which are identity-on-objects, so this follows from \ref{prop:source_of_loc_cart_sq}.
\end{proof}

\begin{rem}\label{rem:nat_equiv}
The equivalence $\Alg{(\T,\U)}(\C) \rightarrow \CAlgPair{\T}{\U}(\C)$ in \eqref{eq:tualg_tualgpair_equiv} is (strictly) natural in $\T,\U \in \ThJ$, as is \eqref{eq:tualgc_tualgpairc_equiv}.
\end{rem}

The following shows that the $(\T,\U)$-algebras obtained in \ref{thm:commutation_and_tu_algs}(2,3) are unique up to two respective notions of isomorphism:

\begin{cor}\label{thm:assoc_tu_alg_on_c}
Let $(A,B)$ be a commuting $\T$-$\U$-algebra pair on an object $C$ of $\C$.  
\begin{enumerate}
\item Up to $\V$-natural isomorphism, there is a unique $(\T,\U)$-algebra $D:\T \otimes \U \rightarrow \C$ with $(D_\ell,D_r) \cong (A,B)$ in the pullback $\Alg{\T}(\C) \times_\C \Alg{\U}(\C) = \AlgPair{\T}{\U}(\C)$.
\item Up to isomorphism in $\Alg{(\T,\U)}(C)$, there is a unique $(\T,\U)$-algebra $D$ on $C$ with $D_\ell \cong A$ in $\Alg{\T}(C)$ and $D_r \cong B$ in $\Alg{\U}(C)$.
\end{enumerate}
\end{cor}
\begin{proof}
This follows from \ref{thm:tsalgs_equiv_talgs_in_salg}, since $\CAlgPair{\T}{\U}(C)$ is a full subcategory of the product $\Alg{\T}(C) \times \Alg{\U}(C)$ \pref{para:cat_ts_algs}
\end{proof}

\begin{notn}\label{notn:tualg_ind_by_tucpair}
If $(A,B)$ is a commuting $\T$-$\U$-algebra pair on an object $C$ of $\C$, then we write $\langle A,B\rangle:\T \otimes \U \rightarrow \C$ for the $(\T,\U)$-algebra $D$ in \ref{thm:assoc_tu_alg_on_c}(2).  Thus $\langle A,B\rangle$ is defined up to isomorphism in $\Alg{(\T,\U)}(C)$.  If $B$ is normal, then in view of \ref{thm:commutation_and_tu_algs} we can and will choose $\langle A,B \rangle$ in such a way that $\langle A,B\rangle(-,I) = A$ and $\langle A,B \rangle(I,-) = B$.
\end{notn}

By \ref{thm:assoc_tu_alg_on_c}, the $(\T,\U)$-algebra $D = \langle A,B\rangle$ may also be characterized uniquely up to \textit{$\V$-natural isomorphism} by the property that $(D_\ell,D_r) \cong (A,B)$ in $\AlgPair{\T}{\U}(\C)$.

\begin{para}\label{thm:funcs_ind_by_cmt_mor_th}
Let $\sS$ be a $\J$-theory, let $M:\T \rightarrow \sS$ and $N:\U \rightarrow \sS$ be morphisms of $\J$-theories, and suppose that $M$ commutes with $N$. Then $(M,N)$ is a commuting pair of normal algebras on $I$ (\ref{para:mor_jth}, \ref{defn:cmt}), so by \ref{notn:tualg_ind_by_tucpair} we obtain a $(\T,\U)$-algebra $\langle M,N\rangle:\T \otimes \U \rightarrow \sS$ on $I$ with left and right faces $M$ and $N$, respectively.  For each $\sS$-algebra $A:\sS \rightarrow \C$ the composite $A\langle M,N\rangle:\T \otimes \U \rightarrow \C$ is a $(\T,\U)$-algebra on $|A|$ whose left and right faces are precisely $AM$ and $AN$, respectively, so $(AM,AN)$ is a commuting $\T$-$\U$-algebra pair, and $A\langle M,N\rangle = \langle AM,AN\rangle$ with the notation of \ref{notn:tualg_ind_by_tucpair}.  Since $\langle M,N\rangle(-,I) = M$ and $\langle M,N\rangle(I,-) = N$, we obtain a commutative diagram
$$
\xymatrix{
&\Alg{\sS}(\C) \ar[dl]_{M^*} \ar[d]|{\langle M,N \rangle^*} \ar[dr]^{N^*} & \\
\Alg{\T}(\C)     &\Alg{(\T,\U)}(\C) \ar[l]^(.55){(-,I)^*}  \ar[r]_(.55){(I,-)^*} &\Alg{\U}(\C)
}
$$
in $\VCAT$, with the notation of \eqref{eq:inducing_square}, where $\langle M,N \rangle^*$ is given by pre-composing with $\langle M,N\rangle$.
\end{para}

Next we use Theorem \ref{thm:tsalgs_equiv_talgs_in_salg} to show that the category $\BAlg^\sx(\C)$ of bifold algebras and strong cross-morphisms is equivalent to a category of commuting algebra pairs $\CPair^\sx(\C)$ that we now define:

\begin{defn}
Let us regard both $\Algs^\s(\C)$ and its opposite $\Algs^\s(\C)^\op$ as objects of the slice $\CAT \slash \C_\si$ by way of the functor $\ca{\text{$-$}}:\Algs^\s(\C) \rightarrow \C_\si$ and the composite
$$\Algs^\s(\C)^\op \xrightarrow{|-|^\op} \C_\si^\op \xrightarrow{i^\op} \C_\si$$
with the notation of \eqref{eq:lr_commt_w_carrier}.  Let us write $\Pair^\sx(\C)$ for the product of $\Algs^\s(\C)$ and $\Algs^\s(\C)^\op$ in $\CAT \slash \C_\si$.  Hence, the objects of $\Pair^\sx(\C)$ are algebra pairs $(A,B)$ in $\C$, in the sense of \ref{defn:pairs_over_various_theories}.  A morphism $(f,g):(A,B) \rightarrow (X,Y)$ in $\Pair^\sx(\C)$ consists of morphisms $f:A \rightarrow X$ and $g:Y \rightarrow B$ in $\Algs^\s(\C)$ (noting the different directions) with $|f| = |g|^{-1}:|A| = |B| \rightarrow |X| = |Y|$.  We call these morphisms \textbf{strong cross-morphisms of algebra pairs}.  We write
$$\CPair^\sx(\C)$$
to denote the full subcategory of $\Pair^\sx(\C)$ consisting of commuting algebra pairs.  We write
\begin{equation}\label{eqn:cpair_sc_projns}L':\CPair^\sx(\C) \rightarrow \Algs^\s(\C),\;\;\;R':\CPair^\sx(\C) \rightarrow \Algs^\s(\C)^\op\end{equation}
to denote the evident projection functors, given by $L'(A,B) = A$ and $R'(A,B) = B$.
\end{defn}

\begin{para}\label{para:functor_inducing_cpair_sc}
For each pair of $\J$-theories $\T$ and $\U$, let us write $\CAlgPair{\T}{\U}_\si(\C)$ to denote the groupoid core of $\CAlgPair{\T}{\U}(\C)$.  Using \ref{para:func_cpair}, we obtain an evident functor
$$\Upsilon = \CAlgPair{-}{\blanktwo}_\si(\C)\;:\;\ThJ^\op \times \ThJ^\op \longrightarrow \CAT$$
that is given on objects by $(\T,\U) \mapsto \CAlgPair{\T}{\U}_\si(\C)$.  A brief and straightforward verification now yields the following:   
\end{para}

\begin{prop}\label{thm:cpairsc_as_two-sided_fibr}
$\CPair^\sx(\C)$ is isomorphic to the category $\mathsf{TwoSided}(\ThJ,\ThJ^\op,\Upsilon)$ with the notation of \ref{para:str_gr_constr_two-sided} and \ref{para:functor_inducing_cpair_sc}.
\end{prop}

\begin{thm}\label{thm:bif_sc_equiv_cpair_sc}
The category of bifold algebras and strong cross-morphisms, $\BAlg^\sx(\C)$, is equivalent to the category of commuting algebra pairs with strong cross-morphisms, $\CPair^\sx(\C)$.  Indeed,
$$\BAlg^\sx(\C) \;\;\simeq\;\;\CPair^\sx(\C)$$
by way of an equivalence $\Theta:\BAlg^\sx(\C) \rightarrow \CPair^\sx(\C)$ that is given on objects by $D \mapsto (D_\ell,D_r)$ and has $L'\Theta = L$ and $R'\Theta = R$ with the notation of \ref{notn:bif} and \ref{para:lr_face_functors}. 
\end{thm}
\begin{proof}
In view of \ref{defn:bifold_algs_str_cross_morphs} and \ref{thm:cpairsc_as_two-sided_fibr}, this follows from Theorem \ref{thm:tsalgs_equiv_talgs_in_salg} and Remark \ref{rem:nat_equiv}.
\end{proof}

\begin{para}\label{para:langle_rangle_functor}
Given a commuting algebra pair $(A,B)$, the bifold algebra $D = \langle A,B \rangle$ of \ref{notn:tualg_ind_by_tucpair} has $(D_\ell,D_r) \cong (A,B)$ in $\CPair^\sx(\C)$.  Hence, by \ref{thm:bif_sc_equiv_cpair_sc}, there is an equivalence of categories
$$\langle-,\blanktwo\rangle\;:\;\CPair^\sx(\C) \xrightarrow{\sim} \BAlg^\sx(\C)$$
that is given on objects by $(A,B) \mapsto \langle A,B\rangle$ and is a pseudo-inverse of $\Theta$.

For example, if $A$ is an algebra, then since $\RCom A$ is a bifold algebra with left and right faces $A$ and $A^\perp$ \pref{thm:rcom_lcom}, respectively, we find that $\RCom A \cong \langle A,A^\perp\rangle$
in $\BAlg^\sx(\C)$, and similarly $\LCom A \cong \langle A^\perp,A\rangle$.
\end{para}

\begin{cor}\label{thm:rcom_lcom_cpairs}
The functor $L':\CPair^\sx(\C) \rightarrow \Algs^\s(\C)$ has a fully faithful left adjoint $\RCom' = \Theta \RCom$, and the functor $R':\CPair^\sx(\C) \rightarrow \Algs^\s(\C)^\op$ has a fully faithful right adjoint $\LCom' = \Theta\LCom$.  On objects, $\RCom'A = (A,A^\perp)$ and $\LCom'A = (A^\perp,A)$.
\end{cor}
\begin{proof}
This follows from \ref{thm:rcom_lcom}, \ref{thm:cmtnt_adjn}, \ref{thm:bif_sc_equiv_cpair_sc}.
\end{proof}

\begin{defn}\label{defn:rcom_pair}
Let $(A,B)$ be an algebra pair on an object $C$ of $\C$.  We say that $A$ \textbf{is the commutant of} $B$ if $A \cong B^\perp$ in $\Algs(C)$.  $(A,B)$ is a \textbf{right-commutant algebra pair} (resp. a \textbf{left-commutant algebra pair}) if $B$ is the commutant of $A$ (resp. $A$ is the commutant of $B$).  $(A,B)$ is a \textbf{commutant algebra pair} if $(A,B)$ is both left-commutant and right-commutant, in which case we say that \textbf{$A$ and $B$ are mutual commutants}.  We write $\RComPair^\sx(\C)$, $\LComPair^\sx(\C)$, and $\ComPair^\sx(\C)$ for the full subcategories of $\CPair^\sx(\C)$ consisting of right-commutant, left-commutant, and commutant algebra pairs, respectively.  All of these full subcategories of $\CPair^\sx(\C)$ are replete, simply because of the functoriality of $\Com_C = (-)^\perp$ in \ref{eq:cmt_adjn_over_c}.
\end{defn}

\begin{rem}\label{rem:com_balg_com_pair}
By \ref{para:alg_com_w_cmt}, \ref{thm:comm_symm}, and \ref{thm:comm_closed_under_incoming_morphs}, if $(A,B)$ is a right-commutant algebra pair, then $(A,B)$ is a commuting algebra pair.  A commuting algebra pair $(A,B)$ on an object $C$ is right-commutant iff the unique morphism $[B|A]:B \rightarrow A^\perp$ in $\Algs(C)$ is an isomorphism, iff the morphism of $\J$-theories $[B|A]$ is an isomorphism \eqref{eq:b_pipe_a}.  Note also that a bifold algebra $D$ is right-commutant iff the algebra pair $(D_\ell,D_r)$ is right-commutant.  Hence, by \ref{para:langle_rangle_functor} and \ref{defn:rcom_pair}, it follows that a commuting algebra pair $(A,B)$ is right-commutant iff the bifold algebra $\langle A,B\rangle$ is right-commutant, iff $(A,B) \cong (X,X^\perp)$ in $\CPair^\sx(\C)$ for some algebra $X$ (in view of \ref{thm:charns_rcom_balg}, \ref{para:langle_rangle_functor}, \ref{thm:rcom_lcom_cpairs}).  Similar remarks apply to left-commutant algebra pairs.  Hence, the equivalence in \ref{thm:bif_sc_equiv_cpair_sc}, \ref{para:langle_rangle_functor} restricts to equivalences $\RComBAlg^\sx(\C) \simeq \RComPair^\sx(\C)$, $\LComBAlg^\sx(\C) \simeq \LComPair^\sx(\C)$, and $\ComBAlg^\sx(\C) \simeq \ComPair^\sx(\C)$.
\end{rem}

\begin{cor}
$\RComPair^\sx(\C)$ is a replete, coreflective subcategory of $\CPair^\sx(\C)$, and $\RComPair^\sx(\C) \simeq \Algs^\s(\C)$.  $\LComPair^\sx(\C)$ is a replete, reflective subcategory of $\CPair^\sx(\C)$, and $\LComPair^\sx(\C) \simeq \Algs^\s(\C)^\op$.  $\ComPair^\sx(\C)$ is a both a coreflective subcategory of $\LComPair^\sx(\C)$ and a reflective subcategory of $\RComPair^\sx(\C)$, and $\ComPair^\sx(\C) \simeq \SatAlgs^\s(\C) \simeq \SatAlgs^\s(\C)^\op$.
\end{cor}
\begin{proof}
In view of \ref{thm:bif_sc_equiv_cpair_sc} and \ref{rem:com_balg_com_pair}, this follows from \ref{thm:rcmt_corefl} and \ref{thm:cmt_bifold_algs}.
\end{proof}

\section{Preservation and reflection of commuting algebra pairs}\label{sec:pres_refl_calgpair}

We now pause to record some results on preservation and reflection of commutation that we shall need in the next section.

\begin{para}\label{para:pres_refl_calgpairs}
Let $H:\X \rightarrow \Y$ be a $\V$-functor that preserves $\J$-cotensors, where $\X$ and $\Y$ are $\V$-categories with $\J$-cotensors.   Then $H$ necessarily \textit{preserves commuting algebra pairs}, in the sense that if $(\T,A)$, $(\U,B)$ is a commuting algebra pair in $\X$, then the algebra pair $(\T,HA)$, $(\U,HB)$ commutes.  Indeed, in view of \ref{notn:tualg_ind_by_tucpair}, $D = H\langle A,B\rangle$ is a $(\T,\U)$-algebra in $\Y$ with $(D_\ell,D_r) \cong (HA,HB)$ in $\AlgPair{\T}{\U}(\Y)$, so $HA$ commutes with $HB$ by \ref{thm:commutation_and_tu_algs}.  Let us say that $H$ \textbf{reflects commuting algebra pairs} if for every algebra pair $(\T,A)$, $(\U,B)$ in $\X$, if $HA$ commutes with $HB$ then $A$ commutes with $B$.
\end{para}

\begin{prop}\label{thm:ff_refl_calgpairs}
Let $H:\X \rightarrow \Y$ be a faithful $\V$-functor that preserves $\J$-cotensors, where $\X$ and $\Y$ are $\V$-categories with $\J$-cotensors.  Then $H$ reflects commuting algebra pairs.
\end{prop}
\begin{proof}
Let $(\T,A)$, $(\U,B)$ be an algebra pair on an object $X$ of $\X$, and suppose that $HA$ commutes with $HB$.  By \ref{para:morph_between_full_ths}, there is a subtheory embedding $H_X:\X_X \rightarrow \Y_{HX}$ such that $H_X[A] = [HA]$ and $H_X[B] = [HB]$.  Hence $H_X[A]$ commutes with $H_X[B]$ by \ref{thm:cmtn_via_normaln}, so $[A]$ commutes with $[B]$ by \cite[5.15]{Lu:Cmt}, whence $A$ commutes with $B$ by \ref{thm:cmtn_via_normaln}.
\end{proof}

The results on preservation and reflection of commutation for algebra pairs in \ref{para:pres_refl_calgpairs} and \ref{thm:ff_refl_calgpairs} can be seen as analogues of results for morphisms of monads and theories in \cite[4.3]{Kock:DblDln} and \cite[5.15]{Lu:Cmt}.  We now go further to treat also \textit{joint reflection} of commutation, thus obtaining also a result in the presence of an enriched generator:

\begin{para}\label{para:jrefl_calg_pairs}
Let $\Lambda$ be a (possibly large) set, and for each $\lambda \in \Lambda$ let $H_\lambda:\X \rightarrow \Y_\lambda$ be a $\V$-functor that preserves $\J$-cotensors, where $\X$ and $\Y_\lambda$ are $\V$-categories with $\J$-cotensors.  By definition, the family $(H_\lambda)_{\lambda \in \Lambda}$ \textbf{jointly reflects commuting algebra pairs} provided that for every algebra pair $(\T,A)$, $(\U,B)$ in $\X$, if $H_\lambda A$ commutes with $H_\lambda B$ for every $\lambda \in \Lambda$, then $A$ commutes with $B$.  Let us say that $(H_\lambda)_{\lambda \in \Lambda}$ is \textbf{jointly faithful} (in the enriched sense) if the induced $\V'$-functor $H = (H_\lambda)_{\lambda \in \Lambda}:\X \rightarrow \prod_{\lambda \in \Lambda}\Y_\lambda$ is faithful, where the product is taken in $\eCAT{\V'}$.
\end{para}

\begin{prop}
In the situation of \ref{para:jrefl_calg_pairs}, if the family of $\V$-functors $H_\lambda:\X \rightarrow \Y_\lambda$ $(\lambda \in \Lambda)$ is jointly faithful, then this family jointly reflects commuting algebra pairs.
\end{prop}
\begin{proof}
Let $(\T,A)$, $(\U,B)$ be an algebra pair such that $H_\lambda A$ commutes with $H_\lambda B$ for every $\lambda \in \Lambda$.  Then for each $\lambda \in \Lambda$, the bifold algebra $E^\lambda = \langle H_\lambda A,H_\lambda B\rangle$ has $(E^\lambda_\ell,E^\lambda_r) \cong (H_\lambda A,H_\lambda B)$ in $\AlgPair{\T}{\U}(\Y_\lambda)$.  Thus we obtain an induced bifold algebra $E = (E^\lambda)_{\lambda \in \Lambda}:\T \otimes \U \rightarrow \Y$ where $\Y = \prod_{\lambda \in \Lambda} \Y_\lambda$, since $\J$-cotensors in the product $\Y$ are pointwise, and we find that $(E_\ell,E_r) \cong (H A,H B)$ in $\AlgPair{\T}{\U}(\Y)$ where $H:\X \rightarrow \Y$ is defined as in \ref{para:jrefl_calg_pairs}.  Hence $HA$ commutes with $HB$, by \ref{thm:commutation_and_tu_algs}, but $H$ is faithful, so $A$ commutes with $B$ by \ref{thm:ff_refl_calgpairs}.
\end{proof}

A (possibly large) set of objects $\G$ of $\C$ is a \textit{generating class} for $\C$ (in the enriched sense) if the $\V$-functors $\C(X,-):\C \rightarrow \V$ $(X \in \G)$ are jointly faithful.  For example, $\ob\C$ is a generating class for $\C$, by the Yoneda lemma.

\begin{cor}\label{thm:commutation_in_v}
Given a generating class $\G$ for $\C$, the $\V$-functors $\C(X,-):\C \rightarrow \V$ $(X \in \ob\G)$ jointly reflect commuting algebra pairs.  Hence, if $(\T,A)$, $(\U,B)$ is an algebra pair in $\C$, then $A$ commutes with $B$ if and only if $\C(X,A-)$ commutes with $\C(X,B-)$ for all $X \in \ob\G$.
\end{cor}

\section{Commutative, contracommutative, and balanced algebras}\label{sec:comm_contracomm_bal}

\begin{defn}\label{defn:commutative}
An algebra $A$ in $\C$ is \textbf{commutative} if $A$ commutes with itself, i.e., if there exists a (necessarily unique) morphism $[A|A]:A \rightarrow A^\perp$ in $\Algs^\s(C)$, where $C = |A|$.  Let us write $\CAlgs^\s(\C)$ for the full subcategory of $\Algs^\s(\C)$ consisting of commutative algebras.
\end{defn}

\begin{prop}\label{thm:calgs_replete}
If $f:X \rightarrow A$ is a morphism in $\Algs^\s(\C)$ and $A$ is commutative, then $X$ is commutative.  Consequently, the full subcategory $\CAlgs^\s(\C) \hookrightarrow \Algs^\s(\C)$ is replete.
\end{prop}
\begin{proof}
We may apply \ref{thm:comm_closed_under_incoming_morphs} to the morphism $(f,f):(X,X) \rightarrow (A,A)$ in $\Pair^\s(\C)$.
\end{proof}

\begin{prop}\label{thm:charns_comm_th}
Let $\T$ be a $\J$-theory.  The following are equivalent: (1) $\T$ is commutative \pref{para:cmtn}; (2) every $\T$-algebra in $\V$ is commutative; (3) every $\T$-algebra in every $\V$-category $\X$ with $\J$-cotensors is commutative; (4) $(\T,1_\T)$ is a commutative algebra.
\end{prop}
\begin{proof}
This follows readily from \ref{para:pres_refl_calgpairs} and \ref{thm:commutation_in_v}.
\end{proof}

\begin{prop}\label{thm:sat_alg_comm_iff_th_comm}
Let $(\T,A)$ be an algebra in $\C$.  If $A$ is faithful, then $(\T,A)$ is commutative if and only if $\T$ is commutative.  In particular, a saturated algebra $(\T,A)$ is commutative if and only if $\T$ is commutative.
\end{prop}
\begin{proof}
This follows from \ref{para:pres_refl_calgpairs} and \ref{thm:ff_refl_calgpairs}.
\end{proof}

For commutative $\J$-theories, \ref{thm:funcs_ind_by_cmt_mor_th} entails the following:

\begin{cor}\label{thm:cor_for_cmt_th}
Let $\T$ be a commutative $\J$-theory.  Then there is a $(\T,\T)$-algebra $\langle 1_\T, 1_\T\rangle:\T \otimes \T \rightarrow \T$ whose left and right faces are both $1_\T$, and there is a $\V$-functor $\langle 1_\T, 1_\T\rangle^*:\Alg{\T}(\C) \rightarrow \Alg{(\T,\T)}(\C)$
that is given on objects by $A \mapsto \langle A,A\rangle$ and is a common section of $(-,I)^*$ and $(I,-)^*$, with the notation of \ref{thm:funcs_ind_by_cmt_mor_th}.
\end{cor}

\begin{defn}\label{defn:contra}
An algebra $A$ in $\C$ is \textbf{contracommutative} if $A^\perp$ is commutative.
\end{defn}

While contracommutativity is defined in terms of commutativity, it also turns out that commutativity can be characterized in terms of contracommutativity:

\begin{prop}\label{thm:a_comm_iff_aperp_contracomm}
Given an algebra $A$ in $\C$, the following are equivalent:\newline (1) $A$ is commutative, (2) $A^\perp$ is contracommutative, (3) $A^{\perp\perp}$ is commutative, (4) $A$ commutes with $A^{\perp\perp}$.
\end{prop}
\begin{proof}
Let $B = A^{\perp\perp}$.  Then $\Algs(C)(A,B^\perp) \cong \Algs(C)(B,A^\perp)$ by \ref{eq:cmt_adjn_over_c}, and $B^\perp \cong A^\perp$ in $\Algs(C)$ by  \ref{thm:com_adj_idem}, so $\Algs(C)(A,A^\perp) \cong \Algs(C)(B,A^\perp) \cong \Algs(C)(B,B^\perp)$.
\end{proof}

\begin{para}\label{para:cats_contra_com}
Let us write $\CoCAlgs^\s(\C)$, $\SatCAlgs^\s(\C)$, $\SatCoCAlgs^\s(\C)$ to denote the full subcategories of $\Algs^\s(\C)$ consisting of contracommutative, saturated commutative, and saturated contracommutative algebras, respectively.  Each of these subcategories is replete, in view of  \ref{thm:cmtnt_adjn}, \ref{thm:cat_sat_algs}, \ref{thm:calgs_replete}.
\end{para}

\begin{cor}\label{thm:calgs_cocalgs_adjn}
The idempotent adjunction $\Com \dashv \Com^\op:\Algs^\s(\C)^\op \rightarrow \Algs^\s(\C)$ of \eqref{eq:cmt_adj} and \ref{thm:com_adj_idem} restricts to an idempotent adjunction
$$
\xymatrix{
\CAlgs^\s(\C) \ar@{}[rr]|\top \ar@/_2ex/[rr]_{(-)^\perp} & & \CoCAlgs^\s(\C)^\op. \ar@/_2ex/[ll]_{{(-)^\perp}^\op}
}
$$
between the categories of commutative and of contracommutative algebras in $\C$.  This adjunction restricts to an equivalence
\begin{equation}\label{eq:equiv_satcalgs}\SatCAlgs^\s(\C) \;\simeq\; \SatCoCAlgs^\s(\C)^\op\;.\end{equation}
\end{cor}

The following definition provides a convenient new way of describing $\J$-theories $\T$ that are \textit{balanced} with respect to a $\T$-algebra $A$ in the sense of \cite[3.3.5]{Lu:FDistn}:

\begin{defn}
An algebra $A$ in $\C$ is \textbf{balanced} if $A \cong A^\perp$ in $\Algs(C)$, where $C = |A|$.  Let us write $\BalAlgs^\s(\C)$ for the full subcategory of $\Algs^\s(\C)$ consisting of all balanced algebras.
\end{defn}

\begin{rem}\label{rem:bal_alg}
An algebra $A$ on an object $C$ of $\C$ is balanced if and only if $A$ is commutative and the unique morphism $[A|A]:A \rightarrow A^\perp$ in $\Algs(C)$ \pref{defn:commutative} is an isomorphism (equivalently, if its underlying morphism of $\J$-theories is an isomorphism).
\end{rem}

\newpage

\begin{prop}\label{thm:charns_balalgs}
Let $A$ be an algebra in $\C$.  Then 
\begin{enumerate}
\item $A$ is balanced if and only if $A$ is commutative, contracommutative, and saturated.
\item $A^\perp$ is balanced if and only if $A$ is commutative and contracommutative.
\end{enumerate}
\end{prop}
\begin{proof}
Letting $C = |A|$, we first prove (2).  $A^\perp$ is balanced iff $A^\perp \cong A^{\perp\perp}$ in $\Algs(C)$, but $A^\perp$ and $A^{\perp\perp}$ are saturated and $\SatAlgs(C)$ is a preordered class \pref{rem:satalg_preordered_class}, so $A^\perp$ is balanced iff there exist morphisms $A^\perp \rightarrow A^{\perp\perp}$ and $A^{\perp\perp} \rightarrow A^\perp$ in $\Algs(C)$.  By definition, the existence of a morphism $A^\perp \rightarrow A^{\perp\perp}$ in $\Algs(C)$ is the contracommutativity of $A$, while the existence of a morphism $A^{\perp\perp} \rightarrow A^\perp$ in $\Algs(C)$ is the statement that $A$ commutes with $A^{\perp\perp}$, which is equivalent to the commutativity of $A$, by \ref{thm:a_comm_iff_aperp_contracomm}.

We now prove (1).  By (2) it suffices to show that $A$ is balanced if and only if $A^\perp$ is balanced and $A$ is saturated.  If $A$ is balanced, then $A \cong A^\perp$ in $\Algs(C)$ and hence $A$ is saturated, by \ref{thm:charns_sat_algs}, so $A^\perp \cong A \cong A^{\perp\perp}$ in $\Algs(C)$ and hence $A^\perp$ is balanced.  Conversely, if $A^\perp$ is balanced and $A$ is saturated, then $A^\perp \cong A^{\perp\perp}$ and $A^{\perp\perp} \cong A$ in $\Algs(C)$, so $A$ is balanced.
\end{proof}

\begin{rem}\label{rem:balalgs_repl}
The full subcategory $\BalAlgs^\s(\C) \hookrightarrow \Algs^\s(\C)$ is replete, in view of \ref{para:cats_contra_com} and \ref{thm:charns_balalgs}.
\end{rem}

\begin{cor}\label{thm:balalg_equiv_balalgop}
The equivalence \eqref{eq:equiv_satcalgs} restricts to an equivalence $\BalAlgs^\s(\C) \simeq \BalAlgs^\s(\C)^\op$ given by $A \mapsto A^\perp$.  Furthermore, $\BalAlgs^\s(\C)$ is a groupoid, and the latter equivalence is isomorphic to the identity-on-objects isomorphism $\BalAlgs^\s(\C) \rightarrow \BalAlgs^\s(\C)^\op$ given on morphisms by $f \mapsto f^{-1}$.
\end{cor}
\begin{proof}
The first claim follows from \ref{thm:charns_balalgs}.  For the second, let $A$ and $B$ be balanced algebras in $\C$, let $C = |A|$ and $C' = |B|$, and let $f:A \rightarrow B$ be a morphism in $\Algs^\s(\C)$.  Then we obtain a morphism $f^\perp:B^\perp \rightarrow A^\perp$ in $\Algs^\s(\C)$ with $|f^\perp| = |f|^{-1}$ in $\C_\si$, by \ref{rem:com_cmt_w_car}.  But $[A|A]:A \rightarrow A^\perp$ and $[B|B]:B \rightarrow B^\perp$ are isomorphisms in the (non-full) subcategories $\Algs(C)$ and $\Algs(C')$ of $\Algs^\s(\C)$, respectively.  Hence, there is a unique morphism $g:B \rightarrow A$ in $\Algs^\s(\C)$ with $[A|A] \cdot g = f^\perp \cdot [B|B]$, and we find that $|g| = |f|^{-1}$.  Therefore, since $A$ and $B$ are saturated, it follows that $g \cdot f$ and $f \cdot g$ are endomorphisms in the preordered classes $\SatAlgs(C)$ and $\SatAlgs(C')$ \pref{rem:satalg_preordered_class} and so must be identities.  Therefore $f$ is an isomorphism in $\BalAlgs^\sx(\C)$ with $f^{-1} = g$, and the result follows.
\end{proof}

\section{Commutant bifold algebras with one commutative face}\label{sec:one_comm_face}

By Theorem \ref{thm:rcmt_corefl} and Proposition \ref{thm:calgs_replete}, the category of commutative algebras, $\CAlgs^\s(\C)$, is equivalent to the full subcategory of $\BAlg^\sx(\C)$ consisting of right-commutant bifold algebras whose left face is commutative.  If $D$ is a right- or left-commutant bifold algebra, then its left face $D_\ell$ is commutative if and only if its right face $D_r$ is \mbox{contracommutative}, by \ref{thm:calgs_replete}, \ref{thm:a_comm_iff_aperp_contracomm}, and \ref{para:cats_contra_com}.  For commutant bifold algebras, we arrive at the following corollary to \ref{thm:cmt_bifold_algs} and \ref{thm:calgs_cocalgs_adjn}:

\begin{thm}\label{thm:balgs_w_lface_comm}
The equivalent categories $\SatCAlgs^\s(\C)$ and $\SatCoCAlgs^\s(\C)^\op$ are equivalent to the full subcategory of $\BAlg^\sx(\C)$ consisting of commutant bifold algebras whose left face is commutative (equivalently, whose right face is contracommutative).
\end{thm}

Every right-commutant bifold algebra with commutative left face has a central morphism from its left face to its right face:

\begin{prop}\label{thm:rcom_lcommutative_central_morph}
Let $(\T,\U,D)$ be a right-commutant bifold algebra whose left face $(\T,D_\ell)$ is commutative.  Then (1) there is a unique morphism $M:(\T,D_\ell) \rightarrow (\U,D_r)$ in $\Algs(C)$, where $C = |D|$, (2) the morphism of $\J$-theories $M:\T \rightarrow \U$ is central, and (3) there is a $(\T,\U)$-algebra $\langle M,1_\U\rangle:\T \otimes \U \rightarrow \U$ with left and right faces $M$ and $1_\U$, respectively.
\end{prop}
\begin{proof}
Let $A = D_\ell$ and $B = D_r$.  Then $(\U,B) \cong (\T,A)^\perp$ in $\Algs(C)$, so (1) follows from the fact that $(\T,A)$ is commutative.  By \ref{para:algs_on_c}, $A \cong BM$ in $\Alg{\T}(C)$, but $A$ commutes with $B$ \pref{thm:dl_commutes_w_dr} and hence $BM$ commutes with $B$, by \ref{thm:calg_pair_repl}.  But $B$ is saturated and hence faithful, by \ref{thm:charns_sat_algs} and \ref{rem:sat_alg_faithful}, so since $BM$ commutes with $B1_\U$, we deduce by \ref{thm:ff_refl_calgpairs} that $M$ commutes with $1_\U$, i.e. $M$ is central \pref{para:cmtn}.  (3) now follows, by \ref{thm:funcs_ind_by_cmt_mor_th}.
\end{proof}

\begin{thm}\label{thm:rc_comm_ualg_to_tu_alg}
Let $(\T,\U,D)$ be a right-commutant bifold algebra whose left face is commutative.  Then every $\U$-algebra $B$ in $\C$ is the right face of a $(\T,\U)$-algebra.  Moreover, there is a $\V$-functor $\langle M, 1_\U\rangle^*:\Alg{\U}(\C) \rightarrow \Alg{(\T,\U)}(\C)$ that is a section of $(I,-)^*$, with the notation of \ref{thm:funcs_ind_by_cmt_mor_th} and \ref{thm:rcom_lcommutative_central_morph}.  Explicitly, $\langle M,1_\U\rangle^*$ sends each $\U$-algebra $B$ to the $(\T,\U)$-algebra $B\langle M,1_\U\rangle = \langle BM,B\rangle$.
\end{thm}
\begin{proof}
This follows from \ref{thm:funcs_ind_by_cmt_mor_th} and \ref{thm:rcom_lcommutative_central_morph}.
\end{proof}

\begin{para}\label{para:f_an}
A \textbf{functional-analytic context} \cite{Lu:FDistn}, $(\V,\J,\T,A)$, consists of an eleutheric system of arities $\J \hookrightarrow \V$, a commutative $\J$-theory $\T$, and a saturated $\T$-algebra $A$.  But a saturated algebra $(\T,A)$ is commutative if and only if $\T$ is commutative \pref{thm:sat_alg_comm_iff_th_comm}, and $\V$ itself is always $\J$-admissible when $\J$ is eleutheric \pref{para:talg}.  Hence, by Theorem \ref{thm:balgs_w_lface_comm}, a functional-analytic context over an eleutheric system of arities $\J \hookrightarrow \V$ is equivalently given by a commutant bifold algebra whose left face is commutative, and Theorem \ref{thm:rc_comm_ualg_to_tu_alg} applies.
\end{para}

\section{Balanced bifold algebras}\label{sec:bal_bif}

\begin{defn}
A \textbf{balanced bifold algebra} is a commutant bifold algebra $D$ such that $D_\ell \cong D_r$ in $\Algs(C)$, where $C = |D|$.  We write $\BalBAlg^\sx(\C)$ to denote the full subcategory of $\BAlg^\sx(\C)$ consisting of balanced bifold algebras.
\end{defn}

\begin{prop}\label{thm:charns_bal_balg}
Given a bifold algebra $D$ in $\C$, the following are equivalent: (1) $D$ is a balanced bifold algebra, (2) $D$ is a commutant bifold algebra whose left and right faces $D_\ell$ and $D_r$ are both commutative, (3) $D \cong \langle A,A\rangle$ in $\BAlg^\sx(\C)$ for some balanced algebra $A$ in $\C$, (4) $D$ is a commutant bifold algebra whose left and right faces are both balanced, (5) the algebras $D_\ell$, $D_r$, $D_\ell^\perp,D_r^\perp$ are all isomorphic in $\Algs(C)$, where $C = |D|$.  Consequently, the full subcategory $\BalBAlg^\sx(\C) \hookrightarrow \BAlg^\sx(\C)$ is replete.
\end{prop}
\begin{proof}
It is immediate from the definitions that $(1)\Leftrightarrow(5)\Leftrightarrow(4)$, while $(4)\Rightarrow(2)$ by \ref{thm:charns_balalgs}.  Conversely, if (2) holds, then $D_\ell \cong D_r^\perp$ and $D_r \cong D_\ell^\perp$ in $\Algs(C)$, so since $D_\ell$ and $D_r$ are commutative we deduce that $D_\ell$ and $D_r$ are contracommutative \pref{defn:contra}, but $D_\ell$ and $D_r$ are also saturated, by \ref{thm:charns_sat_algs}, so (4) holds, by \ref{thm:charns_balalgs}.  Next, if the equivalent conditions (4) and (5) hold, then $D \cong \langle D_\ell,D_r\rangle \cong \langle D_\ell,D_\ell \rangle$ in $\BAlg^\sx(\C)$ by \ref{para:langle_rangle_functor}, and hence (3) holds.  Lastly, supposing (3), it suffices to prove (4).  But the properties of $D$ required for (4) are stable under isomorphism in $\BAlg^\sx(\C)$, by \ref{thm:cmt_bifold_algs} and \ref{rem:balalgs_repl}, so we may assume that $D = \langle A,A\rangle$ with $A$ balanced.  Hence $|A| = C$, $A \cong A^\perp$ in $\Algs(C)$, and $D_\ell \cong A \cong D_r$ in $\Algs(C)$ by \ref{notn:tualg_ind_by_tucpair}, so (5) holds, and hence (4) holds.
\end{proof}

\begin{cor}\label{thm:charn_of_balalgs_via_bifoldalgs}
Given an algebra $A$, the following are equivalent: (1) $A$ is balanced, (2) $\RCom A$ is balanced, (3) $\LCom A$ is balanced, (4) $A$ is commutative and $\langle A,A\rangle$ is balanced, (5) $A$ is commutative and $\langle A,A\rangle$ is commutant.  If $A$ is balanced then $\RCom A \cong \LCom A \cong \langle A,A\rangle$ in $\BAlg^\sx(\C)$.
\end{cor}
\begin{proof}
If $A$ is balanced, then $A \cong A^\perp$ in $\Algs(C)$, so $\RCom A \cong \langle A,A^\perp \rangle \cong \langle A,A\rangle \cong \langle A^\perp,A\rangle \cong \LCom A$ in $\BAlg^\sx(\C)$ by \ref{para:langle_rangle_functor}.  The needed equivalences follow readily from \ref{thm:charns_bal_balg}, using \ref{notn:tualg_ind_by_tucpair} and \ref{rem:balalgs_repl}.
\end{proof}

\begin{cor}
There is an equivalence of categories
$$\BalAlgs^\s(\C) \;\simeq\; \BalBAlg^\sx(\C)$$
that sends each balanced algebra $A$ to the balanced bifold algebra $\langle A,A\rangle$.
\end{cor}
\begin{proof}
The equivalence $\SatAlgs^\s(\C) \xrightarrow{\sim} \ComBAlg^\sx(\C)$
in \ref{thm:cmt_bifold_algs} is given on objects by $A \mapsto \RCom A = \langle A,A^\perp \rangle$, and the result readily follows, by using \ref{thm:charns_bal_balg} and \ref{thm:charn_of_balalgs_via_bifoldalgs}.
\end{proof}

\section{Examples}\label{sec:exa_bif}

\begin{para}[\textbf{Basic example: Bimodules}]\label{exa:dln_mod}
In this example, we consider the system of arities $\{\ZZ\} \hookrightarrow \Ab$ consisting of just the unit object $\ZZ$ of $\V = \Ab$ \pref{para:exa_sys_ar}.  Algebras in $\Ab$ for this system of arities are \textit{left modules}, i.e., pairs $(T,A)$ in which $T$ is a (unital) ring and $A = \tensor[_T]{A}{}$ is a left $T$-module.  The commutant $T^\perp_A$ is the ring $\End(\tensor[_T]{A}{})$ of $T$-linear endomorphisms of $A$, which is equally the ring-theoretic commutant (or centralizer) of the image $\varphi(T) \hookrightarrow \End(\tensor[_\ZZ]{A}{})$ of the canonical ring homomorphism $\varphi:T \rightarrow \End(\tensor[_\ZZ]{A}{})$, where $\End(\tensor[_\ZZ]{A}{})$ is the ring of endomorphisms of the abelian group underlying $A$.  The commutant $(T,A)^\perp$ of $(T,A)$ is obtained by regarding $A$ as a left $\End(\tensor[_T]{A}{})$-module.  $(T,A)$ is saturated iff the canonical ring homomorphism from $T$ into the double centralizer of the subring $\varphi(T) \hookrightarrow \End(\tensor[_\ZZ]{A}{})$ is an isomorphism.  $(T,A)$ is commutative iff the ring $\varphi(T)$ commutative; if $A$ is faithful (meaning that $\varphi$ is injective) then $(T,A)$ is commutative iff $T$ is commutative.  $(T,A)$ is contracommutative iff $\End(\tensor[_T]{A}{})$ is a commutative ring.  $(T,A)$ is balanced (equivalently: commutative, contracommutative, and saturated, \ref{thm:charns_balalgs}) iff $T$ is commutative and the resulting ring homomorphism $T \rightarrow \End(\tensor[_T]{A}{})$ is an isomorphism.  For example, if $T$ is commutative and we take $A = \tensor[_T]{T}{}$, then $(T,A) = (T,\tensor[_T]{T}{})$ is balanced.

By \ref{exa:bim}, bifold algebras in $\Ab$ for this system of arities $\{\ZZ\} \hookrightarrow \Ab$ are \textit{bimodules}, i.e., triples of the form $(T,U,D) = (T,S^\op,D)$ where $T$ and $S$ are rings and $D = \tensor[_T]{D}{_S}$ is a $T$-$S$-bimodule.  The left face $D_\ell$ of $D$ is the left $T$-module $\tensor[_T]{D}{}$, while the right face $D_r$ is the right $S$-module $\tensor{D}{_S}$ (regarded as a left $S^\op$-module).  $(T,U,D)$ is right-commutant (resp. left-commutant) iff the canonical ring homomorphism $S^\op \rightarrow \End(\tensor[_T]{D}{})$ (resp. $T \rightarrow \End(\tensor{D}{_S})$) is an isomorphism.  For example, given any ring $T$, if we regard $T$ itself as a $T$-$T$-bimodule then by \cite[8.14]{Lu:Cmt} we obtain a commutant bifold algebra $(T,U,D) = (T,T^\op,\tensor[_T]{T}{_T})$, which is balanced iff $T$ is commutative.

Given any ring $T$ and any left $T$-module $A$, we obtain a right-commutant bifold algebra $D = \RCom(T,A) = (T,S^\op,\tensor[_T]{A}{_S})$ with $S^\op = \End(\tensor[_T]{A}{})$, by regarding $A$ as a $T$-$S$-bimodule, and $D$ is a commutant bifold algebra iff $(T,A)$ is saturated \pref{thm:a_sat_iff_rcom_is_cbalg}, in which case $D$ is balanced iff $T$ and $\End(\tensor[_T]{A}{})$ are both commutative \pref{thm:charns_bal_balg}, iff $(T,A)$ is balanced \pref{thm:charn_of_balalgs_via_bifoldalgs}.

A strong morphism of algebras $(m,f):(T,A) \rightarrow (U,B)$ in $\Algs^\s(\Ab)$ for this system of arities $\{\ZZ\}$ consists of a ring homomorphism $m:T \rightarrow U$ and an isomorphism of $T$-modules $f:A \rightarrow \tensor[_T]{B}{}$, where $\tensor[_T]{B}{}$ is the $T$-module obtained from $B$ by restriction of scalars along $m$.  A strong cross-morphism of bifold algebras $(m,n^\op,f):(T_1,S_1^\op,D) \rightarrow (T_2,S_2^\op,E)$ is given by a pair of ring homomorphisms $m:T_1 \rightarrow T_2$ and $n:S_2 \rightarrow S_1$ together with an isomorphism of $T_1$-$S_2$-bimodules $f:\tensor[_{T_1}]{D}{_{S_2}} \rightarrow \tensor[_{T_1}]{E}{_{S_2}}$, where $\tensor[_{T_1}]{D}{_{S_2}}$ and $\tensor[_{T_1}]{E}{_{S_2}}$ are obtained from $D$ and $E$ by restriction of scalars along $n$ and $m$, respectively.
\end{para}

\subsection{Examples in cartesian closed categories, with finite arities}

Let $\V$ be a cartesian closed category, and suppose that $\V_0$ is countably complete and countably cocomplete.  We now discuss several examples involving the system of arities $j:\DF(\V) \rightarrow \V$ of \ref{para:exa_sys_ar}(5), recalling that $\DF(\V)$-theories are called discretely finitary-algebraic theories enriched in $\V$, which here we call simply \textbf{theories} for brevity.  

\begin{exasub}[\textbf{Modules for internal rings and rigs in $\V$}]\label{exa:internal_rig_bifold_alg}
Let us write $\CMon(\V)$ and $\Ab(\V)$ for the categories of (internal) commutative monoids in $\V$ and abelian groups in $\V$, respectively.  Then $\CMon(\V)$ and $\Ab(\V)$ underlie symmetric monoidal closed categories \cite[6.3.2--6.3.7]{Lu:FDistn}.  A monoid in the monoidal category $\CMon(\V)$ is equivalently described as an \textbf{(internal) rig $R$ in $\V$}, i.e., an internal unital semiring in $\V$ \cite[2.7, 6.4.1]{Lu:FDistn}.  Left (resp. right) $R$-modules for the monoid $R$ in $\CMon(\V)$ are called simply \textbf{left (\textnormal{resp.} right) $R$-modules} for the rig $R$ in $\V$.  Similarly, monoids in $\Ab(\V)$ are \textbf{(internal) rings} in $\V$ and may be regarded equivalently as internal rigs whose underlying additive monoid is an internal group \cite[2.7]{Lu:FDistn}.

Given a rig $R$ in $\V$, we obtain by \ref{para:exa_sys_ar}(4) a $\CMon(\V)$-category $\Mod{R}$, which has an underlying $\V$-category that we denote also by $\Mod{R}$, or $\Mod{R}(\V)$, obtained by change of base along the symmetric (lax) monoidal forgetful functor $U:\CMon(\V) \rightarrow \V$ \cite[3.4.1, 6.3.6]{Lu:FDistn}.   There is a theory $\Mat(R)$ enriched in $\V$, in which $\Mat(R)(m,n) = R^{n \times m}$ is the object of $n \times m$-matrices over $R$, with composition given by (internal) matrix multiplication \cite[6.4.5]{Lu:FDistn}, and $\Alg{\Mat(R)}^\nml(\V) \cong \Mod{R}(\V)$ \cite[6.4.6]{Lu:FDistn}, so that we may identify normal $\Mat(R)$-algebras with left $R$-modules in $\V$.

Regarded as a left (resp. right) $R$-module, $R$ itself is a normal $\Mat(R)$-algebra (resp. $\Mat(R^\op)$-algebra).  By \cite[7.2]{Lu:FDistn}, the algebras $(\Mat(R),R)$ and $(\Mat(R^\op),R)$ are mutual commutants, so by \ref{rem:com_balg_com_pair} these algebras are the left and right faces of a commutant bifold algebra that we may write as $(\Mat(R),\Mat(R^\op),R)$.  If $R$ is commutative, then $R^\op = R$, and the algebras $(\Mat(R),R)$ and $(\Mat(R^\op),R)$ are identical, so $(\Mat(R),\Mat(R),R)$ is a balanced bifold algebra and $(\Mat(R),R)$ is a balanced algebra (and in particular, $\Mat(R)$ is commutative).

In view of \cite[6.3]{Lu:FDistn}, the commutative rig of natural numbers $\NN$ determines an internal commutative rig in $\V$, namely the copower $\NN \cdot 1$ in $\V$ of the terminal object $1$ of $\V$.  Also, $\NN \cdot 1$ is the unit object of the monoidal category $\CMon(\V)$, by \cite[3.4.3]{Lu:FDistn}.  Hence, $\CMon(\V)$ may be identified with the category of internal $(\NN \cdot 1)$-modules in $\V$.  By the above, $\NN \cdot 1$ underlies a balanced algebra for the $\V$-enriched theory $\Mat(\NN\cdot 1)$ of internal commutative monoids.  The same reasoning yields analogous conclusions with regard to the internal commutative ring $\ZZ \cdot 1$ in $\V$ and its category of modules in $\V$, which may be identified with $\Ab(\V)$.
\end{exasub}

\begin{exasub}[\textbf{Internal left $R$-affine spaces and pointed right $R$-modules}]\label{exa:r_aff}
Let $R$ be an internal rig in $\V$.  There is a subtheory $\Mat^\aff(R)$ of $\Mat(R)$ for which $\Mat^\aff(R)(m,n)$ is \textit{the object of all $n \times m$-matrices over $R$ in which each row has sum $1$} \cite[8.4]{Lu:FDistn}.  More precisely, $\Mat^\aff(R)$ is the \textit{affine core} of $\Mat(R)$ \cite[8.3]{Lu:FDistn}.  Normal $\Mat^\aff(R)$-algebras in $\V$ are called \textbf{(left) $R$-affine spaces in $\V$} \cite[8.5]{Lu:FDistn}, and we write $\Aff{R} = \Alg{\Mat^\aff(R)}^\nml(\V)$.  By way of the subtheory embedding $\Mat^\aff(R) \hookrightarrow \Mat(R)$, every left $R$-module has an underlying left $R$-affine space.

In particular, $R$ itself is a left $R$-affine space, and by \cite[9.3]{Lu:FDistn} the resulting algebra $(\Mat^\aff(R),R)$ is always the commutant of another algebra that we now describe.  A \textbf{pointed (left) $R$-module} is a left $R$-module $M$ equipped with a designated morphism $*:1 \rightarrow M$ in $\V$.  Pointed left $R$-modules are the objects of a $\V$-category $\Mod{R}^*$ that is isomorphic to the $\V$-category $\Alg{\Mat^*(R)}(\V)$ of normal $\Mat^*(R)$-algebras for a theory $\Mat^*(R)$ \cite[9.2]{Lu:FDistn}.  A \textbf{pointed right $R$-module} is a pointed left $R^\op$-module.  Let us regard $R$ as a pointed right $R$-module, whose designated morphism $1 \rightarrow R$ is the multiplicative identity of $R$, and write $R$ also to denote the corresponding $\Mat^*(R^\op)$-algebra.

The algebra $(\Mat^\aff(R),R)$ is the commutant of the algebra $(\Mat^*(R^\op),R)$, by \cite[9.3]{Lu:FDistn}.  Hence $(\Mat^\aff(R),R)$ is saturated, and, by \ref{rem:com_balg_com_pair}, $R$ carries the structure of a left-commutant bifold algebra $(\Mat^\aff(R),\Mat^*(R^\op),R)$ with left and right faces $(\Mat^\aff(R),R)$ and $(\Mat^*(R^\op),R)$, respectively.  If $R$ is a commutative rig in $\V$, then the theory $\Mat^\aff(R)$ is commutative \cite[8.7]{Lu:FDistn}, so $(\Mat^*(R^\op),R)$ is contracommutative.

If the internal rig $R$ is a ring, then by \cite[10.5]{Lu:FDistn} the algebra $(\Mat^*(R^\op),R)$ is the commutant of $(\Mat^\aff(R),R)$, so that $(\Mat^\aff(R),R)$ and $(\Mat^*(R^\op),R)$ are mutual commutants, and hence $(\Mat^\aff(R),\Mat^*(R^\op),R)$ is a commutant bifold algebra.  Hence, if $R$ is, moreover, a commutative ring in $\V$, then by \ref{thm:rc_comm_ualg_to_tu_alg}, every pointed right $R$-module carries the structure of a $(\Mat^\aff(R),\Mat^*(R))$-algebra.

On the other hand, for arbitrary \textit{rigs} $R$ in $\V$, it is \textit{not} in general the case that $(\Mat^*(R^\op),R)$ is the commutant of $(\Mat^\aff(R),R)$, by \cite[\S 8]{Lu:CvxAffCmt}, and we shall illustrate this in \ref{exa:dualn_semilattices}.  But there are interesting examples of rigs that are not rings and yet still have the desired property; the following provides a class of examples:
\end{exasub}

\begin{exasub}[\textbf{$R$-convex spaces and pointed $R_+$-modules}]
A \textbf{preordered ring in $\V$} is a ring $R$ in $\V$ equipped with a sub-rig $R_+ \hookrightarrow R$, called the \textbf{positive part} of $R$; see \cite[8.8]{Lu:FDistn} for the reasons for this terminology.  For example, the ring of real numbers $\RR$ is a preordered ring in $\Set$ with $\RR_+ = [0,\infty)$.  By definition, a \textbf{(left) $R$-convex space} in $\V$ is a (left) $R_+$-affine space in $\V$, so with the notation of \ref{exa:r_aff} we call $\Mat^\aff(R_+)$ the \textit{theory of left $R$-convex spaces}.  We denote the $\V$-category $\Aff{R_+} = \Alg{\Mat^\aff(R_+)}^\nml(\V)$ by $\Cvx{R}$.  By the preceding Example \ref{exa:r_aff}, we may regard $R_+$ as both a left $R$-convex space and a pointed right $R_+$-module, and the resulting algebra $(\Mat^\aff(R_+),R_+)$ is the commutant of $(\Mat^*(R_+^\op),R_+)$.  In particular, $R_+$ carries the structure of a left-commutant bifold algebra $(\Mat^\aff(R_+),\Mat^*(R_+^\op),R_+)$.

Theorem 10.6 in \cite{Lu:FDistn} gives sufficient conditions that entail that $(\Mat^*(R_+^\op),R_+)$ is also the commutant of $(\Mat^\aff(R),R_+)$, so that $(\Mat^\aff(R_+),\Mat^*(R_+^\op),R_+)$ is then a commutant bifold algebra.  Indeed, for this it suffices to assume, as we now shall, that $R_+ \hookrightarrow R$ is a strong monomorphism in $\V_0$ and that $\V_0$ has a generating class $\G$ such that for each object $X \in \G$, the preordered ring $\V_0(X,R)$ in $\Set$ is a \textit{firmly archimedean preordered algebra over the dyadic rationals} \cite[10.1, 10.6]{Lu:FDistn}.  For example, the preordered ring $\RR$ in $\Set$ satisfies these assumptions, because $\RR \cong \Set(1,\RR)$ itself is a firmly archimedean preordered algebra over the dyadic rationals \cite[10.21]{Lu:CvxAffCmt}.  As another example, if $\V$ is the category of convergence spaces \pref{para:exa_sys_ar}, then the preordered ring object $\RR$ in $\V$ (with its usual convergence structure) satisfies these assumptions \cite[10.9]{Lu:FDistn}, because $\V_0(1,\RR)$ is again the usual preordered ring of reals.  As a further example, in the case where $\V$ is the Cahiers topos \pref{para:exa_sys_ar}, the line object $R$ in $\V$ also satisfies these assumptions \cite[10.12]{Lu:FDistn}.  Let us assume also that $R$ is commutative, as is the case in these three examples, so that $R = R^\op$.  Under these assumptions $(\Mat^\aff(R_+),\Mat^*(R_+),R_+)$ is a commutant bifold algebra whose left face is commutative, and by \ref{thm:rc_comm_ualg_to_tu_alg}, every pointed $R_+$-module carries the structure of a $(\Mat^\aff(R_+),\Mat^*(R_+))$-algebra.
\end{exasub}

\begin{exasub}[\textbf{Several kinds of semilattices}]\label{exa:dualn_semilattices}
We follow \cite{Joh:StSp} in taking the term \textit{join semilattice} to mean a partially ordered set with finite joins (and, in particular, a bottom element $\bot$).  By contrast, a partially ordered set that is only assumed to have a join of each \textit{pair} of elements we call a \textit{binary-join semilattice} (while many authors call these join semilattices).  Join semilattices (resp. binary-join semilattices) are the objects of a category $\SLat_{\vee\bot}$ (resp $\SLat_{\vee}$) in which the morphisms are mappings that preserve finite joins (resp. binary joins).

The two-element set $2 = \{0,1\}$ is a distributive lattice, so we may equip $2$ with the structure of a commutative rig by taking $\vee$ as addition and $\wedge$ as multiplication.  The category $\Mod{2}$ of $2$-modules for this rig $2$ is isomorphic to $\SLat_{\vee\bot}$, while the category $\Aff{2}$ of $2$-affine spaces is isomorphic to $\SLat_{\vee}$.  Thus we obtain special cases of \ref{exa:internal_rig_bifold_alg} and \ref{exa:r_aff} with $\V = \Set$ and $R = 2$.  Firstly, by \ref{exa:internal_rig_bifold_alg} we obtain a balanced bifold algebra $(\Mat(2),\Mat(2),2)$.  Secondly, by \ref{exa:r_aff} we obtain a left-commutant bifold algebra $(\Mat^\aff(2),\Mat^*(2),2)$, but this bifold algebra is \textit{not} a right-commutant bifold algebra, by \cite[8.2]{Lu:CvxAffCmt}.  Hence, the algebra $(\Mat^*(2),2)$ is not saturated, by \ref{thm:charn_comm_balg}, while it is contracommutative by \ref{exa:r_aff}.

On the other hand, $(\Mat^\aff(2),2)$ is a saturated algebra, by \ref{exa:r_aff}, so by \ref{thm:a_sat_iff_rcom_is_cbalg} it induces a commutant bifold algebra $\RCom(\Mat^\aff(2),2)$, which we now describe.  Let $\SLat_{\vee\bot\top}$ denote the category whose objects are join semilattices with a top element (in addition to the bottom element that we require in any join semilattice), with maps that preserve not only finite joins (and, in particular, the bottom element) but also the top element.  This category $\SLat_{\vee\bot\top}$ is isomorphic to the category of normal $\U$-algebras for a Lawvere theory $\U$ that is described in \cite[8.1]{Lu:CvxAffCmt}.  The join semilatttice $2$ also has a top element, so the set $2$ carries the structure of a $\U$-algebra.  By \cite[8.2]{Lu:CvxAffCmt}, the algebras $(\Mat^\aff(2),2)$ and $(\U,2)$ are mutual commutants, so by \ref{rem:com_balg_com_pair} we find that $2$ carries the structure of a commutant bifold algebra $(\Mat^\aff(2),\U,2)$, whose left face is commutative, so $(\U,2)$ is a saturated contracommutative algebra.  By \ref{thm:rc_comm_ualg_to_tu_alg}, every join semilattice with top element carries the structure of a $(\Mat^\aff(2),\U)$-algebra.  The Lawvere theories $\T = \Mat^\aff(2)$ and $\U$ are not isomorphic, e.g. because $\T(0,1) = \emptyset$ while $\U(0,1) \cong 2$ \cite[3.3, 8.1]{Lu:CvxAffCmt}, so $(\Mat^\aff(2),\U,2)$ is not balanced.  Hence, the saturated contracommutative algebra $(\U,2)$ is not commutative, by \ref{thm:charns_bal_balg}, and in particular $\U$ is not commutative, by \ref{thm:charns_comm_th}.
\end{exasub}

\begin{exasub}[\textbf{Locally compact groups and the circle group}]\label{exa:lcgroups_circle_group}
In this example, we let $\V = \Set$ and consider the Lawvere theory $\Mat(\ZZ)$ of abelian groups.  Let us write $\Top$ for the category of topological spaces and $\LCHaus$ for the full subcategory of $\Top$ consisting of locally compact Hausdorff spaces.  Then $\LCHaus$ has finite products, formed as in $\Top$.  The category of normal $\Mat(\ZZ)$-algebras in $\C = \LCHaus$ is isomorphic to the category $\LCAb = \Ab(\LCHaus)$ of locally compact Hausdorff topological abelian groups, which we call simply \textit{locally compact groups}.  In particular, the circle group $\TT \cong \RR \slash \ZZ$ is an object of $\LCAb$ and so carries the structure of an algebra $(\Mat(\ZZ),\TT)$ in $\LCHaus$.  The \textit{Pontryagin duality} is an adjoint equivalence $\Hom(-,\TT) \dashv \Hom(-,\TT):\LCAb^\op \rightarrow \LCAb$ under which each locally compact group $A$ corresponds to its \textit{Pontryagin dual} $\Hom(A,\TT)$, which is the abelian group of all continuous group homomorphisms from $A$ to $\TT$, equipped with the compact-open topology; see e.g. \cite{Mor}, \cite{Roe}.  Hence for each each locally compact group $A$ we have an isomorphism
\begin{equation}\label{eq:pontryagin_double_dual_iso}A \xrightarrow{\sim} \Hom(\Hom(A,\TT),\TT)\end{equation}
in $\LCAb$.  Regarding $\ZZ$ as a discrete group and so as a locally compact group, the Pontryagin dual of $\ZZ$ is isomorphic to $\TT$, i.e. $\Hom(\ZZ,\TT) \cong \TT$ in $\LCAb$ \cite[p. 51]{Mor}.  Since $\LCAb$ is an additive category \cite{Roe}, finite products in $\LCAb$ are biproducts, so for each $n \in \NN$, the power $\ZZ^n$ in $\LCAb$ (which is again discrete) is also a copower in $\LCAb$ and is sent by the equivalence $\Hom(-,\TT):\LCAb^\op \rightarrow \LCAb$ to a power $\Hom(\ZZ^n,\TT) \cong \Hom(\ZZ,\TT)^n \cong \TT^n$ in $\LCAb$.

The Lawvere theory $\Mat(\ZZ)$ is commutative \pref{exa:internal_rig_bifold_alg}, so the algebra $(\Mat(\ZZ),\TT)$ in $\LCHaus$ is commutative \pref{thm:charns_comm_th}, and hence there is a canonical morphism of Lawvere theories
$$[\TT|\TT]:\Mat(\ZZ) \rightarrow \Mat(\ZZ)^\perp_\TT$$
valued in the commutant of $\Mat(\ZZ)$ with respect to the $\Mat(\ZZ)$-algebra $\TT$ in $\LCHaus$ \pref{eq:b_pipe_a}.  For each $n \in \ob\Mat(\ZZ) = \NN$,
\begin{equation}\label{eq:tt_lcab}[\TT|\TT]_{n1}:\Mat(\ZZ)(n,1) = \ZZ^{1\times n} = \ZZ^n \rightarrow \Mat(\ZZ)^\perp_\TT(n,1) = \LCAb(\TT^n,\TT)\end{equation}
is the map that sends each row vector $(a_1,a_2,...,a_n)$ to the homomorphism $\TT^n \rightarrow \TT$ given by $(t_1,t_2,...,t_n) \mapsto \sum_{i=1}^n a_it_i$, where we use additive notation in the group $\TT$.  This map $[\TT|\TT]_{n1}$ underlies the canonical isomorphism
\begin{equation}\label{eq:zn_double-dual_iso}\ZZ^n \overset{\sim}{\longrightarrow} \Hom(\Hom(\ZZ^n,\TT),\TT) \cong \Hom(\TT^n,\TT)\;\;\;\;\text{in $\LCAb$}\end{equation}
that exhibits the discrete group $\ZZ^n$ as isomorphic to its Pontryagin double-dual.  Hence, by \ref{para:mor_jth}, $[\TT|\TT]$ is an isomorphism of Lawvere theories $\Mat(\ZZ) \cong \Mat(\ZZ)^\perp_\TT$.  Consequently, $(\Mat(\ZZ),\TT)$ is a balanced algebra in $\LCHaus$, so $\TT$ carries the structure of a balanced bifold algebra in $\LCHaus$ that we may write as $(\Mat(\ZZ),\Mat(\ZZ),\TT)$.
\end{exasub}

\begin{exasub}[\textbf{Convergence groups and the circle group}]\label{exa:conv_grps_circle_grp}
We now develop a variation on Example \ref{exa:lcgroups_circle_group} in which we instead take the base of enrichment $\V$ to be the cartesian closed category of convergence spaces, $\V = \Conv$ \pref{para:conv_cah}.  $\LCHaus$ may be identified with a full subcategory of $\V$ \pref{para:conv_cah}, and in this example we may take $\C$ to be either $\V$ itself or its full sub-$\V$-category $\LCHaus$.  Abelian groups in $\Conv$ are called \textbf{convergence abelian groups} \cite{Binz} and are the objects of a symmetric monoidal closed category $\Ab(\Conv)$, by \ref{exa:internal_rig_bifold_alg}.  Since the inclusions $\LCHaus \hookrightarrow \Top \hookrightarrow \Conv$ preserve finite products (\ref{exa:lcgroups_circle_group}, \ref{para:conv_cah}), $\LCAb = \Ab(\LCHaus)$ underlies a full subcategory of $\Ab(\Conv)$.  By \ref{exa:internal_rig_bifold_alg}, $\Ab(\Conv)$ may be identified with the category of internal left $\ZZ$-modules for the discrete internal ring $\ZZ = \ZZ \cdot 1$ in $\Conv$, and $\Ab(\Conv)$ underlies a $\V$-category $\Mod{\ZZ}(\V) \cong \Alg{\Mat(\ZZ)}^\nml(\V)$ for a theory $\Mat(\ZZ)$ that is enriched in $\V$ (but is discrete in the sense of \ref{para:conv_cah}).

Given any convergence abelian group $A$, we write $\Hom(A,\TT)$ for the internal hom from $A$ to the circle group $\TT$ in $\Ab(\Conv)$.  This generalizes our notation in \ref{exa:lcgroups_circle_group}, because if $A$ is a locally compact group, then $\Hom(A,\TT)$ carries (the convergence structure determined by) the compact-open topology, by \cite{Binz}.  Writing $\Mat(\ZZ)^\perp_\TT$ for the $\V$-enriched commutant of the discrete $\V$-enriched theory $\Mat(\ZZ)$ with respect to $\TT$, we do not know \textit{a priori} that $\Mat(\ZZ)^\perp_\TT$ is discrete, because in general $\Conv$-enriched commutants of discrete theories need not be discrete.  But since \eqref{eq:zn_double-dual_iso} is an isomorphism in $\LCAb$ and hence in $\Ab(\Conv)$, we find that \eqref{eq:tt_lcab} is an isomorphism in $\Conv$, so in this $\Conv$-enriched setting $[\TT|\TT]:\Mat(\ZZ) \rightarrow \Mat(\ZZ)^\perp_\TT$ is an isomorphism of theories enriched in $\Conv$, and in particular $\Mat(\ZZ)^\perp_\TT$ is discrete since $\Mat(\ZZ)$ is so.  Thus we have shown that $(\Mat(\ZZ),\TT)$ is again a balanced algebra and that $\TT$ carries the structure of a balanced bifold algebra $(\Mat(\ZZ),\Mat(\ZZ),\TT)$, where all these concepts are now interpreted in the $\Conv$-enriched sense.
\end{exasub}

\subsection{Examples with arbitrary arities over \texorpdfstring{$\Set$}{Set}}

In the following examples, we let $\V = \Set$ and we take $\J = \Set$ to be the system of arities consisting of \textit{all} (small) sets, so that in this case $\J$-theories are precisely Lawvere-Linton theories, which are equivalently described as arbitrary monads on $\Set$ \pref{para:exa_sys_ar}.  Given a monad $\TT$ on $\Set$, the opposite of the Kleisli category of $\TT$ is a Lawvere-Linton theory $\T$ for which $\Alg{\T}^!(\Set) \cong \Set^\TT$.

\begin{exasub}[\textbf{Complete sup-lattices}]
Let $\Sup$ denote the category of complete lattices with maps that preserve joins of arbitrary families of elements.  We call the objects of this category \textbf{sup-lattices}.  It is well known that $\Sup$ is strictly monadic over $\Set$ \cite[20.5]{AHS} and that the free sup-lattice on a set $X$ is the powerset $\mathcal{P}(X)$, which we identify with the power $2^X$ in $\Sup$, where $2$ denotes the two-element chain.  The resulting monad $\mathcal{P}$ on $\Set$ is called the powerset monad, and its Kleisli category is isomorphic to the category $\Rel$ whose objects are sets and whose morphisms are relations \cite[20B]{AHS}.  But $\Rel \cong \Rel^\op$, so $\Rel$ is a Lawvere-Linton theory, and $\Sup$ is isomorphic to the category $\Alg{\Rel}^!(\Set)$ of normal $\Rel$-algebras in $\Set$.  In particular, $2$ carries the structure of a normal $\Rel$-algebra that we write as $2:\Rel \rightarrow \Set$.  For each set $X$, the map 
\begin{equation}\label{eq:str_map_2_sup}2_{X1}:\Rel(X,1) = 2^X \longrightarrow \Set(2^X,2)\end{equation}
sends each map $A:X \rightarrow 2$ to the operation $\bigvee_A:2^X \rightarrow 2$ given by $\bigvee_A(B) = \bigvee_{x \in X} A(x) \wedge B(x)$.  
But $2^X$ is the free sup-lattice on $X$, and $\bigvee_A$ is also the sup-lattice homomorphism induced by $A$, so the map \eqref{eq:str_map_2_sup} factors through the inclusion
$\Sup(2^X,X) \hookrightarrow \Set(2^X,2)$ by way of a bijection $2^X = \Set(X,2) \xrightarrow{\sim} \Sup(2^X,2)$.  The commutant $\Rel^\perp_2$ of the Lawvere-Linton theory $\Rel$ has $\Rel^\perp_2(X,1) = \Sup(2^X,2)$ for each set $X$, so $(\Rel,2)$ is a commutative algebra, by \eqref{eq:diagr_b_pipe_a}, and the resulting morphism of Lawvere-Linton theories $[2|2]:\Rel \rightarrow \Rel^\perp_2$ is an isomorphism, by \ref{para:mor_jth}.  Hence $(\Rel,2)$ is a balanced algebra, by \ref{rem:bal_alg} (so in particular $(\Rel,2)$ is saturated and $\Rel$ is commutative).  Therefore, $2$ carries the structure of a balanced bifold algebra $(\Rel,\Rel,2)$, by \ref{thm:charn_of_balalgs_via_bifoldalgs}.
\end{exasub}

The following example demonstrates that the question of whether a given commutative algebra is balanced can in general depend on set-theoretic assumptions on the existence of large cardinals.  For an introduction to measurable cardinals, the reader is referred to \cite{Jech}.

\begin{exasub}[\textbf{Abelian groups and a theorem of Ehrenfeucht and \L o\'s}]\label{exa:ab_ehr_los}
The category of (small) abelian groups, $\Ab$, is strictly monadic over the category of small sets, $\Set$, and the induced monad is commutative, so we may identify $\Ab$ with the category $\Alg{\T_{\Ab}}^\nml(\Set)$ of normal $\T_{\Ab}$-algebras in $\Set$ for a commutative Lawvere-Linton theory $\T_{\Ab}$.  For each small set $J$, $\T_{\Ab}(J,1)$ is the (set underlying the) free abelian group on $J$, that is, the copower $J \cdot \ZZ = \oplus_{j \in J}\ZZ$ in $\Ab$.  Writing $A^* = \Hom_\ZZ(A,\ZZ)$ for the ($\ZZ$-linear) dual of each abelian group $A$, note that the dual of $J \cdot \ZZ$ is $(J \cdot \ZZ)^* \cong \ZZ^J$, so the double dual of $J \cdot \ZZ$ is $(J\cdot\ZZ)^{**}  \cong \Hom_\ZZ(\ZZ^J,\ZZ)$.  Since $\T_{\Ab}$ is commutative, the algebra $(\T_{\Ab},\ZZ)$ is commutative, so we may consider the canonical morphism of theories $[\ZZ|\ZZ]:\T_{\Ab} \rightarrow (\T_{\Ab})^\perp_\ZZ$ valued in the commutant of $\T_{\Ab}$ with respect to $\ZZ$ \pref{eq:b_pipe_a}.  For each small set $J$, the map
$$[\ZZ|\ZZ]_{J1}:\T_{\Ab}(J,1) = J \cdot \ZZ \rightarrow (\T_{\Ab})^\perp_\ZZ(J,1) = \Hom_\ZZ(\ZZ^J,\ZZ)$$
may be identified with the canonical homomorphism
$$J \cdot \ZZ \overset{\varphi}{\longrightarrow} (J \cdot \ZZ)^{**} \cong \Hom_\ZZ(\ZZ^J,\ZZ)$$
from $J \cdot \ZZ$ into its double-dual.  By a theorem of Ehrenfeucht and \L o\'s \cite{EhrLos} (also see \cite{Eda,Shela}), if the cardinality of $J$ is less than the first measurable cardinal, or if there are no measurable cardinals, then $J \cdot \ZZ$ is a reflexive abelian group, meaning that the canonical homomorphism $\varphi$ is an isomorphism, while if the cardinality of $J$ is greater than or equal to the first measurable cardinal, then $J \cdot \ZZ$ is not reflexive.  Therefore, by \ref{rem:bal_alg} and \ref{para:mor_jth} we deduce that
$$
\textit{$(\T_{\Ab},\ZZ)$ is a balanced algebra if and only if there are no measurable cardinals in $\Set$.}
$$
Since $(\T_{\Ab},\ZZ)$ is a commutative algebra, $\ZZ$ carries the structure of a bifold algebra $(\T_{\Ab},\T_{\Ab},\ZZ)$, and by \ref{thm:charn_of_balalgs_via_bifoldalgs} we deduce the following:
$$
\begin{minipage}{5.2in}
\textit{$(\T_{\Ab},\T_{\Ab},\ZZ)$ is a commutant bifold algebra if and only if there are no measurable cardinals in $\Set$, in which case $(\T_{\Ab},\T_{\Ab},\ZZ)$ is a balanced bifold algebra.}
\end{minipage}
$$
\end{exasub}

\begin{exasub}[\textbf{Compact groups and discrete groups}]
Let us write $\Ab$ for the category of abelian groups, which we identify with discrete topological abelian groups, and $\textnormal{CAb}$ for the category of compact Hausdorff topological abelian groups, which we call simply \textit{compact groups}.  By \cite[\S 5]{Lin:Eq}, there are Lawvere-Linton theories $\T$ and $\U$ whose categories of normal $\T$- and $\U$-algebras may be identified with $\Ab$ and $\textnormal{CAb}$, respectively.  For each set $J$, let us write $F^\T J$ and $F^\U J$ for the free normal $\T$- and $\U$-algebras on $J$.  Then $\T(J,1)$ is the underlying set of the free abelian group $F^\T J = \oplus_{j \in J} \ZZ$ on $J$, while $\U(J,1)$ is the underlying set of the free compact group $F^\U J$ on $J$.  Writing $\TT$ to denote circle group, considered as an object of $\textnormal{CAb}$, we shall also write $\TT_d$ for the discrete abelian group underlying $\TT$.  Hence the circle group underlies an algebra pair $(\T,\TT_d)$, $(\U,\TT)$, and we may consider the commutants  $\T^\perp_{\TT_d}$ and $\U^\perp_{\TT}$.  For each set $J$, we write $\TT^J$ for the $J$-th power of $\TT$ in $\textnormal{CAb}$ and $(\TT^J)_d$ for its underlying discrete group, which is the $J$-th power of $\TT_d$ in $\Ab$, so that $\T^\perp_{\TT_d}(J,1) = \Ab\bigl((\TT^J)_d,\TT_d\bigr)$ and $\U^\perp_{\TT}(J,1) = \textnormal{CAb}(\TT^J,\TT)$.

We now employ the material in Example \ref{exa:lcgroups_circle_group} on Pontryagin duality, as well as the fact that the Pontryagin duality restricts also to an equivalence $\textnormal{CAb}^\op \simeq \Ab$ \cite[Thm. 12]{Mor}.  For each set $J$, the discrete group $F^\T J = \oplus_{j \in J}\ZZ$ serves also as a copower of $\ZZ$ by $J$ in $\LCAb$, so by \ref{exa:lcgroups_circle_group} its Pontryagin dual is $\Hom(F^\T J,\TT) \cong \Hom(\ZZ,\TT)^J \cong \TT^J$, the $J$-th power of $\TT$ in $\textnormal{CAb}$.  By Pontryagin duality, we obtain an isomorphism
$$\oplus_{j \in J} \ZZ = F^\T J \rightarrow \Hom(\Hom(F^\T J,\TT),\TT) \cong \Hom(\TT^J,\TT)$$
in $\LCAb$ whose underlying bijective map provides a factorization
$$
\xymatrix{
\T(J,I) \ar[drr]_{(\TT_d)_{JI}} \ar@{-->}[rr]^(.4){[\TT_d|\TT]_{JI}} & & \textnormal{CAb}(\TT^J,\TT) = \U^\perp_\TT(J,I) \ar@{^(->}[d]\\
                          & & \Set(\TT^J,\TT)
}
$$
that witnesses not only that $(\U,\TT)$ commutes with $(\T,\TT_d)$, by \eqref{eq:diagr_b_pipe_a}, but also that the induced morphism of Lawvere-Linton theories $[\TT_d|\TT]:\T \rightarrow \U^\perp_\TT$ is an isomorphism (by \ref{para:mor_jth}).  Hence, by \ref{rem:com_balg_com_pair} $(\T,\TT_d)$, $(\U,\TT)$ is a left-commutant algebra pair, so the circle group underlies a left-commutant bifold algebra $(\T,\U,\TT)$.  For each set $J$, the Pontryagin dual $\Hom(F^\U J,\TT)$ of the compact group $F^\U J$ is discrete and carries the pointwise abelian group structure, and its underlying set $\textnormal{CAb}(F^\U J,\TT)$ is canonically isomorphic to $\Set(J,\TT)$, so $\Hom(F^\U J,\TT) \cong (\TT^J)_d$ in $\LCAb$.  We have a canonical morphism of Lawvere-Linton theories $[\TT|\TT_d]:\U \rightarrow \T^\perp_{\TT_d}$, and for each set $J$, its component $[\TT|\TT_d]_{J1}:\U(J,1) \rightarrow \Ab\bigl((\TT^J)_d,\TT_d\bigr)$ underlies the canonical isomorphism
$$F^\U J \rightarrow \Hom(\Hom(F^\U J,\TT),\TT) \cong \Hom\bigl((\TT^J)_d,\TT\bigr)$$
that exhibits $F^\U J$ as isomorphic to its Pontryagin double-dual.  Hence $[\TT|\TT_d]$ is an isomorphism, by \ref{para:mor_jth}, so $(\T,\TT_d)$, $(\U,\TT)$ is a commutant algebra pair.  This proves the following:
$$
\begin{minipage}{4.5in}\textit{The circle group $\TT$ underlies a commutant bifold algebra $(\T,\U,\TT)$, where $\T$ and $\U$ are the Lawvere-Linton theories of abelian groups and of compact Hausdorff abelian groups, respectively.}
\end{minipage}
$$
\end{exasub}

\bibliographystyle{amsplain}
\bibliography{bib}

\end{document}

%% file: pre.tex
\usepackage{amsmath, amsthm, amssymb, stmaryrd}
\usepackage[all]{xy}
\usepackage{fullpage}
\usepackage{titlesec}
\usepackage{mathrsfs}
\usepackage{amsbsy}
\usepackage{turnstile}
\usepackage{bbm}
\usepackage{yhmath}
\usepackage{tensor}
\usepackage{graphics}
\usepackage{enumitem}
\usepackage{calligra}
\usepackage{verbatim}
\usepackage{setspace}
\usepackage{hhline}
\usepackage{array}

\xyoption{rotate}

\usepackage[pdftex,bookmarks=true]{hyperref}
\usepackage[svgnames]{xcolor}
\hypersetup{
    linkbordercolor = {PaleTurquoise},
    citebordercolor = {PaleGreen},
}

\usepackage{titletoc}

\titlecontents{section}
[0pt]                                               % left margin
{}%
{\contentsmargin{0pt}                               % numbered entry format
    \thecontentslabel\enspace%
    }
{\contentsmargin{0pt}}                        % unnumbered entry format
{\titlerule*[.5pc]{.}\contentspage}                 % filler-page format (e.g dots)
[]                                                  % below code (e.g vertical space)

\makeatletter
\newcommand*{\@old@slash}{}\let\@old@slash\slash
\def\slash{\relax\ifmmode\delimiter"502F30E\mathopen{}\else\@old@slash\fi}
\makeatother

\titleformat{\section}{\normalsize\bfseries}{\thesection}{1em}{}
\titleformat{\subsection}{\normalsize\bfseries}{\thesubsection}{1em}{}

\numberwithin{equation}{subsection}

\theoremstyle{plain}

\newtheorem{prop}[subsection]{Proposition}
\newtheorem{lem}[subsection]{Lemma}
\newtheorem{cor}[subsection]{Corollary}
\newtheorem{thm}[subsection]{Theorem}

\theoremstyle{definition}

\newtheorem{defn}[subsection]{Definition}
\newtheorem{exa}[subsection]{Example}
\newtheorem{rem}[subsection]{Remark}
\newtheorem{para}[subsection]{}

\newtheorem{notn}[subsection]{Notation}

\newtheorem{exasub}[subsubsection]{Example}
\newtheorem{parasub}[subsubsection]{}

\newcommand*{\emptybox}{\leavevmode\hbox{}}

\DeclareMathOperator{\Hom}{Hom}

\newdir{ >}{{}*!/-10pt/@{>}}

\DeclareMathAlphabet{\mathpzc}{OT1}{pzc}{m}{it}
\DeclareMathAlphabet{\mathcalligra}{T1}{calligra}{m}{n}

\newcommand{\pref}[1]{\textnormal{(\ref{#1})}}

\newcommand{\A}{\ensuremath{\mathscr{A}}}
\newcommand{\B}{\ensuremath{\mathscr{B}}}
\newcommand{\C}{\ensuremath{\mathscr{C}}}
\newcommand{\D}{\ensuremath{\mathscr{D}}}
\newcommand{\E}{\ensuremath{\mathscr{E}}}
\newcommand{\F}{\ensuremath{\mathscr{F}}}
\newcommand{\G}{\ensuremath{\mathscr{G}}}

\newcommand{\normalJ}{\ensuremath{\mathscr{J}}}
\newcommand{\J}{\ensuremath{{\kern -0.4ex \mathscr{J}}}}

\newcommand{\K}{\ensuremath{\mathscr{K}}}

\newcommand{\sS}{\ensuremath{\mathscr{S}}}
\newcommand{\T}{\ensuremath{\mathscr{T}}}
\newcommand{\U}{\ensuremath{\mathscr{U}}}
\newcommand{\V}{\ensuremath{\mathscr{V}}}

\newcommand{\X}{\ensuremath{\mathscr{X}}}
\newcommand{\Y}{\ensuremath{\mathscr{Y}}}

\newcommand{\CAT}{\ensuremath{\operatorname{\textnormal{\text{CAT}}}}}
\newcommand{\eCAT}[1]{\ensuremath{#1\textnormal{-\text{CAT}}}}
\newcommand{\VCAT}{\ensuremath{\V\textnormal{-\text{CAT}}}}

\newcommand{\ALGCAT}{\ensuremath{\operatorname{\textnormal{\text{ALGCAT}}}}}
\newcommand{\ALGCATJ}{\ensuremath{\ALGCAT_{\kern -0.6ex \normalJ}}}
\newcommand{\ADA}{\ensuremath{\operatorname{\textnormal{\text{ADA}}}}}
\newcommand{\ADAJ}{\ensuremath{\ADA_{\kern -0.6ex \normalJ}}}

\newcommand{\NN}{\ensuremath{\mathbb{N}}}

\newcommand{\RR}{\ensuremath{\mathbb{R}}}

\newcommand{\TT}{\ensuremath{\mathbb{T}}}

\newcommand{\ZZ}{\ensuremath{\mathbb{Z}}}

\newcommand{\ob}{\ensuremath{\operatorname{\textnormal{\textsf{ob}}}}}

\newcommand{\End}{\ensuremath{\operatorname{\textnormal{End}}}}

\newcommand{\ca}[1]{\scalebox{0.85}{\raisebox{0.3ex}{$|$}} #1\scalebox{0.85}{\raisebox{0.3ex}{$|$}}}

\newcommand{\Th}{\ensuremath{\textnormal{Th}}}
\newcommand{\ThJ}{\ensuremath{\Th_{\kern -0.5ex \normalJ}}}

\newcommand{\otimesJ}[2]{\ensuremath{#1 \kern-0.15ex \otimes_{\kern -1ex \normalJ} \kern-0.5ex #2}}
\newcommand{\totimes}{\mathbin{\tilde{\otimes}}}
\newcommand{\totimesJ}[2]{\ensuremath{#1 \kern-0.15ex \totimes_{\kern -1ex \normalJ} \kern-0.5ex #2}}

\newcommand{\Vfp}{\ensuremath{\V_{\kern -0.5ex fp}}}
\newcommand{\TTperpJ}{\ensuremath{\TT^\perp_{\kern-.1ex\scriptscriptstyle\J}}}
\newcommand{\TTperpJwrt}[1]{\ensuremath{\TT^\perp_{\kern-.1ex\scriptscriptstyle\J#1}}}

\newcommand{\Ev}{\ensuremath{\textnormal{\textsf{Ev}}}}

\newcommand{\VCATJ}{\VCAT_{\kern -0.7ex \normalJ}}
\newcommand{\PhiJ}{\Phi_{\kern -0.8ex \normalJ}}
\newcommand{\MndJ}{\Mnd_{\kern -0.8ex \normalJ}}

\newcommand{\Set}{\ensuremath{\operatorname{\textnormal{\text{Set}}}}}
\newcommand{\FinCard}{\ensuremath{\operatorname{\textnormal{\text{FinCard}}}}}
\newcommand{\DF}{\ensuremath{\operatorname{\textnormal{\text{DF}}}}}

\newcommand{\SLat}{\ensuremath{\operatorname{\textnormal{\text{SLat}}}}}

\newcommand{\Sup}{\ensuremath{\operatorname{\textnormal{\text{Sup}}}}}
\newcommand{\Rel}{\ensuremath{\operatorname{\textnormal{\text{Rel}}}}}

\newcommand{\Ab}{\ensuremath{\operatorname{\textnormal{\text{Ab}}}}}

\newcommand{\Mod}[1]{\ensuremath{#1\textnormal{-\text{Mod}}}}

\newcommand{\Aff}[1]{\ensuremath{#1\textnormal{-\text{Aff}}}}
\newcommand{\Cvx}[1]{\ensuremath{#1\textnormal{-\text{Cvx}}}}

\newcommand{\Mat}{\ensuremath{\textnormal{\text{Mat}}}}

\newcommand{\Top}{\ensuremath{\operatorname{\textnormal{\text{Top}}}}}

\newcommand{\LCHaus}{\ensuremath{\operatorname{\textnormal{\text{LCHaus}}}}}
\newcommand{\LCAb}{\ensuremath{\operatorname{\textnormal{\text{LCAb}}}}}
\newcommand{\Conv}{\ensuremath{\operatorname{\textnormal{\text{Conv}}}}}

\newcommand{\Mnd}{\ensuremath{\operatorname{\textnormal{\text{Mnd}}}}}

\newcommand{\CMon}{\ensuremath{\operatorname{\textnormal{\text{CMon}}}}}

\newcommand{\Alg}[1]{\ensuremath{#1\kern -.5ex\operatorname{\textnormal{-Alg}}}}
\newcommand{\CAlg}[1]{\ensuremath{#1\kern -.5ex\operatorname{\textnormal{-\text{CAlg}}}}}
\newcommand{\CAlgPair}[2]{\ensuremath{(#1,#2)\kern -.5ex\operatorname{\textnormal{-CPair}}}}
\newcommand{\AlgPair}[2]{\ensuremath{(#1,#2)\kern -.5ex\operatorname{\textnormal{-Pair}}}}
\newcommand{\CinftyRing}{\ensuremath{C^\infty\kern -.5ex\operatorname{\textnormal{-\text{Ring}}}}}

\newcommand{\CRProf}{\ensuremath{\operatorname{\textnormal{\text{CRProf}}}}}
\newcommand{\CRProfJ}{\ensuremath{\CRProf_{\kern -0.6ex \normalJ}}}
\newcommand{\Algs}{\ensuremath{\operatorname{\textnormal{\text{Alg}}}}}
\newcommand{\AlgsJ}{\Algs_{\kern -0.6ex \normalJ}}

\newcommand{\BifoldAlg}{\ensuremath{\operatorname{\textnormal{\text{BAlg}}}}}
\newcommand{\BifoldAlgJ}{\BifoldAlg_{\kern -0.6ex \normalJ}}
\newcommand{\BAlg}{\BifoldAlg}
\newcommand{\RComBAlg}{\ensuremath{\operatorname{\textnormal{\text{RComBAlg}}}}}
\newcommand{\LComBAlg}{\ensuremath{\operatorname{\textnormal{\text{LComBAlg}}}}}
\newcommand{\ComBAlg}{\ensuremath{\operatorname{\textnormal{\text{ComBAlg}}}}}
\newcommand{\BalBAlg}{\ensuremath{\operatorname{\textnormal{\text{BalBAlg}}}}}
\newcommand{\SatAlgs}{\ensuremath{\operatorname{\textnormal{\text{SatAlg}}}}}
\newcommand{\CAlgs}{\ensuremath{\operatorname{\textnormal{\text{CAlg}}}}}
\newcommand{\CoCAlgs}{\ensuremath{\operatorname{\textnormal{\text{CoCAlg}}}}}
\newcommand{\SatCAlgs}{\ensuremath{\operatorname{\textnormal{\text{SatCAlg}}}}}
\newcommand{\SatCoCAlgs}{\ensuremath{\operatorname{\textnormal{\text{SatCoCAlg}}}}}
\newcommand{\BalAlgs}{\ensuremath{\operatorname{\textnormal{\text{BalAlg}}}}}
\newcommand{\CPair}{\ensuremath{\operatorname{\textnormal{\text{CPair}}}}}
\newcommand{\Pair}{\ensuremath{\operatorname{\textnormal{\text{Pair}}}}}
\newcommand{\LComPair}{\ensuremath{\operatorname{\textnormal{\text{LComPair}}}}}
\newcommand{\RComPair}{\ensuremath{\operatorname{\textnormal{\text{RComPair}}}}}
\newcommand{\ComPair}{\ensuremath{\operatorname{\textnormal{\text{ComPair}}}}}

\newcommand{\RCom}{\ensuremath{\operatorname{\mathsf{RCom}}}}
\newcommand{\LCom}{\ensuremath{\operatorname{\mathsf{LCom}}}}

\newcommand{\Com}{\ensuremath{\operatorname{\mathsf{Com}}}}

\newcommand{\Adjn}[6]{\xymatrix {#1 \ar@/_0.5pc/[rr]_{#2}^(0.4){#4}^(0.6){#5}^{\top} & & #6 \ar@/_0.5pc/[ll]_{#3}}}
\newcommand{\Equiv}[6]{\xymatrix {#1 \ar@/_0.5pc/[rr]_{#2}^(0.4){#4}^(0.6){#5}^{\sim} & & #6 \ar@/_0.5pc/[ll]_{#3}}}
\newcommand{\EquivAlt}[6]{\xymatrix {#1 \ar@{<-}@/_0.5pc/[rr]_{#3}^(0.4){#4}^(0.6){#5}^{\sim} & & #6 \ar@{<-}@/_0.5pc/[ll]_{#2}}}
\newcommand{\AlgDual}[4]{\xymatrix {#1:#2 \ar@/_0.2pc/[r] & #3:#4 \ar@/_0.2pc/[l] }}

\newcommand{\op}{\ensuremath{\textnormal{\tiny op}}}

\newcommand{\aff}{\ensuremath{\textnormal{\tiny aff}}}

\newcommand{\pushoutcorner}{\ar@{}[dr]|(.3)\ulcorner}
\newcommand{\pullbackcorner}{\ar@{}[dr]|(.3)\lrcorner}

\newcommand{\blanktwo}{\mathop{=}}

\newcommand{\s}{\textnormal{s}}
\newcommand{\sx}{{\textnormal{sc}}}
\newcommand{\si}{{\scriptscriptstyle\cong}}

\newcommand{\nml}{!}

\newcommand{\cmt}[1]{}